\documentclass[10pt,leqno]{amsart}
\usepackage{graphicx}
\usepackage{indentfirst,csquotes,amsthm, amscd, amsfonts, amssymb}

\usepackage{mathtools}
\usepackage{enumitem}
\usepackage{tikz-cd}
\usepackage{amsfonts}
\usepackage[matrix,arrow,curve,cmtip]{xy}
\usepackage{amssymb}
\usepackage{dsfont}

\input diagxy

\newcommand{\donotbreakdash}[1]{#1\nobreakdash-\hspace{0pt}}
\newcommand{\Q}{\mathcal{Q}}
\newcommand{\cat}[1]{\ensuremath{\mathtt{#1}}} 
\newcommand{\objectQuantaloida}{p}
\newcommand{\objectQuantaloidb}{q}
\newcommand{\objectQuantaloidc}{r}
\newcommand{\arrowQuantaloidf}{u}
\newcommand{\newmorphism}[4]{\ensuremath{{#2}\xymatrix@1@C=#1pt{\ar@{->}[r]|{\textstyle{#3}}&}{#4}}}
\newcommand{\arrowQuantaloidg}{v}
\newcommand{\tbigvee}{\mathop{\textstyle \bigvee}}
\newcommand{\set}[1]{\{\,#1\,\}}
\newcommand{\arrowQuantaloidh}{h}
\newcommand{\arrowQuantaloidk}{k}
\newcommand{\objectQuantaloidd}{s}
\newcommand{\arrowQuantaloidl}{l}
\newcommand{\alg}[1]{\ensuremath{\mathfrak{#1}}} 
\newcommand{\elementQa}{a}
\newcommand{\elementQb}{b}
\newcommand{\elementQlambda}{\lambda}
\newcommand{\elementQmu}{\mu}%
\newcommand{\elementQc}{c}%
\newcounter{num}[section]

\newcommand{\newal}{\widetilde{a_{\ell}}}
\newcommand{\newar}{\widetilde{a_{r}}}
\newcommand{\circlearrow}{\xymatrix@1@C=5mm{\ar@{->}[r]|{\circ}&}}
\newcommand{\Qpresheavef}{f}
\newcommand{\D}{\mathcal{D}}
\newcommand{\Qpresheaveg}{g}
\newcommand{\R}{\mathcal{R}}
\newcommand{\tbigwedge}{\mathop{\textstyle \bigwedge}}
\newcommand{\QpresheaveG}{G}
\newcommand{\Qpresheaveh}{h}
\newcommand{\newdownarrow}{{{\rlap{$\ $}\hbox{$\downarrow$}}}}%
\newcommand{\objectQuantaloide}{e}
\newcommand{\bigset}[1]{\bigl\{\,#1\,\bigr\}}
\newcommand{\elementQy}{y}
\newcommand{\elementQx}{x}

\renewcommand{\hom}{\text{\rm hom}}
\newcommand{\id}{\text{\rm id}}
\newcommand{\dom}{\text{\rm dom}}
\newcommand{\codom}{\text{\rm codom}}
\newcommand{\morph}{\text{\rm morph}}
\newcommand{\true}{\text{\rm true}}

\theoremstyle{plain}
          \newtheorem{theorem}{Theorem}[section]
          \newtheorem{lemma}[theorem]{Lemma}
          \newtheorem{proposition}[theorem]{Proposition}
	\newtheorem{corollary}[theorem]{Corollary}

	\newtheorem*{fact*}{Fact}

        \theoremstyle{definition}
          \newtheorem{example}[theorem]{Example}
                   
          \newtheorem*{addition*}{Addition}
\newtheorem*{standingassumption*}{Standing Assumption}

        \theoremstyle{remark}
\newtheorem*{remark*}{Remark}          
\newtheorem{histremark}[theorem]{Historical Remark}
\newtheorem{remark}[theorem]{Remark}
\newtheorem{remarks}[theorem]{Remarks}
\newtheorem{comment}[theorem]{Comment}
\newtheorem{comments}[theorem]{Comments}
\newtheorem{comment*}[theorem]{Comment}

\begin{document}
\title[Quantaloid-enriched categories:]{Quantaloid-enriched categories:\\ Factorization, weak classifiers, and symmetry} 
\author[J. Guti\'errez Garc{\'\i}a]{Javier Guti\'errez Garc{\'\i}a}
\date{\today}
\address{Departamento de Matem\'aticas, Universidad del Pa\'{\i}s Vasco (UPV/EHU), 48080, Bilbao, SPAIN}
\email{javier.gutierrezgarcia@ehu.eus}
\author[U. H\"ohle]{Ulrich H\"ohle}
\date{\today}
\address{Fakult\"{a}t f\"{u}r Mathematik und Naturwissenchaften, Bergische Universit\"{a}t, D-42097, Wuppertal, GERMANY}
\email{uhoehle@uni-wuppertal.de}
\maketitle

\let\thefootnote\relax

\begin{abstract}
This paper provides a comprehensive overview of some of the foundational properties of categories enriched over quantaloids, along with several new results. 
We demonstrate that the category whose objects are quantaloid-enriched categories and whose morphisms are left adjoint distributors admits an (epi, extremal mono)--factorization system. Furthermore, we prove that the category of cocomplete quantaloid-enriched categories satisfies the weak subobject classifier axiom, under stability conditions on the underlying quantaloid.  
As an application, we discuss how these structural results extend to quantale-valued sets, thereby generalizing the classical theory of \donotbreakdash{$\Omega$}valued sets.
\end{abstract} 
\bigskip

\section*{Introduction}

Quantaloid-enriched categories generalize sheaves on locales (cf.\ \cite{Walters81}) and can be understood as temporally dynamic extensions of quantale-enriched categories. 
This survey is intended to unify and extend existing results, while offering a coherent framework for further exploration of categorical structures enriched in quantaloids. 

Let $\Q$ be a small quantaloid. We begin by providing a detailed treatment of \donotbreakdash{$\Q$}enriched presheaves and their role in establishing cocompleteness and constructing the Cauchy completion of \donotbreakdash{$\Q$}enriched categories. 

Beyond these foundational aspects, we introduce new structural insights into \donotbreakdash{$\Q$}enriched categories, including:
\begin{enumerate}[label=\textup{--},leftmargin=12pt,topsep=3pt,itemsep=0pt
]
\item When the underlying quantaloid $\Q$ is stable, the category of separated and cocomplete \donotbreakdash{$\Q$}enriched categories satisfies the weak subobject classifier axiom (cf.\ \cite{ho25}).
\item The category \donotbreakdash{$\Q$}$\cat{Set}$, consisting of \donotbreakdash{$\Q$}enriched categories and left adjoint distributors, forms an \donotbreakdash{(epi, extremal mono)}category. Our proof relies on a characterization of epimorphisms and extremal monomorphisms in \donotbreakdash{$\Q$}$\cat{Set}$, which has only appeared in fragmented form in the literature (cf.\ \cite{PuDexue}).
\item We provide a complete characterization of the conditions under which the Cauchy completion preserves the symmetry axiom of \donotbreakdash{$\Q$}enriched categories, and we explore several notable examples arising from involutive quantales.
\end{enumerate}
As a concluding perspective, we examine how these constructions manifest 
 in the context of \donotbreakdash{quantale}valued sets,
 which can be interpreted as symmetric, \donotbreakdash{quantale}valued preordered sets. 
In this setting, the underlying involutive quantaloids are derived from involutive and unital quantales via diagonal arrows --- a construction originally introduced by I. Stubbe (cf.\ \cite{Stubbe14}).

The paper is organized as follows. We begin in Section~\ref{section1} by exploring general properties of small quantaloids and their relationship to quantales, including  those that are not necessarily unital. This sets the stage for the development of the theory of \donotbreakdash{quantaloid}enriched categories. In Section~\ref{section2}, we introduce the notion of \donotbreakdash{$\Q$}enriched categories and distributors, and establish the foundational properties of adjoint distributors and functors. Section~\ref{section3} is devoted to \donotbreakdash{$\Q$}enriched presheaves, including covariant and contravariant variants, and the construction of the presingleton space, which plays a central role in later sections. In Section~\ref{section4}, we explore cocompleteness in the enriched setting, characterize cocontinuous functors, and prove that the category of separated and cocomplete \donotbreakdash{$\Q$}enriched categories admits weak subobject classifiers of specific types. In Section~\ref{section5}, we study the category \donotbreakdash{$\Q$}Set of \donotbreakdash{$\Q$}enriched categories and left adjoint distributors, and show that it forms an  (epi,extremal mono)-category. This includes a detailed analysis of epimorphisms and extremal monomorphisms, culminating in a canonical factorization result. In Section~\ref{section6} we recall the Cauchy completion and give a full characterization of the property that the Cauchy completion preserves the symmetry axiom. Finally Section~\ref{section7} serves as an epilogue, illustrating the previous results in the context of quantale-valued sets.

\section{Quantaloids and their relationship to quantales}\label{section1}
Let $\cat{Sup}$ be the category of complete lattices and join-preserving maps, equipped with the tensor product of complete lattices (cf.\ \cite[Sect.~2.1.2]{EGHK}). It is well known that $\cat{Sup}$ is a symmetric and monoidal closed category.

A \emph{small quantaloid} is a small category $\Q$ with the following additional properties:
\begin{enumerate}[label=$\bullet$,leftmargin=12pt,topsep=3pt,itemsep=0pt
]
\item Each hom-set is a complete lattice.
 \item The composition of morphisms $\circ\colon\hom ({\objectQuantaloidb},{\objectQuantaloidc})\times \hom ({\objectQuantaloida},{\objectQuantaloidb})\rightarrow \hom ({\objectQuantaloida},{\objectQuantaloidc})$ preserves arbitrary joins in both variables.
\end{enumerate}
As a result, quantaloids can be viewed as categories enriched over $\cat{Sup}$. Hence for each pair of objects ${\objectQuantaloida}$ and ${\objectQuantaloidb}$ of a quantaloid the hom-set $\hom ({\objectQuantaloida},{\objectQuantaloidb})$ is a hom-space given by some complete lattice.
Regarding notation, we recall that domain and codomain are intrinsic to the notion of a \emph{morphism} (cf.\ \cite[Sect.~1.2]{Borceux1}). Thus, for a morphism $\arrowQuantaloidf \in \hom ({\objectQuantaloida},{\objectQuantaloidb})$, we may also write
\newmorphism{16}{\objectQuantaloida}{\arrowQuantaloidf}{\objectQuantaloidb}.

Let $\Q_0$ be the set of objects in $\Q$. If two arrows share the same domain (e.g.\ $\newmorphism{16}{\objectQuantaloida}{\arrowQuantaloidf}{\objectQuantaloidb}
$ and $\newmorphism{16}{\objectQuantaloida}{\arrowQuantaloidg}{\objectQuantaloidc}
$), then
\begin{equation*}
{\arrowQuantaloidf}\swarrow {\arrowQuantaloidg}:=\tbigvee\set{{\arrowQuantaloidh}\in \hom ({\objectQuantaloidc},{\objectQuantaloidb})\mid {\arrowQuantaloidh}\circ {\arrowQuantaloidg}\le {\arrowQuantaloidf}}.
\end{equation*}
If two arrows share the same codomain (e.g.\ \newmorphism{16}{\objectQuantaloida}{\arrowQuantaloidf}{\objectQuantaloidb}
 and \newmorphism{16}{\objectQuantaloidc}{\arrowQuantaloidg}{\objectQuantaloidb}), we define:
\begin{equation*} {\arrowQuantaloidg}\searrow {\arrowQuantaloidf}:=\tbigvee\set{{\arrowQuantaloidh}\in\hom ({\objectQuantaloida},{\objectQuantaloidc})\mid {\arrowQuantaloidg}\circ {\arrowQuantaloidh}\le {\arrowQuantaloidf} }.
\end{equation*}
Furthermore, for each object ${\objectQuantaloida}\in\Q_0$, 
the hom-space $\hom ({\objectQuantaloida},{\objectQuantaloida})$ 
forms a unital quantale, with the unit element denoted by $1_{\objectQuantaloida}$. 
Conversely, any unital quantale can be regarded as the hom-space of a quantaloid with a single object. In this sense, quantaloids can be interpreted as ``varying unital quantales'', where each object has its own associated quantale structure, and morphisms between objects reflect interactions between these structures.

  A quantaloid $\Q$ is said to be \emph{involutive} if there exists a contravariant functor $j\colon {\Q \rightarrow \Q}$ enriched in $\cat{Sup}$, which acts as the identity on objects and satisfies the involution condition $j\circ j=\id_\Q$. Since $j$ is enriched in $\cat{Sup}$, its action on the hom-spaces is arbitrary join-preserving.
For concepts not defined here, we refer to \cite{EGHK,Rosenthal90,Rosenthal96}.

The following remark describes a canonical extension of a quantaloid $\Q$ based on its diagonal arrows of $\Q$  (cf.\ \cite{Stubbe14}).

\begin{remark} \label{newremark1.1} (1) Let $\Q$ be a small quantaloid. Then $\Q$ induces a new quantaloid $D\Q$ constructed as follows: 
\begin{enumerate}[label=\textup{--},leftmargin=12pt,topsep=3pt,itemsep=0pt
]
\item The objects of $D\Q$ are the morphisms of $\Q$.
\item The hom-spaces of $D\Q$ are constructed using square diagrams in $\Q$ of the form
\begin{center}$\bfig
\qtriangle(0,0)|mmm|/.>`>`>/<500,500>[{\objectQuantaloida}`{\objectQuantaloidc}`{\objectQuantaloidd};{\arrowQuantaloidg}\searrow{\arrowQuantaloidk}`\textstyle{\arrowQuantaloidk}`\textstyle{\arrowQuantaloidg}]
\btriangle(0,0)|mmm|/>`>`.>/<500,500>[{\objectQuantaloida}`{\objectQuantaloidb}`{\objectQuantaloidd};\textstyle{\arrowQuantaloidf}`{\textstyle\arrowQuantaloidk}`{\arrowQuantaloidk}\swarrow {\arrowQuantaloidf}]
 \efig$
\end{center}
Then $\hom ({\arrowQuantaloidf},{\arrowQuantaloidg})$ consists of all diagonal arrows making the diagram commutative, i.e.,~
\begin{equation*}
\hom ({\arrowQuantaloidf},{\arrowQuantaloidg})
=\set{{\arrowQuantaloidk}\in \hom (\dom({\arrowQuantaloidf}),\codom ({\arrowQuantaloidg}))\mid  ({\arrowQuantaloidk}\swarrow {\arrowQuantaloidf})\circ {\arrowQuantaloidf}={\arrowQuantaloidk}= {\arrowQuantaloidg}\circ ({\arrowQuantaloidg}\searrow {\arrowQuantaloidk})},\end{equation*}
which is a complete sublattice of $\hom (\dom ({\arrowQuantaloidf}),\codom ({\arrowQuantaloidg}))$ in the sense of $\cat{Sup}$.
\item The composition of morphisms ${\arrowQuantaloidk}\in \hom ({\arrowQuantaloidf},{\arrowQuantaloidg})$ and ${\arrowQuantaloidl}\in \hom ({\arrowQuantaloidg},{\arrowQuantaloidh})$ is defined by:
\begin{equation*}{\arrowQuantaloidl} \mathbin{\circ_{\arrowQuantaloidg}} {\arrowQuantaloidk}= ({\arrowQuantaloidl} \swarrow {\arrowQuantaloidg}) \circ {\arrowQuantaloidg} \circ ({\arrowQuantaloidg}\searrow {\arrowQuantaloidk})={\arrowQuantaloidl} \circ ({\arrowQuantaloidg}\searrow {\arrowQuantaloidk})=({\arrowQuantaloidl}\swarrow {\arrowQuantaloidg}) \circ {\arrowQuantaloidk}.\end{equation*}
This composition $\circ_{\arrowQuantaloidg}$ is clearly join-preserving in each variable separately --- i.e.,~
 \begin{equation*}
 \hom ({\arrowQuantaloidg},{\arrowQuantaloidh})\otimes  \hom ({\arrowQuantaloidf},{\arrowQuantaloidg}) \to^{ \circ_{\arrowQuantaloidg}} \hom ({\arrowQuantaloidf},{\arrowQuantaloidh})
 \end{equation*}
 is join-preserving, where $\otimes$ denotes the tensor product in $\cat{Sup}$.\footnote{Due to the universal property of the tensor product, any map $X\times Y\rightarrow Z$ that preserves arbitrary joins in each variable separately can be identified with a join preserving map $X\otimes Y \rightarrow Z$.}
\item Given a pair of objects $\arrowQuantaloidf$ and $\arrowQuantaloidg$ in the category $D\Q$, $\arrowQuantaloidf$ is a right unit and $ \arrowQuantaloidg$ is a left unit for the hom-space $\hom ({\arrowQuantaloidf},{\arrowQuantaloidg})$. In particular, $\hom ({\arrowQuantaloidf},{\arrowQuantaloidf})$ forms a unital quantale w.r.t.\ the composition operation $\circ_{\arrowQuantaloidf}$, and $ \arrowQuantaloidf$ itself is the unit element of $\hom ({\arrowQuantaloidf},{\arrowQuantaloidf})$.
\end{enumerate}
Moreover, there exists an embedding functor from $\Q$ 
 into $D\Q$, defined as follows. Objects ${\objectQuantaloida}\in \Q_0$ are identified with their respective identity morphisms
 $1_{\objectQuantaloida}$. Under this identification, the hom-spaces $\hom ({\objectQuantaloida},{\objectQuantaloidb})$ and $\hom  (1_{\objectQuantaloida},1_{\objectQuantaloidb})$ coincide, and this correspondence preserves the related compositions.
\\[2pt]
(2) When extending 
   a small involutive quantaloid $(\Q,j)$ to an involutive quantaloid via diagonal arrows, the previous construction 
must be modified.
 First, recall that a morphism $\arrowQuantaloidf$ of $\Q$ is \emph{hermitian} if $\dom (\arrowQuantaloidf)=\codom ( \arrowQuantaloidf)$ and $j(\arrowQuantaloidf)= \arrowQuantaloidf$.  In particular, identity morphisms in $\Q$ are always hermitian. The set of objects of $D\Q$ is then defined as the set of all hermitian morphisms of $\Q$ and is consequently a subset of $\morph(\Q)$. 
   The constructions of the composition and identities from (1) remain unchanged.
Since now $\arrowQuantaloidf$ and $\arrowQuantaloidg$ are hermitian, the involution $j$ on $\Q$ extends naturally to $D\Q$ by setting $\overline{j}(\arrowQuantaloidk):=j(\arrowQuantaloidk)$ for all $\arrowQuantaloidk\in \hom (\arrowQuantaloidf,\arrowQuantaloidg)$. With these modifications, $(D\Q,\overline{j})$ becomes an involutive quantaloid, and there exists also an embedding functor $\!\xy \morphism(0,0)/{ (}->/<500,0>[(\Q,j)`(D\Q,\overline{j});]\endxy\!$ as described in (1).
\end{remark}

In this context, we recall that a quantale $(\alg{{Q}},\ast)$ is a semigroup in $\cat{Sup}$, and an involution $j\colon\alg{{Q}}\to \alg{{Q}}$ on a quantale satisfies  $j\circ j=1_{\alg{{Q}}}$ and is an anti-homomorphism ---  i.e.,~a join-preserving map making the following diagram commutative:
\[
\bfig
\square(0,0)|alra|/`>`>`>/<1200,300>[\alg{{Q}}\otimes \alg{{Q}}`\alg{{Q}}\otimes \alg{{Q}}`\alg{{Q}}`\alg{{Q}};`\ast`\ast`j]
\morphism(0,300)<600,0>[\alg{{Q}}\otimes \alg{{Q}}`\alg{{Q}}\otimes \alg{{Q}};j\otimes j]
\morphism(600,300)<600,0>[\alg{{Q}}\otimes \alg{{Q}}`\alg{{Q}}\otimes \alg{{Q}};c_{\alg{{Q}}\alg{{Q}}}]
   \efig
\]
Here, $c_{\alg{{Q}}\alg{{Q}}}$ is a component of the symmetry in $\cat{Sup}$. For simplicity, we write $p^{\prime}$ for $j(p)$, where $p \in \alg{{Q}}$. Note also that a unital and involutive quantale is precisely the hom-space of an involutive quantaloid with a single object.

Since any unital quantale (i.e.,~a monoid in $\cat{Sup}$) can be identified with the hom-space of a quantaloid with a single object, the construction in Remark~\ref{newremark1.1}\,(1) shows that any \emph{unital} quantale $(\alg{{Q}},\ast,e)$ can be embedded into a quantaloid $D\alg{{Q}}$. Furthermore, every unital quantale can be embedded into a unital and involutive quantale via a unital quantale homomorphism (see\ \cite[Sect.~2]{GutHoh25} for details). Therefore, in the next remark, we focus on this situation and present the details of the constructions mentioned in Remark~\ref{newremark1.1}\,(2).

\begin{remark} \label{newremark1.2}
Let $\alg{{Q}}=(\alg{{Q}},\ast,e,\mbox{}^{\prime})$ be a unital and involutive quantale, which we extend to the involutive quantaloid $(D\alg{{Q}},j)$ by means of diagonal arrows. Then the details are as follows.
\begin{enumerate}[label=\textup{--},leftmargin=12pt,topsep=3pt,itemsep=0pt
]
\item The objects of $D\alg{{Q}}$ are hermitian elements of $\alg{{Q}}$ --- i.e.,~$D\alg{{Q}}_0=\set{{\elementQa}\in \alg{{Q}}\mid {\elementQa}^{\prime}={\elementQa}}$.
\item For each ${\elementQa},{\elementQb} \in D\alg{{Q}}_0$, the hom-space $\hom ({\elementQa},{\elementQb})$ is defined by:
\begin{align*}
\hom ({\elementQa},{\elementQb})&=\set{{\elementQlambda}\in \alg{{Q}}\mid \exists {\elementQlambda}_1,{\elementQlambda}_2 \in \alg{{Q}}:\,\ {\elementQlambda}= {\elementQlambda}_1 \ast {\elementQa}={\elementQb}\ast {\elementQlambda}_2
}\\
&=\set{{\elementQlambda}\in \alg{{Q}}\mid  {\elementQlambda}= ({\elementQlambda}\swarrow {\elementQa}) \ast {\elementQa}={\elementQb}\ast ({\elementQb}\searrow {\elementQlambda})
}
\end{align*}
and forms a complete sublattice of $\alg{{Q}}$ in the sense of $\cat{Sup}$. In particular, the hom-space associated with the unit of $\alg{{Q}}$ coincides with $\alg{{Q}}$ --- i.e.,~$\hom ({e},{e})=\alg{{Q}}$. Moreover, the hom-spaces whose domain or codomain is the object $\bot$ consist of a single arrow. Hence, $\bot$ is the zero object in the underlying ordinary category of $D\alg{{Q}}$.
\item The composition of morphisms ${\elementQlambda}\in \hom ({\elementQa},{\elementQb})$ and ${\elementQmu}\in \hom ({\elementQb},{\elementQc})$ is defined by:
\[{\elementQmu} \mathbin{\circ_{\elementQb}} {\elementQlambda}= ({\elementQmu} \swarrow {\elementQb}) \ast {\elementQb} \ast ({\elementQb}\searrow {\elementQlambda})={\elementQmu} \ast ({\elementQb}\searrow {\elementQlambda})=({\elementQmu}\swarrow {\elementQb}) \ast {\elementQlambda}.
\]
This composition $\circ_{\elementQb}$ is join-preserving in each variable separately. Moreover, since 
for all ${\elementQlambda}\in\hom ({\elementQb},{\elementQb})$ we have ${\elementQb} \mathbin{\circ_{\elementQb}} {\elementQlambda}={\elementQlambda}= {\elementQlambda} \mathbin{\circ_{\elementQb}} {\elementQb}$, it   
    follows that ${\elementQb}$ acts as the unit of the quantale $\hom ({\elementQb},{\elementQb})$. 
 It is also straightforward to verify that the composition coincides with the quantale multiplication if and only if all hermitian elements of $\alg{Q}$ are idempotent. In this context the quantale multiplication is read from right to left.
\item The involution  $j$ on $D\alg{{Q}}$ is inherited from the involution $\mbox{}^{\prime}$ on $\alg{{Q}}$ --- i.e.,~$j(\elementQlambda)=\elementQlambda^{\prime}$.
\end{enumerate}
If $\alg{{Q}}=(\alg{{Q}},\ast,e)$ is a commutative and unital quantale (which is trivially an involutive quantale w.r.t.\ the identity of $\alg{{Q}}$), then the residuals satisfy: 
\begin{equation*}\elementQa\rightarrow \elementQb:=\elementQa\searrow \elementQb=\elementQb \swarrow \elementQa
\end{equation*}
 for all ${\elementQa},{\elementQb} \in \alg{{Q}}$. In this case we adopt the notation $\rightarrow$ for the residual, and consequently the hom-spaces in $D\alg{{Q}}$ take the form: 
\begin{equation*}
\hom ({\elementQa},{\elementQb})=\set{{\elementQlambda}\in \alg{{Q}}\mid  {\elementQlambda}= {\elementQa}\ast ({\elementQa}\rightarrow {\elementQlambda})={\elementQb}\ast ({\elementQb}\rightarrow {\elementQlambda})}.
\end{equation*}
 Thus, $(D\alg{{Q}},j)$ becomes an involutive quantaloid, where the involution is 
 the bijection 
 \[j\colon\hom ({\elementQa},{\elementQb})\to \hom ({\elementQb},{\elementQa}),\qquad j(\newmorphism{16}{\elementQa}{\elementQlambda}{\elementQb})=\newmorphism{16}{\elementQb}{\elementQlambda}{\elementQa}.\]
Furthermore, each object ${\elementQa}\in\alg{{Q}}$ is the unit of the its endomorphism quantale $\hom ({\elementQa},{\elementQa})
$, and the composition of  ${\elementQlambda}\in \hom ({\elementQa},{\elementQb})$ and ${\elementQmu}\in \hom ({\elementQb},{\elementQc})$ is given by:
\stepcounter{num} 
\begin{equation}\label{eqn1.2} {\elementQmu} \mathbin{\circ_{\elementQb}} {\elementQlambda}= {\elementQb} \ast ({\elementQb}\rightarrow {\elementQmu})\ast ({\elementQb}\rightarrow {\elementQlambda})={\elementQmu} \ast ({\elementQb}\rightarrow {\elementQlambda})= {\elementQlambda}\ast ({\elementQb}\rightarrow {\elementQmu}).
\end{equation}
Finally, if $\alg{{Q}}$ is commutative and integral (i.e.,~the unit $e$ is the top element of $\alg{{Q}}$) then for all ${\elementQlambda}\in \hom ({\elementQa},{\elementQb})$, the relation ${\elementQlambda}\le{\elementQa}\wedge{\elementQb}$ holds. In this setting, the top element of $\hom ({\elementQa},{\elementQa})$ coincides with ${\elementQa}$, making each endomorphism quantale $\hom ({\elementQa},{\elementQa})$ integral. 
\end{remark}
\begin{histremark}\label{histremark1.3} If $\alg{{Q}}$ is an integral and idempotent quantale (i.e.,~a frame), then the quantaloid $(D\alg{{Q}},j) $ corresponds to a construction originally introduced by Walters in 1981 (\cite{Walters81}).
\end{histremark}

We now generalize this construction to the setting of involutive quantales, which need not be unital.
 
\begin{remark} (Cf.\ \cite[Sect.~6]{HK11})\label{newremark1.3} Let $\alg{{Q}}=(\alg{{Q}},\ast,\mbox{}^{\prime})$ be an involutive quantale. Then $\alg{{Q}}$ induces an involutive quantaloid $(D\alg{{Q}},j)$ as follows. We begin by fixing some terminology: An element $\elementQa\in \alg{{Q}}$ is \emph{self-divisible} if it is both a left \emph{and} a right divisor of itself. Clearly, every idempotent element of $\alg{{Q}}$ is self-divisible.
With this in place, the construction of the involutive quantaloid $(D\alg{{Q}},j)$ proceeds as follows:
\begin{enumerate}[label=\textup{--},leftmargin=12pt,topsep=3pt,itemsep=0pt
]
\item The objects of $D\alg{{Q}}$ are all self-divisible and hermitian elements of $\alg{{Q}}$.
\item For each pair ${\elementQa},{\elementQb} \in D\alg{{Q}}_0$, the hom-space $\hom ({\elementQa},{\elementQb})$ is defined as in Remark~\ref{newremark1.2}, namely:\footnote{In comparison with \cite[Sect.~6]{HK11} and \cite[Ex.~1.3]{ho11}, we note that in this presentation we adopt the convention that, in the case of idempotent quantales, the quantale multiplication is interpreted as composition in the categorical sense. Furthermore, we do not impose the condition that elements of $\hom (\elementQa, \elementQb)$ must be below $\elementQa\wedge  \elementQb$. Instead, we define $\hom (\elementQa, \elementQb)$ as the set of elements ${\elementQlambda}\in \alg{{Q}}$ such that $\elementQa$ is a right divisor of ${\elementQlambda}$ and $\elementQb$ is a left divisor of ${\elementQlambda}$. This approach was presented by H.~Lai and D.~Zhang at the 3rd International Conference on Quantitative Logic and Soft Computing, Xi'an, China, 2012.
}
\begin{equation*}
\hom ({\elementQa},{\elementQb})
=\set{{\elementQlambda}\in \alg{{Q}}\mid  {\elementQlambda}= ({\elementQlambda}\swarrow{\elementQa}) \ast {\elementQa}={\elementQb}\ast ({\elementQb}\searrow{\elementQlambda}}
\end{equation*}
\end{enumerate}
and is a complete sublattice  of $\alg{Q}$ in the sense of $\cat{Sup}$,  
since $\ast$ is join-preserving in each variable separately. Again the universal lower bound in $\alg{{Q}}$ is the zero object in the underlying ordinary category of $D\alg{{Q}}$.

 The definitions of \emph{composition} and  \emph{involution} remain as in Remark~\ref{newremark1.2}, including the treatment of commutative quantales.
In particular, since all objects in $D\alg{{Q}}$ are \emph{self-divisible}, each object ${\elementQa}$ again serves as the unit of the quantale $\hom ({\elementQa},{\elementQa})$. If all self-divisible and hermitian elements of $\alg{Q}$ are idempotent, then again the composition coincides with the quantale multiplication.
\end{remark}

This previous construction extends that of Remark~\ref{newremark1.2}: if the involutive quantale $\alg{{Q}}=(\alg{{Q}},\ast,\mbox{}^{\prime})$ is unital, then every element is self-divisible, and the two constructions of $(D\alg{{Q}},j)$ coincide. 
However, in the non-unital case, some elements of $\alg{{Q}}$ may not be self-divisible. As a result, the involutive quantaloid  $(D\alg{{Q}},j)$ may also arise from a unital subquantale of $\alg{{Q}}$, as illustrated in the next example.

\begin{example}\label{exam1.5} Consider an involutive quantale $\alg{{Q}}$ with at least $3$ elements, satisfying the  condition:
\[ a\ast b =\top,\qquad \text{for all }a ,b \in \alg{{Q}}\setminus\set{\bot}.\]  
Under this assumption, every element of $\alg{{Q}}\setminus \set{\bot,\top}$ fails to be self-divisible. In particular, $\alg{{Q}}$ is not unital and the set of objects in the associated quantaloid reduces to $D\alg{{Q}}_0=\set{\bot,\top}$. Consequently, the involutive quantaloid $(D\alg{{Q}},j)$ is isomorphic to 
 $(D\boldsymbol{2},j)$, where $\boldsymbol{2}$ denotes the \emph{unital} and commutative quantale on the  \donotbreakdash{$2$}chain $\set{\bot,\top}$.\vskip-15pt
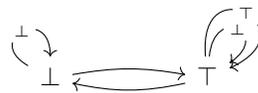
\[
  \begin{tikzcd}[column sep=10ex,row sep=2ex,
 every label/.append style = {font = \tiny}]
\bot\ar[r,bend left=10]
\ar[out=150, in=90, loop, distance=2em,"\bot" description] & \top\ar[l,bend left=12]
\ar[out=85, in=35, loop, distance=2em,"\bot" description]\ar[out=95, in=25, loop, distance=3.5em,"\top" description]
\end{tikzcd}
\]
\end{example}
%
  
However, in general, the construction described in Remark~\ref{newremark1.3} yields quantaloids that cannot be derived from a unital quantale. 
 We present three examples illustrating this situation.
 
\begin{example}\label{exam1.7} Let $\alg{{Q}}$ be a commutative quantale defined on the lattice $\set{\bot,a,b,\top}$, with the following Hasse diagram and multiplication table:
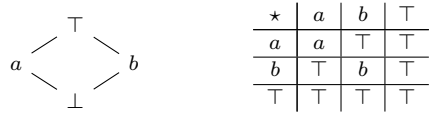
\[
\renewcommand\arraystretch{1.}
\setlength\doublerulesep{0pt}
{\footnotesize\begin{tikzcd}[column sep=10pt,row sep=3pt]
&\top\arrow[-]{dl}\arrow[-]{dr}\\
a&&b\\
&\bot\arrow[-]{ul}\arrow[-]{ur}
\end{tikzcd}}
\qquad\qquad
{\footnotesize
\begin{tabular}{c||c|c|c}
$\star$  & $a$ & $b$ & $\top$\\
\hline\hline
$a$ &  $a$ & $\top$ & $\top$\\\hline
$b$ &  $\top$ & $b$ & $\top$\\\hline
$\top$ &  $\top$ & $\top$ & $\top$\\
\end{tabular}}
\]
Since all elements of $\alg{{Q}}$ are idempotent and therefore self-divisible, the involutive quantaloid $(D\alg{{Q}},j)$
 is given by $ D\alg{{Q}}_0=\set{\bot,a,b,\top}$ and the corresponding hom-spaces whose domain and codomain are different from $\bot$ have the following form: 
\begin{enumerate}[label=\textup{--},leftmargin=12pt,topsep=3pt,itemsep=0pt
]
\item
   $\hom (a,a)=\set{\bot,a,\top}$, $\hom (b,b)=\set{\bot,b,\top}$ and $\hom (\top,\top)=\set{\bot,\top}$.
   \item For distinct $ x,y\in \set{a,b,\top}$, $\hom (x,y)=\set{\newmorphism{18}{x}{\bot}{y},\newmorphism{18}{x}{\top}{y}}$,
which are isomorphic to the \donotbreakdash{$\boldsymbol{2}$}chain.
 \end{enumerate}
Thus, the structure of $D\alg{{Q}}$ includes the following objects and arrows:
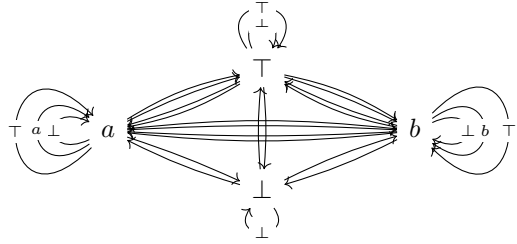
\[\begin{tikzcd}[column sep=10ex,row sep=2.ex,
 every label/.append style = {font = \tiny}]
&\top\ar[dd,bend left=5]
\ar[ld,bend left=5]
\ar[rd,bend left=5]
\ar[ld,bend left=10]
\ar[rd,bend left=10]
\ar[out=115, in=65, loop, distance=1.5em,"\bot" description]
\ar[out=125, in=55, loop, distance=2.5em,"\top" description]\\
a\ar[rr,bend left=3]
\ar[rr,bend left=6]
\ar[out=205, in=155, loop, distance=2.em,"\bot" description]
\ar[out=215, in=145, loop, distance=3.25em,"a" description]
\ar[out=225, in=135, loop, distance=5.5em,"\top" description]
\ar[ru,bend left=5]
\ar[ru,bend left=10]
\ar[rd,bend left=3]
&& b\ar[ul,bend left=5]
\ar[ul,bend left=10]
\ar[out=25, in=-25, loop, distance=2.em,"\bot" description]
\ar[out=35, in=-35, loop, distance=3.25em,"b" description]
\ar[out=45, in=-45, loop, distance=5.5em,"\top" description]
\ar[ll,bend left=3]
\ar[ll,bend left=6]
\ar[ld,bend left=5]
\\
&\bot\ar[uu,bend left=5]
\ar[ru,bend left=5]
\ar[lu,bend left=5]
\ar[out=305, in=235, loop, distance=1.5em,"\bot" description] 
\end{tikzcd}\] 

Note that this quantaloid \emph{cannot} be induced by a unital quantale as described in Remark~\ref{newremark1.2}. Indeed, since $D\alg{{Q}}$ has $4$ elements and every hom-space contains at most $3$ elements, none of $a$, $b$ or $\top$ can serve as a unit in a quantale with at least 4 elements.
\end{example}

\begin{example}\label{exam1.8} Let $\alg{{Q}}_{\boldsymbol{2}}$ denote the \emph{quantization of $\boldsymbol{2}$} (cf.\ \cite[Subsect.~2.2]{GutHoh25}). Specifically, the lattice $\alg{{Q}}_{\boldsymbol{2}}=\set{\bot,b,a_r,a_{\ell},c,\top}$ consists of $6$ elements, with the following Hasse diagram and multiplication table:
 \[
\renewcommand\arraystretch{1.}
\setlength\doublerulesep{0pt}
{\footnotesize
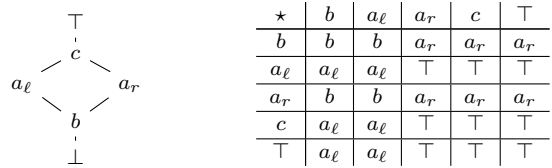
\begin{tikzcd}[column sep=6pt,row sep=2pt]
&\top \arrow[-]{d} &\\
&c\arrow[-]{dl}\arrow[-]{dr}\\
{a_{\ell}}&&{a_r}\\
&b\arrow[-]{ul}\arrow[-]{ur}&\\
&\bot \arrow[-]{u} 
\end{tikzcd}}
\qquad\qquad
{\footnotesize
\begin{tabular}{c||c|c|c|c|c}
$\star$  & $b$ & $a_{\ell}$ & $a_r$ & $c$ & $\top$\\
\hline\hline
$b$ &  $b$ & $b$ & $a_r$ & $a_r$ & $a_r$\\\hline
$a_{\ell}$ & $a_{\ell}$ & $a_{\ell}$ & $\top$ & $\top$ &$\top$\\\hline
$a_r$ & $b$ & $b$ & $a_r$ &$a_r$ & $a_r$\\\hline
$c$ &$a_{\ell}$ & $a_{\ell}$ & $\top$ & $\top$ & $\top$\\\hline
$\top$ &  $a_{\ell}$ & $a_{\ell}$ & $\top$ & $\top$ & $\top$\\
\end{tabular}}
\]
  The quantale $\alg{{Q}}_{\boldsymbol{2}}$ is non-commutative, and the involution is determined by $\top^{\prime}=\top$, $c^{\prime}=c$, $a_{\ell}^{\prime}=a_r$, $a_r^{\prime}=a_{\ell}$, $ b^{\prime}=b$ and  $\bot^{\prime}=\bot$. The set of  hermitian and self-divisible elements is $\set{\bot,b,\top}$. Therefore the \emph{involutive quantaloid} $\Q_{\boldsymbol{2}}:=(D\alg{{Q}}_{\boldsymbol{2}},j)$ induced by $\alg{{Q}}_{\boldsymbol{2}}$ is given by the set of objects $\set{\bot,b,\top}$ and all corresponding hom-spaces whose domain and codomain are different from $\bot$, are \donotbreakdash{$\boldsymbol{2}$}valued, namely,
 $\hom (\top,\top)=\set{\bot,\top}$, $\hom (\top,b)=\set{\bot,a_r}$, $ \hom (b,\top)=\set{\bot,a_{\ell}}$, and $\hom (b,b)=\set{\bot, b}$.
 
Thus, the structure of $\Q_{\boldsymbol{2}}$ includes the following objects and arrows:
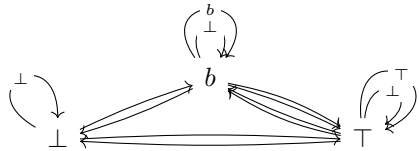
\[\begin{tikzcd}[column sep=10ex,row sep=2.ex,every label/.append style = {font = \tiny}]
&b\ar[dl,bend left=5]
\ar[dr,bend left=5]
\ar[out=115, in=65, loop, distance=1.75em,"\bot" description]
\ar[out=125, in=55, loop, distance=3.em,"b" description]
\ar[dr,bend left=10]
\\
\bot\ar[out=150, in=90, loop, distance=3em,"\bot" description]
\ar[rr,bend left=3]
\ar[ur,bend left=4]
& &\top\ar[ll,bend left=3]
\ar[ul,bend left=5]
\ar[out=85, in=35, loop, distance=2.em,"\bot" description]
\ar[out=95, in=25, loop, distance=3.5em,"\top" description]
\ar[ul,bend left=10]
\end{tikzcd}
\]
Note again that this quantaloid cannot be induced by a unital quantale as described in Remark~\ref{newremark1.2}, 
since all hom-spaces do not contain $3$ elements. Moreover, the relations $a_{\ell}\ast a_r=\top$ and $a_r\ast a_{\ell}=b$ highlight that the
 objects $b$ and $\top$ in $\Q_{\boldsymbol{2}}$ are isomorphic in the underlying ordinary category of $\Q_{\boldsymbol{2}}$, which is equivalent to a category consisting of two objects, where one of them is the zero object. This last observation underlines the difference between the unital quantale $\boldsymbol{2}$ and its quantization $\alg{Q}_{\boldsymbol{2}}$.
\end{example}

The next example extends the quantization of $\boldsymbol{2}$ 
in such a way that the previously non-self-divisible element $c$ becomes self-divisible.
  
\begin{example} (cf.\ $\alg{R}_4$ in \cite{GutHoh24}) \label{exam1.10}
Let $\alg{{Q}}$ be the quantale defined on the lattice 
\[\set{\bot,b,a_{\ell},a_{r},\newal,\newar, c, \top}
\]
 with the following Hasse diagram and multiplication table:
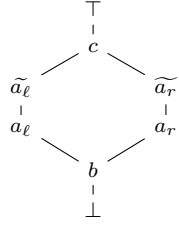
\[
\renewcommand\arraystretch{1.1}
\setlength\doublerulesep{0pt}
{\footnotesize\begin{tikzcd}[column sep=13pt,row sep=4pt]
&\top \arrow[-]{d} &\\
&c\arrow[-]{dl}\arrow[-]{dr}\\
\newal&&\newar\\
{a_{\ell}}\arrow[-]{u}&&{a_{r}}\arrow[-]{u}\\
&b\arrow[-]{ul}\arrow[-]{ur}\\
&\bot \arrow[-]{u} &
\end{tikzcd}}
\qquad\qquad
{\footnotesize
\begin{tabular}{c||c|c|c|c|c|c|c}
$\ast$ 
& $b$ & $a_{\ell}$ & $a_r$ & $\newal$& $\newar$ & $c$  & $\top$ \\
\hline\hline
$b$ 
& $b$ & $b$ & $a_r$ & $b$ & $a_r$ & $a_r$ &  $a_r$ \\\hline
$a_{\ell}$ 
& $a_{\ell}$ & $a_{\ell}$ & $\top$ & $a_{\ell}$ & $\top$ & $\top$ &  $\top$ \\\hline
$a_r$ 
& $b$ & $b$ & $a_r$ &$a_r$ & $a_r$ & $a_r$ &  $a_r$\\\hline
$\newal$ 
& $a_{\ell}$ & $a_{\ell}$ & $\top$ & $\newal$ & $\top$ & $\top$ &  $\top$ \\\hline
$\newar$ 
& $b$ & $a_{\ell}$ & $a_r$ & $c$& $\newar$ & $c$ &  $\top$ \\\hline
 $c$ 
& $a_{\ell}$ & $a_{\ell}$ & $\top$ & $c$ & $\top$ & $\top$ &  $\top$ \\\hline
 $\top$ 
 & $a_{\ell}$ & $a_{\ell}$ & $\top$ & $\top$ & $\top$ & $\top$ &  $\top$ 
 \end{tabular}}
\]
Hence $\alg{{Q}}$ is non-commutative, and the involution is defined by $\top^{\prime}=\top$, $c^{\prime}=c$, $\newal^{\prime}=\newar$, $\newar^{\prime}=\newal$, $a_{\ell}^{\prime}=a_r$, $a_r^{\prime}=a_{\ell}$, $ b^{\prime}=b$ and $\bot^{\prime}=\bot$. Although $c$ is not idempotent, it is self-divisible, and thus the set of hermitian and self-divisible elements is $\set{\bot,b,c,\top}$. Therefore, the involutive quantaloid $(D\alg{{Q}},j)$ is given by $D\alg{{Q}}_0=\set{\bot,b,c,\top}$, and the hom-spaces whose domain and codomain are different from $\bot$, are the following ones:  $\hom (\top,\top)=\set{\bot,\top}$, $\hom (c,c)=\set{\bot,c,\top}$, $\hom (b,b)=\set{\bot, b}$, $\hom (\top,c)=\set{\bot,\top}$, $ \hom (c,\top)=\set{\bot,\top}$, $\hom (c,b)=\set{\bot,a_r}$, $\hom (b,c)=\set{\bot,a_{\ell}}$, $(\top,b)=\set{\bot,a_r}$, and $\hom (b,\top)=\set{\bot, a_{\ell}}$.

Thus, the structure of $D\alg{{Q}}$ includes the following objects and arrows:
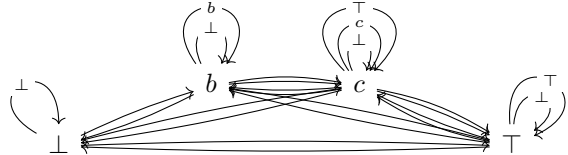
\[\begin{tikzcd}[column sep=10ex,row sep=2.ex,every label/.append style = {font = \tiny}]
&b\ar[dl,bend left=5]
\ar[out=115, in=65, loop, distance=2em,"\bot" description]
\ar[out=125, in=55, loop, distance=3.5em,"b" description]
\ar[drr,bend left=3]
\ar[r,bend left=5]
\ar[r,bend left=10]
&c
\ar[out=115, in=65, loop, distance=1.75em,"\bot" description]
\ar[out=125, in=55, loop, distance=2.90em,"c" description]
\ar[out=135, in=45, loop, distance=4.55em,"\top" description]
\ar[dr,bend left=5]
\ar[dr,bend left=10]
\ar[l,bend left=5]
\ar[l,bend left=10]
\ar[dll,bend left=3]
\\
\bot\ar[out=150, in=90, loop, distance=3em,"\bot" description]
\ar[rrr,bend left=2]
\ar[urr,bend left=3]
\ar[ur,bend left=4]
& &
&\top
\ar[ull,bend left=3]
\ar[lll,bend left=2]
\ar[ul,bend left=5]
\ar[ul,bend left=10]
\ar[out=85, in=35, loop, distance=2.em,"\bot" description]
\ar[out=95, in=25, loop, distance=3.5em,"\top" description]
\end{tikzcd}
\]
Since $\alg{{Q}}_{\boldsymbol{2}}$ is a subquantale of $\alg{{Q}}$, the quantaloid $(\Q_{\boldsymbol{2}},j)$ is a involutive subquantaloid of $(D\alg{{Q}},j)$.
\end{example}

We conclude this section with a property of quantaloids that plays an important role in Section~\ref{section4}.
 
Let $\Q$ be a quantaloid and let ${\objectQuantaloida}\in\Q_0$.  Then $\Q$ is said to be \emph{\donotbreakdash{$\objectQuantaloida$}stable}, if the following condition holds for all ${\objectQuantaloidb}\in\Q_0$\textup:
\begin{equation*}
1_{\objectQuantaloidb}\le (\tbigvee \hom (\objectQuantaloida,\objectQuantaloidb))\circ (\tbigvee \hom (\objectQuantaloidb,\objectQuantaloida)).
\end{equation*}
  
A quantaloid $\Q$ is \emph{stable} if there exists an object $\objectQuantaloida\in \Q_0$ such that $\Q$ is \donotbreakdash{$\objectQuantaloida$}stable.

\begin{remarks} \label{remarks1.12} (1) If $\alg{{Q}}$ is an involutive quantale and ${\elementQa}\in D\alg{{Q}}_0$, then $(D\alg{{Q}},j)$ is \donotbreakdash{$\elementQa$}stable if and only if ${\elementQb}\le (\tbigvee \hom ({\elementQa},{\elementQb}))\mathbin{\circ_{{\elementQa}}} (\tbigvee \hom ({\elementQb},{\elementQa}))$ for all ${\elementQb}\in D\alg{{Q}}_0$. Moreover, if $\alg{{Q}}$ is integral,    
then $\tbigvee \hom ({\top},{\elementQb})=
\newmorphism{16}{\top}{\raisebox{3pt}{$\elementQb$}}{\elementQb}$, $\tbigvee \hom ({\elementQb},{\top})=\newmorphism{16}{\elementQb}{\raisebox{3pt}{$\elementQb$}}{\top}$ and $(\newmorphism{16}{\top}{\raisebox{3pt}{$\elementQb$}}{\elementQb})\mathbin{\circ_{\top}}\newmorphism{16}{\elementQb}{\elementQb}{\top}=b\ast b$ hold in $(D\alg{{Q}},j)$, and consequently, $(D\alg{{Q}},j)$ is \donotbreakdash{$\top$}stable if and only if $\alg{{Q}}$ is idempotent --- i.e.,~$\alg{{Q}}$ is a frame (cf.\ Historical Remark~\ref{histremark1.3}).
 \\[2pt]
(2) If $\alg{{Q}}$ is an involutive and integral quantale with a non-idempotent and 
hermitian element $\elementQb$, 
then $(D\alg{{Q}},j)$ is \emph{not} \donotbreakdash{$\top$}stable, since $(\tbigvee \hom (\top,\elementQb)) \mathbin{\circ_{\top}} (\tbigvee \hom (\elementQb,\top))= \elementQb \ast \elementQb < \elementQb$.
\\[2pt]
(3) Regarding Examples  \ref{exam1.7}, \ref{exam1.8} and \ref{exam1.10}, we observe the following:
All involutive quantaloids discussed are stable. More precisely, each one is \donotbreakdash{$\objectQuantaloida$}stable for all $\objectQuantaloida \in D\alg{{Q}}_0\setminus \{\bot\}$.
\end{remarks}

\section{$\Q$-enriched categories}\label{section2}
A \emph{typed set} is an object $X$ of the slice category $\cat{Set}/\Q_0$ --- i.e.,~a map $t\colon X_0\to \Q_0$. Since we treat $t$ as a generic notation, we will also denote it by $|{\phantom{x}}|$, interpreting 
$|{x}|$ as the type of the element $x\in X_0$.
A \emph{\donotbreakdash{$\Q$}enriched category} (or \donotbreakdash{$\Q$}category, for short) is a pair $(X,\alpha)$ where $X=(X_0,t)$ is a typed set, and $\alpha\colon X_0\times X_0 \to \morph(\Q)$ is a hom-arrow-assignment satisfying the following conditions: 
\begin{enumerate}[label=\textup{(C\arabic*)},topsep=3pt,itemsep=0pt
]
\item\label{C1} $\alpha(x,y)\in \hom (|y|,|x|), \qquad x,y\in X_0$.
\item\label{C2} $\alpha(x,y)\circ \alpha(y,z)\le \alpha (x,z), \qquad x,y,z\in X_0.$
\item\label{C3} $ 1_{|x|} \le \alpha(x,x), \qquad x\in X_0$.
\end{enumerate}
As an immediate consequence of \ref{C2} and \ref{C3} we obtain the following identity:
\stepcounter{num} 
\begin{equation}\label{n2.1C4} \alpha(x,x)\circ \alpha(x,y)=\alpha(x,y)= \alpha(x,y)\circ \alpha(y,y),\qquad x,y\in X_0.
\end{equation}
In particular, $\alpha(x,x)$ is idempotent w.r.t.\ the composition for all $x\in X_0$.
\begin{comment} Let ${\objectQuantaloida}$ be an object of a quantaloid $\Q$, and let $(X,\alpha)$ be a \donotbreakdash{$\Q$}category. Define the  \emph{\donotbreakdash{${\objectQuantaloida}$}section} of $X$ as $X_{\objectQuantaloida}=\set{x\in X_0\mid  |x|={\objectQuantaloida}}$.
 Then the restriction of $\alpha$ to $X_{\objectQuantaloida}\times X_{\objectQuantaloida}$ yields a quantale-enriched category. This observation allows us to interpret quantaloid-enriched categories as \emph{time-varying} quantale-enriched categories, where the enrichment varies over the objects of the quantaloid.
\end{comment}
\begin{example}\label{exam2.1} Let $\Q$ be a quantaloid, which we now interpret as a typed set --- i.e.,~we consider the identity map $1_{\Q_0}$ on the set of objects of $\Q$ as the type map, and we continue to denote the pair $(\Q_0,1_{\Q_0})$ simply by $\Q$, provided no confusion arises. We define a hom-arrow-assignment $\tau\colon\Q_0\times  \Q_0 \to \morph(\Q)$ 
by $\tau({\objectQuantaloida},{\objectQuantaloidb})= \tbigvee \hom ({\objectQuantaloidb},{\objectQuantaloida})$ for each ${\objectQuantaloida},{\objectQuantaloidb}\in \Q_0$.
Then the pair $(\Q,\tau)$ forms a \donotbreakdash{$\Q$}category.
\end{example}

Let $(X,\alpha)$ and $(Y,\beta)$ be \donotbreakdash{$\Q$}categories. A \emph{distributor}
$\Phi\colon (X,\alpha)\circlearrow(Y,\beta)
$ is a map $\Phi\colon Y_0\times X_0 \to \morph(\Q)$ satisfying the following conditions for all $x,x_1,x_2\in X_0$ and $y,y_1,y_2\in Y_0$:
\begin{enumerate}[label=\textup{(D\arabic*)},topsep=3pt,itemsep=0pt
]
\item\label{D1} $\Phi(y,x)\in \hom (|x|,|y|)$.
\item\label{D2} $\beta(y_1,y_2)\circ \Phi(y_2,x)\le \Phi(y_1,x)$ and $\Phi(y,x_1)\circ \alpha(x_1,x_2)\le \Phi(y,x_2)$.
\end{enumerate}
With regard to \ref{C3} and \ref{D2} every distributor $\Phi\colon (X,\alpha)\circlearrow(Y,\beta)
$ satisfies the following:
\stepcounter{num}
\begin{equation} \label{n2.2CC} \Phi(y,x)=\beta(y,y) \circ \Phi(y,x)=\Phi(y,x)\circ \alpha(x,x), \qquad x\in X_0,\, y\in Y_0.
\end{equation}
The composition of distributors $\Phi\colon (X,\alpha)\circlearrow(Y,\beta)
$ and $\Psi\colon (Y,\beta)\circlearrow(Z,\gamma)
$ is defined by:
\[ (\Psi\otimes \Phi)(z,x)= \tbigvee_{y\in Y _0} \Psi(z,y) \circ \Phi(y,x), \qquad(z,x)\in Z_0\times X_0.
    \]
The hom-arrow-assignment $\alpha\colon (X,\alpha)\circlearrow(X,\alpha)
$ is itself a distributor and serves as the identity w.r.t.\ the composition $\otimes$. Since the pointwise join of distributors is again a distributor, the collection of \donotbreakdash{$\Q$}categories with distributors as morphisms forms a quantaloid, and in particular, a \donotbreakdash{$2$}category.

Two distributors $\Phi\colon (X,\alpha)\circlearrow(Y,\beta)
$ and $\Psi\colon (Y,\beta)\circlearrow(X,\alpha)
$ are adjoint  if and only if the following relations hold for all $x_1,x_2\in X_0$ and $y_1,y_2\in Y_0$:
\[\alpha(x_1,x_2)\le (\Psi\otimes \Phi)(x_1,x_2)\quad \text{and}
 \quad (\Phi\otimes \Psi)(y_1,y_2) \le \beta(y_1,y_2).\]
In light of \ref{D2}, \ref{n2.1C4} and \ref{n2.2CC}, adjointness of distributors is equivalent to the following conditions for all $x\in X_0$ and $y_1,y_2\in Y$:
\begin{enumerate}[label=\textup{(D\arabic*)}, start=3,topsep=3pt,itemsep=0pt
]
\item\label{D3} $1_{|x|} \le \tbigvee_{y\in Y_0} \Psi(x,y) \circ \Phi(y,x) \quad \text{and}\quad \Phi(y_1,x)\circ \Psi(x,y_2)\le \beta(y_1,y_2)$.
\end{enumerate}
We denote this adjointness by $\Phi \dashv \Psi$, where $\Phi$ is the left adjoint distributor and $\Psi$ the right adjoint distributor. Since the right adjoint distributor $\Psi$ is uniquely determined by the left adjoint $\Phi$, $\Psi$ can be expressed by $\Phi$ as follows:
\stepcounter{num}
\begin{equation}\label{n2.2}
\Psi(x,y)= \tbigvee_{y_1\in Y_0}\Phi(y_1,x) \searrow \beta(y_1,y), \qquad x\in X_0,\, y\in Y.
\end{equation}

Let $(X,\alpha)$ and $(Y,\beta)$ be \donotbreakdash{$\Q$}categories.
A \emph{functor} $\varphi\colon(X,\alpha)\to(Y,\beta)$ is a morphism in the slice category $\cat{Set}/\Q_0$ --- i.e.,~a function $\varphi\colon X_0\to Y_0$ making the following diagram commute
\[
\bfig
 \Vtriangle(0,0)|alr|/>`>`>/<450,300>[X_0`Y_0`\Q_0;\varphi`|\phantom{x}|`|\phantom{x}|]
 \efig
\]
and satisfying the following enrichment condition, which has its origin in \cite{EK66}:
\begin{enumerate}[label=\textup{(m)},topsep=3pt,itemsep=0pt
]
\item $\alpha(x_1,x_2) \le \beta(\varphi(x_1),\varphi(x_2)),\qquad x_1,x_2\in X_0$.
\end{enumerate} 

Since the axioms \ref{C2} and \ref{C3} are preserved under pointwisely defined meets of hom-arrow-assignments, the category $\cat{Cat}(\Q)$ of \donotbreakdash{$\Q$}categories and functors is topological over $\cat{Set}/\Q_0$ w.r.t. the forgetful functor from $\cat{Cat}(\Q)$ to $\cat{Set}/\Q_0$. In particular, for any type map $|\phantom{x}|\colon X_0\to \Q_0$ the \emph{discrete} \donotbreakdash{$\Q$}category $(X,\delta)$ is given by:
\[\delta(x,y)=\begin{cases}\bot,&\text{if } x\neq y,\\ 1_{|x|},&\text{if } x=y, \end{cases} 
 \qquad x,y\in X_0.\]
If $\varphi\colon(X,\alpha)\to (Y,\beta)$ is a functor, it induces a pair of adjoint distributors $\varphi_b\dashv \varphi^b$ between $(X,\alpha)$ and $(Y,\beta)$ defined by:
\[ \varphi_b(y,x)=\beta(y,\varphi(x)) \quad \text{and}\quad \varphi^b(x,y)=\beta(\varphi(x),y), \qquad x\in X_0,\, y\in Y_0.
   \]
Obviously this correspondence describes a special relationship between  functors and left adjoint distributors. Anticipating the terminology at the beginning of Section~\ref{section4} this relationship is injective, if the codomain of $\varphi$ is separated. 
 
Let $(X,\alpha)$ be a \donotbreakdash{$\Q$}category. For each element $x\in X_0$ we can associate a singleton \donotbreakdash{$\Q$}category $(\{x\},\gamma)$, where we restrict the type function to $\{x\}$ and
    $\gamma(x,x)=\alpha(x,x)$.
 If $\Phi\colon (X,\alpha)\circlearrow(Y,\beta)$
is left adjoint to thedistributor $\Psi\colon (Y,\beta)\circlearrow(X,\alpha)
$, then for each $x\in X_0$ it induces a left adjoint distributor $\tau_x\colon  (\{x\},\gamma)\circlearrow(Y,\beta)
$ defined by
\[ \tau_x(y,x)=\Phi(y,x), \qquad y\in Y_0.
\]
The corresponding right adjoint distributor $\sigma_x\colon(Y,\beta)\circlearrow  (\{x\},\gamma)$ is determined by 
\[\sigma_x(x,y)=\Psi(x,y),\qquad y\in Y_0.
\]
 Moreover, $\Phi$ is left adjoint if and only if $\tau_x$ is left adjoint for all $x\in X_0$.

\begin{proposition}\label{prop2.1} Let ${\Phi,\Phi^{\prime}\colon (X,\alpha)\circlearrow(Y,\beta)}$ and $\Psi,\Psi^{\prime}\colon(Y,\beta)\circlearrow(X,\alpha)$ be pairs of distributors such that $\Phi \dashv \Psi$, $\Phi^{\prime} \dashv \Psi^{\prime}$, $\Phi^{\prime}\le \Phi$ and $\Psi^{\prime} \le \Psi$.
Then $\Phi=\Phi^{\prime}$ and $\Psi=\Psi^{\prime}$.
\end{proposition}

\begin{proof} Using (\ref{n2.2CC}) and the axioms  \ref{D2} and \ref{D3} we compute:
\begin{align*}\Phi(y,x)&=\Phi(y,x)\circ \alpha(x,x)\le \tbigvee_{z\in Y}\Phi(y,x) \circ \Psi^{\prime}(x,z)\circ \Phi^{\prime}(z,x)\\
&\le \tbigvee_{z\in Y}\Phi(y,x) \circ\Psi(x,z)\circ \Phi^{\prime}(z,x)\le \tbigvee_{z\in Y} \beta(y,z)\circ \Phi^{\prime}(z,x)=\Phi^{\prime}(y,x).
\end{align*}
Hence $\Phi=\Phi^{\prime}$, and by adjointness, it follows that $\Psi=\Psi^{\prime}$.
\end{proof}
\begin{remark}[Terminal object in $\cat{Cat}(\Q)$]\label{rem2.3} 
Recall from Example~\ref{exam2.1} that the identity map $1_{\Q_0}$ serves as the type map for the \donotbreakdash{$\Q$}category $(\Q,\tau)$. For any \donotbreakdash{$\Q$}category $(X,\alpha)$, the type map $|{\phantom{x}}|\colon X_0\to \Q_0$ is the unique functor $|{\phantom{x}}|\colon(X,\alpha)\to (\Q,\tau)$.
Therefore, $(\Q,\tau)$ is the terminal object in the category $\cat{Cat}(\Q)$.
 \end{remark}
\section{$\Q$-enriched presheaves}\label{section3}
Let ${\objectQuantaloida}\in \Q_0$ be an object of a quantaloid $\Q$. We consider the discrete \donotbreakdash{$\Q$}category $(\{\cdot\}, \delta)$ on the singleton $\{\cdot\}$ with type $\objectQuantaloida$, where $\delta(\cdot,\cdot)=1_{\objectQuantaloida}$.
A distributor $\Pi\colon (X,\alpha)\circlearrow(\{\cdot\},\delta)$ is called a \emph{covariant \donotbreakdash{$\Q$}presheaf of type $\objectQuantaloida$ on $(X,\alpha)$}. Such a distributor can be identified with a map ${\Qpresheavef}\colon X \to \morph(\Q)$ satisfying the following conditions:
\begin{enumerate}[label=\textup{(Q\arabic*)},start=0,topsep=3pt,itemsep=0pt
]
\item\label{Q0} ${\Qpresheavef}(x)\in \hom (|x|,{\objectQuantaloida}), \qquad x\in X_0$.
\item\label{Q1} ${\Qpresheavef}(x_1)\circ\alpha(x_1,x_2) \le {\Qpresheavef}(x_2), \qquad x_1,x_2\in X_0$.
\end{enumerate}
Since ${\objectQuantaloida}$ is the type of ${\Qpresheavef}$, we denote a covariant \donotbreakdash{$\Q$}presheaf by the pair $({\objectQuantaloida},{\Qpresheavef})$, where ${\objectQuantaloida}$ is an object of $\Q$ and ${\Qpresheavef}$ is a arrow-valued map satisfying \ref{Q0} and \ref{Q1}.
\begin{remark}\label{rem3.1} Let ${\objectQuantaloidd}\in \Q_0$ and define the set
$\D_0^{\objectQuantaloidd}=\set{{\arrowQuantaloidf}\in\morph(\Q)\mid\codom ({\arrowQuantaloidf})={\objectQuantaloidd}}$.
On $\D_0^{\objectQuantaloidd}$ the type map $|\phantom{x}|$ is given by the restriction of the domain map $\dom $ to $\D_0^{\objectQuantaloidd}$, so the pair $(\D_0^{\objectQuantaloidd},|\phantom{x}|)$ is denoted by $\D^{\objectQuantaloidd}$. 
We define a hom-arrow-assignment $\varkappa\colon \D_0^{\objectQuantaloidd}\times \D_0^{\objectQuantaloidd} \to \morph(\Q)$ as follows:
\[\varkappa({\objectQuantaloida},{\objectQuantaloidb})=  {\objectQuantaloida}\searrow {\objectQuantaloidb}, \qquad {\objectQuantaloida},{\objectQuantaloidb}\in\D_0^{\objectQuantaloidd}.
 \]
Then, a map ${\Qpresheavef}\colon X_0\to\D_0^{\objectQuantaloidd}$ is a covariant \donotbreakdash{$\Q$}presheaf of type ${\objectQuantaloidd}$ on $(X,\alpha)$ if and only if the diagram 
\[
\bfig
 \Vtriangle(0,0)|alr|/>`>`>/<450,300>[X`\D_0^{\objectQuantaloidd}`\Q_0;{\Qpresheavef}`|\phantom{x}|`\dom ]
 \efig
\]
is commutative (cf.\ \ref{Q0}) and the following relation holds (cf.\ \ref{Q1}):
\[\alpha(x_2,x_1)\le \varkappa({\Qpresheavef}(x_2),{\Qpresheavef}(x_1)), \qquad x_1, x_2\in X_0,\]
--- i.e.,~${\Qpresheavef}\colon (X,\alpha)\to (\D^{\objectQuantaloidd},\varkappa)$ is a functor in $\cat{Cat}(\Q)$.
\end{remark}

A distributor $\Pi\colon (\{\cdot\},\delta)\circlearrow(X,\alpha)$ is called a \emph{contravariant \donotbreakdash{$\Q$}presheaf of type $\objectQuantaloida$ on $(X,\alpha)$}. Such a distributor  can be identified with a map ${\Qpresheaveg}\colon X\to \morph(\Q)$ satisfying the following conditions:
\begin{enumerate}[label=\textup{(P\arabic*)},start=0,topsep=3pt,itemsep=0pt
]
\item\label{P0} ${\Qpresheaveg}(x)\in \hom ({\objectQuantaloida},|x|), \qquad x\in X_0$.
\item\label{P1} $\alpha(x_1,x_2)\circ {\Qpresheaveg}(x_2)\le {\Qpresheaveg}(x_1), \qquad x_1,x_2\in X_0$.
\end{enumerate}
Since ${\objectQuantaloida}$ is the type of ${\Qpresheaveg}$, we denote a contravariant \donotbreakdash{$\Q$}presheaf by the pair $({\objectQuantaloida},{\Qpresheaveg})$ where ${\objectQuantaloida}$ is an object  of $\Q$ and ${\Qpresheaveg}$ is an arrow-valued map satisfying \ref{P0} and \ref{P1}.

\begin{remark}\label{rem3.2} Let ${\objectQuantaloidc}\in \Q_0$ and define the set 
$\R_0^{\objectQuantaloidc}=\set{{\arrowQuantaloidf}\in\morph(\Q)\mid\dom ({\arrowQuantaloidf})={\objectQuantaloidc}}$.
On $\R_0^{\objectQuantaloidc}$ the type map $|\phantom{x}|$ is given by the restriction of the codomain map $\codom $ to $\R_0^{\objectQuantaloidd}$, so the pair $(\R_0^{\objectQuantaloidc},|\phantom{x}|)$ is denoted by $\R^{\objectQuantaloidc}$. 

We define a hom-arrow-assignment $\varrho\colon\R_0^{\objectQuantaloidc}\times \R_0^{\objectQuantaloidc} \to \morph(\Q)$ by:
\[\varrho({\arrowQuantaloidf},{\arrowQuantaloidg})=  {\arrowQuantaloidf}\swarrow {\arrowQuantaloidg}, \qquad {\arrowQuantaloidf},{\arrowQuantaloidg}\in \R_0^{\objectQuantaloidc}.\]
Then, a map ${\Qpresheaveg}\colon X_0\to \R_0^{\objectQuantaloidc}$ is a contravariant \donotbreakdash{$\Q$}presheaf of type ${\objectQuantaloidc}$ on $(X,\alpha)$ if and only if the diagram 
\[ 
\bfig
 \Vtriangle(0,0)|alr|/>`>`>/<450,300>[X`\R_0^{\objectQuantaloidc}`\Q_0;{\Qpresheaveg}`|\phantom{x}|`\codom ]
 \efig
\]
is commutative (cf.\ \ref{P0}) and the following relation holds (cf.\ \ref{P1}):
\[\alpha(x_1,x_2)\le \varrho({\Qpresheaveg}(x_1),{\Qpresheaveg}(x_2)), \qquad x_1,x_2\in X_0,\]
--- i.e.,~${\Qpresheaveg}\colon(X,\alpha)\to (\R^{\objectQuantaloidc},\varrho)$ is a functor in $\cat{Cat}(\Q)$.
\end{remark}

\begin{remark}\label{rem3.3} Let $(\Q,j)$ be an involutive quantaloid. Then the \emph{dual \donotbreakdash{$\Q$}category} $(X,\alpha^{op})$ of $(X,\alpha)$ exists --- i.e.,~
\[\alpha^{op}(x,y)= j(\alpha(y,x)), \qquad x,y \in X_0. \]
Hence, a pair $({\objectQuantaloida},{\Qpresheaveg})$ is a contravariant \donotbreakdash{$\Q$}presheaf on $(X,\alpha)$ if and only if the composition $ j\circ {\Qpresheaveg}\colon (X,\alpha^{op}) \to (D^{\objectQuantaloida},\varkappa)$ is a functor. This condition provides a justification for the use of the term \emph{contravariant} in the chosen terminology.
\end{remark}

We now proceed with the construction of \donotbreakdash{$\Q$}categories consisting of covariant (resp.\ contravariant) \donotbreakdash{$\Q$}presheaves on a given \donotbreakdash{$\Q$}category $(X,\alpha)$.

Let $\Q(X,\alpha)_0$ be the set of all covariant \donotbreakdash{$\Q$}pre\-sheaves on $(X,\alpha)$.     
Define the type map $|\phantom{x}|\colon \Q(X,\alpha)_0 \to \Q_0$ by $|({\objectQuantaloida},{\Qpresheavef})|={\objectQuantaloida}$. 
Consequently the pair $(\Q(X,\alpha)_0,|\phantom{x}|)$ is denoted by $\Q(X,\alpha)$. It is not difficult to show that the pair $(\Q(X,\alpha),\upsilon)$ equipped 
with the hom-arrow-assignment $\upsilon$ defined by:
\begin{align*}\upsilon(({\objectQuantaloida}_1,{\Qpresheavef}_1),({\objectQuantaloida}_2,{\Qpresheavef}_2))&= \tbigwedge_{x\in X_0} ({\Qpresheavef}_1(x) \swarrow {\Qpresheavef}_2(x))\\
&=\tbigvee\{{\objectQuantaloidc}\in\hom ({\objectQuantaloida}_2,{\objectQuantaloida}_1)\mid {\objectQuantaloidc} \circ {\Qpresheavef}_2(x) \le {\Qpresheavef}_1(x) \text{ for all } x\in X_0\}\end{align*}
forms a \donotbreakdash{$\Q$}category. Finally, if $(\{\objectQuantaloida\},\delta)$ is the discrete \donotbreakdash{$\Q$}category with $|\objectQuantaloida|=\objectQuantaloida$, then it is interesting to observe that the \donotbreakdash{$\Q$}categories $(\Q(\{\objectQuantaloida\},\delta),\upsilon)$ and $(\R^{\objectQuantaloida},\varrho)$ are isomorphic.
\\
Similarly, let $P(X,\alpha)_0$ be the set of all contravariant \donotbreakdash{$\Q$}pre\-sheaves on $(X,\alpha)$. Define the type map $| \phantom{x} |\colon P(X,\alpha)_0 \to \Q_0$ by $|({\objectQuantaloida},{\Qpresheaveg})|={\objectQuantaloida}$. Consequently the pair $(P(X,\alpha)_0,|\phantom{x}|)$ is denoted by $P(X,\alpha)$. 
 It is not difficult to show that the pair $(P(X,\alpha),\pi)$ equipped  with the hom-arrow-assignment $\pi$ defined by:
\begin{align*}\pi(({\objectQuantaloida}_1,{\Qpresheaveg}_1),({\objectQuantaloida}_2,{\Qpresheaveg}_2))&= \tbigwedge_{x\in X_0} ({\Qpresheaveg}_1(x) \searrow {\Qpresheaveg}_2(x))\\
&=\tbigvee\set{ {\objectQuantaloidc}\in \hom ({\objectQuantaloida}_2,{\objectQuantaloida}_1)\mid {\Qpresheaveg}_1(x)\circ {\objectQuantaloidc} \le {\Qpresheaveg}_2(x) \text{ for all } x\in X_0}.\end{align*}
forms a \donotbreakdash{$\Q$}category. Again, if $(\{\objectQuantaloida\},\delta)$ is the discrete \donotbreakdash{$\Q$}category with $|\objectQuantaloida|=\objectQuantaloida$, then  the \donotbreakdash{$\Q$}categories $(P(\{\objectQuantaloida\},\delta),\pi)$ and $(\D^{\objectQuantaloida},\varkappa)$ are isomorphic.
\\
Let $({\objectQuantaloida},{\Qpresheavef})$ be a covariant \donotbreakdash{$\Q$}presheaf  on $(X,\alpha)$ and let $({\objectQuantaloida},{\Qpresheaveg})$ be a contravariant \donotbreakdash{$\Q$}pre\-sheaf on $(X,\alpha)$. If the relation $({\objectQuantaloida},{\Qpresheaveg}) \dashv ({\objectQuantaloida},{\Qpresheavef})$ holds --- i.e.,~if $({\objectQuantaloida},{\Qpresheaveg})$ is a left 
adjoint contravariant \donotbreakdash{$\Q$}presheaf, then the triple $\mu=({\Qpresheavef},{\objectQuantaloida},{\Qpresheaveg})$ is called a \emph{presingleton} of the \donotbreakdash{$\Q$}
category $(X,\alpha)$. Referring to (\ref{n2.2}) the following relation holds:
\stepcounter{num}
\begin{equation}\label{n3.0}  {\Qpresheavef}(y)=\tbigwedge_{x\in X_0} ({\Qpresheaveg}(x) \searrow \alpha(x,y)),\qquad y \in X_0.
   \end{equation}
Moreover, the type of a presingleton $({\Qpresheavef},{\objectQuantaloida},{\Qpresheaveg})$ is given by ${\objectQuantaloida}\in \Q_0$. This means that if $\widehat{X}_0$ is the set of all presingletons of the \donotbreakdash{$\Q$}category $(X,\alpha)$,    
   then the type map $|\phantom{x}|\colon\widehat{X}_0 \to \Q_0$ is given by $|({\Qpresheavef},{\objectQuantaloida},{\Qpresheaveg})|={\objectQuantaloida}$ and allows us to define the pair $(\widehat{X}_0,|\phantom{x}|)$, which is denoted by $\widehat{X}$.
  
Since the adjunction $({\objectQuantaloida},{\Qpresheaveg})\dashv ({\objectQuantaloida},{\Qpresheavef})$ is equivalent to the following two conditions:
\begin{enumerate}[label=\textup{(P\arabic*)},topsep=3pt,itemsep=0pt,start=2
]
\item\label{P2} $ 1_{\objectQuantaloida}\le \tbigvee_{x\in X_0} ({\Qpresheavef}(x)\circ {\Qpresheaveg}(x))$,
\item\label{P3} ${\Qpresheaveg}(x_1) \circ {\Qpresheavef}(x_2) \le \alpha(x_1,x_2),\qquad x_1,x_2\in X_0$,
\end{enumerate}
there exists an hom-arrow-assignment $\widehat{\alpha}\colon\widehat{X}_0\times \widehat{X}_0 \to \morph(\Q)$ defined by:
\stepcounter{num}
\begin{equation}\label{NN}
\widehat{\alpha}(({\Qpresheavef}_1,{\objectQuantaloida}_1,{\Qpresheaveg}_1),({\Qpresheavef}_2,{\objectQuantaloida}_2,{\Qpresheaveg}_2))=\tbigvee_{x\in X_0} {\Qpresheavef}_1(x) \circ {\Qpresheaveg}_2(x).
\end{equation}
Thus $(\widehat{X},\widehat{\alpha})$ is a \donotbreakdash{$\Q$}category and is called   the \emph{presingleton space} of $(X,\alpha)$.

Furthermore, there exists a distributor $\Xi\colon (X,\alpha)\circlearrow(\widehat{X},\widehat{\alpha})$ defined by:
\stepcounter{num}
\begin{equation}\label{n3.1AA} \Xi(\mu,x)={\Qpresheavef}(x),\qquad \mu \in \widehat{X}_0,\, x\in X_0.
\end{equation}
This distributor is an isomorphism between both \donotbreakdash{$\Q$}categories $(X,\alpha)$ and $(\widehat{X},\widehat{\alpha})$ in the sense of  \donotbreakdash{$\Q$}\cat{Set} (see Section~\ref{section5} infra).

\begin{example}\label{example3.4} Let $(X,\alpha)$ be a \donotbreakdash{$\Q$}category. 
Then each $x\in X_0$ determines a presingleton $\widetilde{x}=({\Qpresheavef}_x,|x|,{\Qpresheaveg}_x)$,
where ${\Qpresheavef}_x(y)=\alpha(x, y)$ and ${\Qpresheaveg}_x(y)=\alpha(y,x)$ for all $y\in X_0$.
It follows immediately from the axiom \ref{Q1}, \ref{P1}, \ref{C2} and \ref{C3} that  the following relations hold for each $x,x_1,x_2\in X_0$ and $\mu=({\Qpresheavef},{\objectQuantaloida},{\Qpresheaveg})\in \widehat{X}_0$:
\stepcounter{num}
\begin{equation} \label{n3.3}
\widehat{\alpha}(\widetilde{x},\mu)={\Qpresheaveg}(x),\quad \widehat{\alpha}(\mu,\widetilde{x})=
 {\Qpresheavef}(x)\quad\text{and}\quad \widehat{\alpha}(\widetilde{x_1},\widetilde{x_2})=\alpha(x_1,x_2).
 \end{equation}
\end{example} 
\section{Cocomplete  $\Q$-enriched categories}\label{section4}
A first application of \donotbreakdash{$\Q$}pre\-sheaves is the concept of cocompleteness (resp.\ completeness) of \donotbreakdash{$\Q$}categories. As a preparation for the separation axiom and the cocompleteness, we introduce the following lemmas.

 \begin{lemma} \label{lem4.1AA} Let $(X,\alpha)$ be a \donotbreakdash{$\Q$}category and $x,y\in X_0$.  
 \begin{enumerate}[label=\textup{(\alph*)},topsep=3pt,itemsep=0pt
 ,leftmargin=20pt,labelwidth=10pt,itemindent=5pt,labelsep=5pt,topsep=5pt,itemsep=3pt
 ]
\item The following properties are equivalent\textup:
 \begin{enumerate}[label=\textup{(\roman*)},topsep=0pt,itemsep=0pt
 ]
\item $\alpha(x,x)=\alpha(x,y)$ and $\alpha(y,y)= \alpha(y,x)$.
\item $1_{|x|}\le \alpha(x,y)$ and $1_{|y|}\le \alpha(y,x)$.
\end{enumerate}
\item The following properties are equivalent\textup:
 \begin{enumerate}[label=\textup{(\roman*)},start=3,topsep=0pt,itemsep=0pt
 ]
\item $\alpha(x,x)=\alpha(y,x)$ and $\alpha(y,y)= \alpha(x,y)$.
\item $1_{|x|}\le \alpha(y,x)$ and $1_{|y|}\le \alpha(x,y)$.
\end{enumerate}
\item If $|x|=|y|$, then the  assertion {\rm (i)} is equivalent to {\rm (iii)}
\end{enumerate}
  \end{lemma}

  \begin{proof} From axiom \ref{C3}, it follows that (i) implies (ii). 
  Conversely, assume that (ii) holds. Then, using (\ref{n2.1C4}) and \ref{C2}, we have:
  \[\alpha(x,x) \le \alpha(x,x) \circ \alpha(x,y)=\alpha(x,y)\le\alpha(x,y)\circ \alpha(y,x)  \le \alpha(x,x),
\]
  which implies $\alpha(x,x)= \alpha(x,y)$. Similarly, we obtain $\alpha(y,y)= \alpha(y,x)$. \\
The equivalence in (b) follows analogously to (a), and (c) is a direct corollary of (a) and (b) under the assumption $|x|=|y|$.
  \end{proof}
  
  A \donotbreakdash{$\Q$}category $(X,\alpha)$ is said to be \emph{separated} (or \emph{skeletal}) if the following implication holds for all $x,y \in X_0$:
\[\bigl(|x|=|y|,\quad \alpha(x,x)= \alpha(x,y)\quad \text{and} \quad \alpha(y,y)=\alpha(y,x)\bigr)\implies x=y.
\]
Referring to Lemma~\ref{lem4.1AA}, this condition is equivalent to the implication (cf.\ Example~\ref{example3.4}):
\[\widetilde{x}=\widetilde{y}\implies x=y.\] 
  
\begin{lemma} \label{lem4.0AA} Let $(X,\alpha)$ be a \donotbreakdash{$\Q$}category and let $\eta_{(X,\alpha)}\colon (X,\alpha) \to (P(X,\alpha),\pi)$ be the 
\donotbreakdash{$\Q$}enriched Yoneda embedding, given by
\[\eta_{(X,\alpha)}(x)=(|x|,\alpha(\underline{\phantom{x}},x)), \qquad x\in X_0.\]
Suppose that there exists a functor $\xi\colon (P(X,\alpha),\pi)\to (X,\alpha)$, then the following assertions are equivalent\textup:
 \begin{enumerate}[label=\textup{(\roman*)},topsep=3pt,itemsep=0pt
 ]
\item $\xi$ is left adjoint to $\eta_{(X,\alpha)}$ --- i.e.,~$\alpha(\xi({\objectQuantaloida},{\Qpresheaveg}),x)=\pi(({\objectQuantaloida},{\Qpresheaveg}), \eta_{(X,\alpha)}(x))$ for all $x\in X_0$ and $({\objectQuantaloida},{\Qpresheaveg})\in P(X,\alpha)_0$.
\item $1_{|\xi(\eta_{(X,\alpha)}(x))|}\le \alpha(x,\xi(\eta_{(X,\alpha)}(x)))$ and $1_{|\xi(\eta_{(X,\alpha)}(x))|}\le \alpha(\xi(\eta_{(X,\alpha)}(x)),x)$  for all $x\in X_0$.
\item For all $x\in X_0$ the following relations hold\textup: 
\[\alpha(x,x)=\alpha(x,\xi(\eta_{(X,\alpha)}(x)))\quad\text{and}\quad \alpha(\xi(\eta_{(X,\alpha)}(x)),\xi(\eta_{(X,\alpha)}(x)))=\alpha(\xi(\eta_{(X,\alpha)}(x)),x).\]
\end{enumerate}
\end{lemma}

\begin{proof} Assume that (i) holds. Then for all $x\in X_0$ we have:
\begin{align*}
1_{|\xi(\eta_{(X,\alpha)}(x))|}&\le \alpha(\xi(\eta_{(X,\alpha)}(x)),\xi(\eta_{(X,\alpha)}(x)))=\pi(\eta_{(X,\alpha)}(x), \eta_{(X,\alpha)}(\xi(\eta_{(X,\alpha)}(x))))\\
 &= \alpha(x,\xi(\eta_{(X,\alpha)}(x)))\\
\noalign{\noindent\text{and}}
1_{|\xi(\eta_{(X,\alpha)}(x))|}&=1_{|\eta_{(X,\alpha)}(x)|}\le \pi(\eta_{(X,\alpha)}(x),\eta_{(X,\alpha)}(x))= \alpha(\xi(\eta_{(X,\alpha)}(x)),x).
\end{align*}
Hence (ii) is verified.
\\
 Conversely, assume that (ii) holds. For any $({\objectQuantaloida}, {\Qpresheaveg})\in P(X,\alpha)_0$  and $x\in X_0$ observe:
\begin{align*} g(x)&=\pi(\eta_{(X,\alpha)}(x),({\objectQuantaloida},{\Qpresheaveg}))\le \alpha(\xi(\eta_{(X,\alpha)}(x)),\xi({\objectQuantaloida},{\Qpresheaveg}))\\
&\le \alpha(x,\xi(\eta_{(X,\alpha)}(x)))\circ \alpha(\xi(\eta_{(X,\alpha)}(x)),\xi({\objectQuantaloida},{\Qpresheaveg}))\le \alpha(x,\xi({\objectQuantaloida},{\Qpresheaveg})),
\end{align*}
--- i.e.~$\Qpresheaveg\le \eta_{(X,\alpha)}(\xi({\objectQuantaloida},{\Qpresheaveg}))$. An application of this observation leads to the following relation:
\begin{align*}
\alpha(\xi({\objectQuantaloida},{\Qpresheaveg}),x)&=\pi(\eta_{(X,\alpha)}(\xi({\objectQuantaloida},{\Qpresheaveg})), \eta_{(X,\alpha)}(x))\le \pi(({\objectQuantaloida},{\Qpresheaveg}), \eta_{(X,\alpha)}(x))\le \alpha(\xi({\objectQuantaloida},{\Qpresheaveg}),\xi(\eta_{(X,\alpha)}(x)))\\
&\le
 \alpha(\xi({\objectQuantaloida},{\Qpresheaveg}),\xi(\eta_{(X,\alpha)}(x))) \circ \alpha(\xi(\eta_{(X,\alpha)}(x)),x)\le  \alpha(\xi({\objectQuantaloida},{\Qpresheaveg}),x)).
\end{align*}
Thus, (i) follows.
Finally, since $|x|=|\xi(\eta_{(X,\alpha)}(x))|$ for all $x\in X_0$, the equivalence (ii)$\iff$(iii) follows from Lemma~\ref{lem4.1AA}\,(a).
\end{proof}

 A \donotbreakdash{$\Q$}category $(X,\alpha)$ is said to be \emph{cocomplete} if  the \donotbreakdash{$\Q$}enriched Yoneda embedding 
 \[\eta_{(X,\alpha)}\colon(X,\alpha)\to (P(X,\alpha),\pi)\]
 admits a left adjoint functor $\sup\nolimits_{X}\colon (P(X,\alpha),\pi) \to (X,\alpha)$.
 
If $(X,\alpha)$ is separated, then by Lemma~\ref{lem4.1AA} and Lemma~\ref{lem4.0AA}, cocompleteness of  $(X,\alpha)$ 
is equivalent to the existence of a 
functor $\sup\nolimits_{X}\colon (P(X,\alpha),\pi) \to (X,\alpha)$ such that the following diagram commutes:
\[
\bfig
 \qtriangle(0,0)|alr|/>`>`>/<800,300>[{\small(X,\alpha)}`(P(X,\alpha),\pi)`(X,\alpha);\eta_{(X,\alpha)}`1_{(X,\alpha)}`\sup\nolimits_{X}]
 \efig
\]
Referring to \cite{Stubbe05}, the \donotbreakdash{$\Q$}category $(P(X,\alpha),\pi)$ is separated and cocomplete, and serves as  the free cocompletion of $(X,\alpha)$. 

In fact, the component $\mu_{(X,\alpha)}\colon (P(P(X,\alpha),\pi),\pi)\to (P(X,\alpha),\pi)$
 of the multiplication of the contravariant \donotbreakdash{$\Q$}presheaf monad (cf.\  \cite[Rem.~5.4]{ho14}) is given by:
\[\bigl(\mu_{(X,\alpha)}({\objectQuantaloida},{\QpresheaveG})\bigr)(x)=\tbigvee_{ ({\objectQuantaloidb},{\Qpresheaveg})\in P(X,\alpha)_0} \Qpresheaveg(x) \circ 
\QpresheaveG({\objectQuantaloidb},{\Qpresheaveg}),\qquad ({\objectQuantaloida},{\QpresheaveG})\in P(P(X,\alpha),\pi)_0,\, x\in X_0,
\]
and the endofunctor $\mathds{P}\colon\cat{Cat}(\Q) \to \cat{Cat}(\Q)$ of the contravariant \donotbreakdash{$\Q$}presheaf monad acts on morphisms $\varphi\colon(X,\alpha)\to (Y,\beta)$ of $\cat{Cat}(\Q)$ as follows:
\[\bigl(\mathds{P}(\varphi)({\objectQuantaloida},{\Qpresheaveg})\bigr)(y)=\tbigvee\limits_{x\in X_0} \beta(y,\varphi(x))\circ {\Qpresheaveg}(x), \qquad y\in Y_0,({\objectQuantaloida},{\Qpresheaveg})\in P(X,\alpha)_0. 
\]
This functor $\mu_{(X,\alpha)}$ is left adjoint to the \donotbreakdash{$\Q$}enriched Yoneda embedding $\eta_{P(X,\alpha)}$. 
Moreover, if $(X,\alpha)$ is separated and cocomplete 
with $\sup\nolimits_{X}\colon (P(X,\alpha),\pi) \to (X,\alpha)$ being left adjoint to $\eta_{(X,\alpha)}$, 
then the following diagram is commutative:
\[
\bfig
 \square(0,0)|alra|/>`>`>`>/<1200,300>[(P(P(X,\alpha),\pi),\pi)`(P(X,\alpha),\pi)`(P(X,\alpha),\pi)`(X,\alpha);\mathds{P}(\sup\nolimits_{X})`\mu_{(X,\alpha)}`\sup\nolimits_{X}`\sup\nolimits_{X}]
 \efig\]
In fact, the repeated application of the left adjointness of $\sup\nolimits_{X}$ implies  the following relation for all $({\objectQuantaloida},{\QpresheaveG})\in P(P(X,\alpha))$ and all $x\in X_0$:
\begin{align*} \alpha(\sup\nolimits_{X}(\mathds{P}(\sup\nolimits_{X})({\objectQuantaloida},{\QpresheaveG})),x)&=\tbigwedge_{({\objectQuantaloidb},{\Qpresheaveg})\in P(X,\alpha)_0}\bigl(\QpresheaveG({\objectQuantaloidb},{\Qpresheaveg})\searrow \bigl(\tbigwedge_{z\in X_0} (\alpha(z, \sup\nolimits_{X}({\objectQuantaloidb},{\Qpresheaveg}))\searrow\alpha(z,x))\bigr)\bigr)\\
&=\tbigwedge_{({\objectQuantaloidb},{\Qpresheaveg})\in P(X,\alpha)_0}\bigl(\QpresheaveG({\objectQuantaloidb},{\Qpresheaveg})\searrow \alpha(\sup\nolimits_{X}({\objectQuantaloidb},{\Qpresheaveg}),x)\bigr)\\
&= \tbigwedge_{({\objectQuantaloidb},{\Qpresheaveg})\in P(X,\alpha)_0}\bigl(\QpresheaveG({\objectQuantaloidb},{\Qpresheaveg})\searrow \bigl(\tbigwedge_{z\in X_0} ({\Qpresheaveg}(z)\searrow \alpha(z,x))\bigr)\bigr)\\
&= \alpha(\sup\nolimits_{X}(\mu_{(X,\alpha)}({\objectQuantaloida},{\QpresheaveG})),x).
\end{align*}
Since $(X,\alpha)$ is separated, the relation $\sup\nolimits_{X}(\mathds{P}(\sup\nolimits_{X})({\objectQuantaloida},{\QpresheaveG}))=\sup\nolimits_{X}(\mu_{(X,\alpha)}({\objectQuantaloida},{\QpresheaveG}))$ follows. Hence, separated and cocomplete \donotbreakdash{$\Q$}categories are \emph{algebras} of the contravariant \donotbreakdash{$\Q$}presheaf monad (see again \cite[Rem.~5.4]{ho14}).
Moreover, cocompleteness  and completeness are equivalent concepts (cf.\ \cite{Stubbe05}).
Since, analogously to $(P(X,\alpha),\pi)$, the \donotbreakdash{$\Q$}category $(\Q(X,\alpha),\upsilon)$ of covariant \donotbreakdash{$\Q$}presheaves is separated and complete, it is also cocomplete. However, we now provide a direct proof of this property.

 \begin{proposition}\label{prop4.2BB} The \donotbreakdash{$\Q$}category $(\Q(X,\alpha),\upsilon)$ is separated and cocomplete.
 \end{proposition}
 
  \begin{proof} 
  Clearly, $(\Q(X,\alpha),\upsilon)$ is separated. 
  To show cocompleteness, consider the functor 
  \[\sup\nolimits_{\Q(X,\alpha)}\colon\bigl(P(\Q(X,\alpha),\upsilon),\pi\bigr)\to(\Q(X,\alpha),\upsilon)\] defined for all $(\objectQuantaloida,\QpresheaveG)\in P(\Q(X,\alpha),\upsilon)_0$ by:
  \stepcounter{num}
  \begin{equation}\label{n4.1BB} \sup\nolimits_{\Q(X,\alpha)}(\objectQuantaloida,\QpresheaveG)(x)=\tbigwedge_{({\objectQuantaloidb},{\Qpresheavef})\in \Q(X,\alpha)_0} \bigl(\QpresheaveG({\objectQuantaloidb},{\Qpresheavef}) \searrow \Qpresheavef(x)\bigr),\qquad x\in X_0.
  \end{equation}
  The previous formula defines a covariant  \donotbreakdash{$\Q$}pre\-sheaf of type $\objectQuantaloida$ on $(X,\alpha)$. 
  Furthermore, if $x\in X_0$ and  $({\objectQuantaloidb},{\Qpresheavef})\in \Q(X,\alpha)_0$, then we observe:
\[\QpresheaveG_1({\objectQuantaloidb},{\Qpresheavef})\circ \pi(({\objectQuantaloida}_1,\QpresheaveG_1),({\objectQuantaloida}_2,\QpresheaveG_2))\circ \bigl(\cat{sup}_X(\objectQuantaloida_2,\QpresheaveG_2)(x)\bigr)\le \Qpresheavef(x),
\]
which implies $\pi(({\objectQuantaloida}_1,\QpresheaveG_1),({\objectQuantaloida}_2,\QpresheaveG_2))\le \upsilon(({\objectQuantaloida}_1,\sup\nolimits_{X}(\objectQuantaloida_1,\QpresheaveG_1)),({\objectQuantaloida}_2,\sup\nolimits_{X}(\objectQuantaloida_2,\QpresheaveG_2)))$.\\
   Hence  $\sup\nolimits_{\Q(X,\alpha)}$ is a functor.
\\
Finally, using the associativity of the composition in $\Q$, we show that $\sup\nolimits_{\Q(X,\alpha)}$ is left adjoint to the \donotbreakdash{$\Q$}enriched Yoneda embedding $\eta_{(\Q(X,\alpha)),\upsilon)}$.    Let $({\objectQuantaloidc},{\Qpresheaveh})$ be another covariant \donotbreakdash{$\Q$}presheaf on $(X,\alpha)$. Then the left-adjointness is expressed by:
\begin{align*}
\tbigwedge_{x\in X_0}\bigl( \cat{sup}_X(\objectQuantaloida,\QpresheaveG)(x)\swarrow \Qpresheaveh(x)\bigr)&=\tbigwedge_{x\in X_0}\bigl(\bigl(\tbigwedge_{({\objectQuantaloidb},{\Qpresheavef})\in\Q(X,\alpha)_0} \bigl(\QpresheaveG({\objectQuantaloidb},{\Qpresheavef}) \searrow \Qpresheavef(x)\bigr)\bigr)\swarrow \Qpresheaveh(x)\bigr)\\
&=\tbigwedge_{({\objectQuantaloidb},{\Qpresheavef})\in \Q(X,\alpha)_0} \bigl(\QpresheaveG({\objectQuantaloidb},{\Qpresheavef}) \searrow\bigl(\tbigwedge_{x\in X_0}( \Qpresheavef(x)\swarrow \Qpresheaveh(x))\bigr)\bigr).
\end{align*}
This completes the proof.
  \end{proof}


  \begin{corollary}\label{lem4.2AA} The \donotbreakdash{$\Q$}category $(\R^{\objectQuantaloidc},\varrho)$, introduced in Remark~\ref{rem3.2}, is separated and cocomplete.  The structure map $\sup\nolimits_{\R^{\objectQuantaloidc}}\colon (P(\R^{\objectQuantaloidc},\varrho),\pi)\to (\R^{\objectQuantaloidc},\varrho)$ has the form\textup:
 \stepcounter{num}
 \begin{equation} \label{n4.2BB}\sup\nolimits_{\R^{\objectQuantaloidc}}(\objectQuantaloida,\Qpresheaveg)=\tbigwedge_{{\arrowQuantaloidf}\in \R^{\objectQuantaloidc}_0} (\Qpresheaveg({\arrowQuantaloidf})\searrow{\arrowQuantaloidf}),\qquad (\objectQuantaloida,\Qpresheaveg) \in  P(\R^{\objectQuantaloidc},\varrho)_0.
 \end{equation}
\end{corollary}

Let $(X,\alpha)$ and $(Y,\beta)$ be separated and cocomplete categories, and let $\sup\nolimits_{X}$ 
and $\sup\nolimits_{Y}$ be the left adjoint functors of their respective \donotbreakdash{$\Q$}enriched Yoneda embeddings. Then a functor $\varphi\colon(X,\alpha) \to (Y,\beta)$ is cocontinuous if and only if the following diagram  commutes:
\[ 
\bfig
 \square(0,0)|alra|/>`>`>`>/<1000,300>[(P(X,\alpha),\pi)`(P(Y,\beta),\pi)`(X,\alpha)`(Y,\beta);\mathds{P}(\varphi)`\sup\nolimits_{X}`\sup\nolimits_{Y}`\varphi]
 \efig\]
where $\mathds{P}$ is the endofunctor of the contravariant \donotbreakdash{$\Q$}presheaf monad (see above).

An example of a cocontinuous functor is given in the next lemma.

\begin{lemma} \label{lem4.3AA} Let $(X,\alpha)$ a separated and cocomplete  \donotbreakdash{$\Q$}category. Then for each $x \in X_0$ with $|x|=\objectQuantaloidc$, the contravariant \donotbreakdash{$\Q$}presheaf $\alpha(\underline{\phantom{x}},x)\colon (X,\alpha) \to (\R^{\objectQuantaloidc},\varrho)$, viewed as a functor, is cocontinuous.
\end{lemma}

\begin{proof} Let $\sup\nolimits_{X}$ and $\sup\nolimits_{\R^{\objectQuantaloidc}}$ be the left adjoint functors of 
the respective \donotbreakdash{$\Q$}enriched Yoneda embeddings.
For all $(\objectQuantaloida,\Qpresheaveg)\in P(X,\alpha)_0$ we compute:
\begin{align*}
\sup\nolimits_{\R^{\objectQuantaloidc}}(\mathds{P}(\alpha(\underline{\phantom{x}},x))(\objectQuantaloida,\Qpresheaveg))&
=\tbigwedge_{\lambda\in \R^{\objectQuantaloidc}_0}\tbigwedge_{y\in X_0} \bigl((\varrho(\lambda, \alpha(y,x))\circ \Qpresheaveg(y))\searrow\lambda\bigr)\\
&=\tbigwedge_{y\in X_0} \bigl(\Qpresheaveg(y)\searrow\bigl(\tbigwedge_{\lambda\in \R^{\objectQuantaloidc}_0}(\varrho(\lambda, \alpha(y,x))\searrow \lambda)\bigr)\bigr)\\
&=
\tbigwedge_{y\in X_0} (\Qpresheaveg(y)\searrow\alpha(y,x))= \alpha(\sup\nolimits_{X}(\objectQuantaloida,\Qpresheaveg),x).\end{align*}
Hence, the cocontinuity of $\alpha(\underline{\phantom{x}},x)$ is verified.
\end{proof}
 
\begin{comment} The previous lemma is a generalization of the well known fact that in complete lattices the characteristic map of the principal down-set $\newdownarrow x=\set{y\in X\mid y\le x}$ is join-preserving.
\end{comment}
  
Let $\cat{Cat}_{sc}(\Q)$ be the category of separated and cocomplete \donotbreakdash{$\Q$}categories, with cocontinuous functors as morphisms. We claim that the \donotbreakdash{$\Q$}category $(\Q,\tau)$, where $\tau(\objectQuantaloida,\objectQuantaloidb)=\tbigvee \hom (\objectQuantaloidb,\objectQuantaloida)$, is the terminal object of $\cat{Cat}_{sc}(\Q)$. 
The type function of $(\Q,\tau)$ is the 
identity, so $(\Q,\tau)$ is separated. Further, for every contravariant \donotbreakdash{$\Q$}presheaf  $(\objectQuantaloida,\Qpresheaveg)$ on $(\Q,\tau)$, the relation
  $\tau(\objectQuantaloida,\objectQuantaloidb)= \tbigwedge_{\objectQuantaloidc \in \Q_0} (\Qpresheaveg(\objectQuantaloidc) \searrow \tau(\objectQuantaloidc,\objectQuantaloidb))$ holds. This 
  means that the type function of
 $P(\Q,\tau)$ is left adjoint to the \donotbreakdash{$\Q$}enriched Yoneda embedding of $(\Q,\tau)$, so $(\Q,\tau)$ is cocomplete and therefore an object of $\cat{Cat}_{sc}(\Q)$. 
Finally,  the commutative diagram
\[
\bfig
 \square(0,0)|alra|/>`>`>`>/<1000,300>[(P(X,\alpha),\pi)`(P(\Q,\tau),\pi)`(X,\alpha)`(\Q,\tau);\mathds{P}(|\phantom{x}|)`\sup\nolimits_{X}`|\phantom{x}|=\sup\nolimits_{X}`|\phantom{x}|]
 \efig\]
points out that the type function of every separated and cocomplete \donotbreakdash{$\Q$}category $(X,\alpha)$ viewed as a functor $|\phantom{x}|\colon (X,\alpha) \to (\Q,\tau)$ is cocontinuous.\\
 Conclusion: $(\Q,\tau)$ is the terminal object in $\cat{Cat}_{sc}(\Q)$. 

\begin{proposition}[Points in $\cat{Cat}_{sc}(\Q)$] \label{prop4.5AA} Let $(X,\alpha)$ be a separated and cocomplete \donotbreakdash{$\Q$}ca\-tegory. A functor $\varphi\colon(\Q,\tau) \to (X,\alpha)$ is 
cocontinuous if and only if $\varphi$ satisfies the following additional property for all $x\in X_0$ and for all $\objectQuantaloidc\in \Q_0$\textup:
\stepcounter{num}
\begin{equation} \label{n4.1AA} \alpha\bigl(\varphi(\objectQuantaloidc),x\bigr)= \tau\bigl(\objectQuantaloidc,|x|\bigr).
\end{equation}
\end{proposition} 

\begin{proof}
Let $\bot_{\objectQuantaloidb\objectQuantaloidc}$ be the universal lower bound of $\hom (\objectQuantaloidb,\objectQuantaloidc)$. Since the composition in $\Q$ is join-preserving, we first notice that the relation
\stepcounter{num} 
\begin{equation} \label{n4.4CC} \bot_{\objectQuantaloidb\objectQuantaloidc}\circ u\le v
\end{equation}
holds for all $u\in \hom (s,q)$ and $v\in \hom (s,r)$, and we define a contravariant \donotbreakdash{$\Q$}presheaf $(\objectQuantaloidb, \Qpresheaveg)$ on $(\Q,\tau)$ by:
\[ \Qpresheaveg(\objectQuantaloidc)=\bot_{\objectQuantaloidb\objectQuantaloidc}, \qquad \objectQuantaloidc\in \Q_0.
\]
If $\varphi$ is cocontinuous, then we apply (\ref{n4.4CC}) and obtain:
\begin{align*}\alpha\bigl( \varphi( \objectQuantaloidb),x\bigr)&=\alpha\bigl(\sup\nolimits_{X}(\mathds{P}(\varphi)(\objectQuantaloidb, \Qpresheaveg)),x\bigr) =\tbigwedge_{\objectQuantaloidc\in \Q_0} \bigl(\Qpresheaveg(\objectQuantaloidc)\searrow\bigl(\tbigwedge_{z\in X_0} \alpha\bigl(z,\varphi(\objectQuantaloidc)\bigr)\searrow \alpha(z,x)\bigr)\bigr)\\
&= \tbigwedge_{\objectQuantaloidc\in \Q_0}\bot_{\objectQuantaloidb\objectQuantaloidc} \searrow \alpha\bigl(\varphi(\objectQuantaloidc),x\bigr)= \tau(\objectQuantaloidb,|x|).
\end{align*}
Conversely, if condition (\ref{n4.1AA}) holds, then for any contravariant \donotbreakdash{$\Q$}presheaf $(\objectQuantaloida, \Qpresheaveg)$ on $(\Q,\tau)$ and  any $x\in X_0$ we compute:
\begin{align*}\alpha\bigl(\sup\nolimits_{X}(\mathds{P}(\varphi)(\objectQuantaloida, \Qpresheaveg)),x\bigr)&=
\tbigwedge_{\objectQuantaloidc\in \Q_0} \bigl(\Qpresheaveg(\objectQuantaloidc)\searrow \alpha(\varphi(\objectQuantaloidc),x)\bigr)  =\tbigwedge_{\objectQuantaloidc\in \Q_0} \bigl(\Qpresheaveg(\objectQuantaloidc)\searrow \tau(\objectQuantaloidc,|x|)\bigr)\\
&=\tau(\objectQuantaloida,|x|)= \alpha(\varphi(\objectQuantaloida),x).
\end{align*}
 Since $(X,\alpha)$ is separated, it follows that $\varphi\bigl(\sup\nolimits_{\Q}(\objectQuantaloida, \Qpresheaveg)\big)=\varphi(\objectQuantaloida)=\sup\nolimits_{X}(\mathds{P}(\varphi)(\objectQuantaloida, \Qpresheaveg))$ --- i.e.,~$\varphi$ is cocontinuous.
 \end{proof}

\begin{remark} \label{rem4.6AA} (1) (Weak subobject classifier) To prepare the definition of the  weak subobject classifier (cf.\ \cite{ho25}), we introduce some  additional terminology. 
Let $\objectQuantaloidc$ be an object of $\Q$. Then $\tau(\objectQuantaloidc,\objectQuantaloidc)$ is the top element of $\hom (\objectQuantaloidc,\objectQuantaloidc)$. 
A morphism \newmorphism{16}{\objectQuantaloidc}{\arrowQuantaloidf}{\objectQuantaloidb} is called \emph{right-sided} if $\arrowQuantaloidf \circ \tau(\objectQuantaloidc,\objectQuantaloidc)\le \arrowQuantaloidf$. 
If $\tau(\objectQuantaloidc,\objectQuantaloidc)=1_{\objectQuantaloidc}$, then every morphism with domain $\objectQuantaloidc$ is right-sided.
Further, let $\mathds{R}(\R_0^{\objectQuantaloidc})$ denote the set of all right-sided morphisms ${\arrowQuantaloidf}$ with $\dom (\arrowQuantaloidf)=\objectQuantaloidc$. Then $\mathds{R}(\R_0^{\objectQuantaloidc})\subseteq\R_0^{\objectQuantaloidc}$. On $\mathds{R}(\R_0^{\objectQuantaloidc})$, we consider the structure of $(\R_0^{\objectQuantaloidc},\varrho)$    
restricted to $\mathds{R}(\R_0^{\objectQuantaloidc})$ 
 --- i.e.,~the type map of $\mathds{R}(\R^{\objectQuantaloidc})$ is the 
restriction of the type map of $\R^{\objectQuantaloidc}$, the hom-arrow-assignment is given by
\[ \varrho({\arrowQuantaloidf},{\arrowQuantaloidg})={\arrowQuantaloidf}\swarrow {\arrowQuantaloidg}, \qquad {\arrowQuantaloidf},{\arrowQuantaloidg}\in \mathds{R}(\R_0^{\objectQuantaloidc}),\]
and the structure map $\sup\nolimits_{\mathds{R}(\R^{\objectQuantaloidc})}\colon (P(\mathds{R}(\R^{\objectQuantaloidc}),\varrho),\pi)\to( \mathds{R}(\R^{\objectQuantaloidc}),\varrho)$ is defined by
\[\sup\nolimits_{\mathds{R}(\R^{\objectQuantaloidc})}(\objectQuantaloida,\Qpresheaveg)=\tbigwedge_{{\arrowQuantaloidf}\in \mathds{R}(\R_0^{\objectQuantaloidc})} (\Qpresheaveg({\arrowQuantaloidf})\searrow{\arrowQuantaloidf}),\qquad(\objectQuantaloida,\Qpresheaveg) \in P(\mathds{R}(\R^{\objectQuantaloidc}),\varrho)_0.
\]
In fact, $\sup\nolimits_{\mathds{R}(\R^{\objectQuantaloidc})}(\objectQuantaloida,\Qpresheaveg)$ is right-sided, since
\[ \sup\nolimits_{\R^{\objectQuantaloidc}}(\objectQuantaloida,\Qpresheaveg)\circ \tau(\objectQuantaloidc,\objectQuantaloidc)\le \tbigwedge_{{\arrowQuantaloidf}\in \R^{\objectQuantaloidc}_0}\bigl(\Qpresheaveg({\arrowQuantaloidf})\searrow({\arrowQuantaloidf}\circ \tau(\objectQuantaloidc,\objectQuantaloidc))\bigr)\le\sup\nolimits_{\mathds{R}(\R^{\objectQuantaloidc})}(\objectQuantaloida,\Qpresheaveg).
\]
As we will see below, the separated and cocomplete \donotbreakdash{$\Q$}category $(\mathds{R}(\R^{\objectQuantaloidc}),\varrho)$ will play the role of a weak subobject classifier in $\cat{Cat}_{sc}(\Q)$.
\\[2pt]
(2) (Arrow true) The map $\tau(\underline{\phantom{x}},\objectQuantaloidc)$ defines a functor $\tau(\underline{\phantom{x}},\objectQuantaloidc)\colon(\Q,\tau)\to(\mathds{R}(\R^{\objectQuantaloidc}),\varrho)$, and the relation
  ${\tau(\objectQuantaloidb,|{\arrowQuantaloidf}|) \circ {\arrowQuantaloidf}} \le \tau(\objectQuantaloidb,\objectQuantaloidc)$ holds for all $\objectQuantaloidb \in \Q_0$ and ${\arrowQuantaloidf}\in \mathds{R}(\R^{\objectQuantaloidc}_0)$. Hence, 
 \[\varrho(\tau(\objectQuantaloidb,\objectQuantaloidc),{\arrowQuantaloidf})=\tau(\objectQuantaloidb,|{\arrowQuantaloidf}|)\] 
 follows, and Proposition~\ref{prop4.5AA} ensures that the functor $\tau(\underline{\phantom{x}},\objectQuantaloidc)$
   is cocontinuous. This observation motivates the following terminology: The arrow $\tau(\underline{\phantom{x}},\objectQuantaloidc)$ is called the \emph{arrow true of type $\objectQuantaloidc$}, and is denoted by $\true_{\objectQuantaloidc}$. 
\end{remark}


\begin{theorem}\label{them4.7AA} Let $\Q$ be a \donotbreakdash{$\objectQuantaloidc$}stable small quantaloid, $(X,\alpha)$ be a separated and cocomplete \donotbreakdash{$\Q$}category, and let $\varphi\colon(\Q,\tau) \to (X,\alpha)$ be a cocontinuous functor. Then the following diagram is a pullback square in $\cat{Cat}_{sc}(\Q)$\textup:
\stepcounter{num}
\begin{equation} \label{n4.2BB}
\bfig
 \square(0,0)|alra|/<-`>`>`<-/<1000,300>[(X,\alpha)`(\Q,\tau)`{(\mathds{R}(\R^{\objectQuantaloidc}),\varrho)}`(\Q,\tau);\varphi`\alpha(\underline{\phantom{x}},\varphi(\objectQuantaloidc))`1_{\Q_0}`\true_{\objectQuantaloidc}]
 \efig
\end{equation}
\end{theorem}

\begin{proof} Referring to equation (\ref{n2.1C4}), we first observe:
\[\alpha(\underline{\,\,\,\,},\varphi(\objectQuantaloidc))=\alpha(\underline{\,\,\,\,},\varphi(\objectQuantaloidc)\circ \alpha(\varphi(\objectQuantaloidc),\varphi(\objectQuantaloidc))=\alpha(\underline{\,\,\,\,},\varphi(\objectQuantaloidc))\circ \tau(\objectQuantaloidc,\objectQuantaloidc)),\]
which implies that the range of $\alpha(\underline{\,\,\,\,},\varphi(\objectQuantaloidc))$ is contained in $\mathds{R}(\R_0^{\objectQuantaloidc})$. Consequently, the diagram (\ref{n4.2BB}) is commutative. To verify the universal property, let $(Y,\beta)$ be another separated and cocomplete \donotbreakdash{$\Q$}category, and let $\psi\colon(Y,\beta)\to (X,\alpha)$ be a cocontinuous functor satisfying the condition:
\[\alpha(\psi(y),\varphi(\objectQuantaloidc))= \true_{\objectQuantaloidc}(|y|)=\tau(|y|,\objectQuantaloidc),\qquad y\in Y_0.
  \]
By Proposition~\ref{prop4.5AA}, we also have $\alpha(\varphi(\objectQuantaloidc),\psi(y))= \tau(\objectQuantaloidc,|\psi(y)|)=\tau(\objectQuantaloidc,|y|)$. Now we consider the compositions:
\[
 \tau(|y|,\objectQuantaloidc)\circ \tau(\objectQuantaloidc,|y|)= \alpha(\psi(y),\varphi(\objectQuantaloidc))\circ \alpha(\varphi(\objectQuantaloidc), \varphi(|y|))\le \alpha(\psi(y),\varphi(|y|))
 \]
  \noindent and
\[
\tau(|y|,\objectQuantaloidc)\circ \tau(\objectQuantaloidc,|y|)=\alpha(\varphi(|y|),\varphi(\objectQuantaloidc))\circ\alpha(\varphi(\objectQuantaloidc),\psi(y))\le \alpha(\varphi(|y|),\psi(y)).
  \]
By the \donotbreakdash{$\objectQuantaloidc$}stability of $\Q$, it follows that $1_{|y|}\le  \alpha(\psi(y),\varphi(|y|))$ and $1_{|y|}\le  \alpha(\varphi(|y|,\psi(y))$.
 Using the separation axiom (cf.\ Lemma~\ref{lem4.1AA}\,(a)), we conclude $\psi(y)=\varphi(|y|)$, which means the following diagram is commutative:
\[
\bfig
 \dtriangle(0,0)|ara|/>`>`<-/<600,300>[(Y,\beta)`(X,\alpha)`(\Q,\tau);\psi`|\phantom{x}|`\psi]
 \efig
\]
Thus, the universal property of diagram (\ref{n4.2BB}) is verified.
\end{proof}

\begin{theorem}\label{them4.8AA} Let $\Q$ be a quantaloid, and let $(X,\alpha)$ and $(Y,\beta)$ be  separated and cocomplete \donotbreakdash{$\Q$}categories. 
Further, suppose that $\varphi\colon (Y,\beta)\to (X,\alpha)$ and $\chi\colon(X,\alpha)\to (\mathds{R}(\R^{\objectQuantaloidc}),\varrho)$ are  cocontinuous functors that turn the following diagram into a pullback square\textup:
\stepcounter{num}
\begin{equation} \label{n4.3AA}
\bfig
 \square(0,0)|alra|/<-`>`>`<-/<1000,300>[(X,\alpha)`(Y,\beta)`{(\mathds{R}(\R^{\objectQuantaloidc}),\varrho)}`(\Q,\tau);\varphi`\chi`|\phantom{x}|`\true_{\objectQuantaloidc}]
 \efig
\end{equation} 
Then $\chi$ has the form\textup:
\stepcounter{num}
\begin{equation}\label{n4.4AA}
\chi(x)= \tbigvee_{y\in Y_0} \alpha(x,\varphi(y))\circ \tau(|y|,\objectQuantaloidc), \qquad x\in X_0.
\end{equation}
\end{theorem}

\begin{proof} Let $x\in X_0$. From the commutativity of (\ref{n4.3AA}) and axiom \ref{P1}, we obtain: 
\[\tbigvee_{y\in Y_0} \alpha(x,\varphi(y))\circ \tau(|y|,\objectQuantaloidc)\le \tbigvee_{y\in Y_0} \alpha(x,\varphi(y))\circ \chi(\varphi(y))\le \chi(x).
\] 
(1) Now we choose $x\in X_0$ such that $\chi(x)=\tau(|x|,\objectQuantaloidc)$, and show that $x$ is contained in the image of $\varphi$. 
To do this, consider the discrete \donotbreakdash{$\Q$}category $(\{x\},\delta)$ with $\delta(x,x)=1_{|x|}$. Let $(P(\{x\},\delta),\pi)$ be the free cocompletion of $(\{x\},\delta)$, and let $\psi\colon (P(\{x\},\delta),\pi)\to (X,\alpha)$ be the extension of the embedding of $x$ into $X_0$ to a cocontinuous functor --- i.e.,~$\psi$ is given by (cf.\ \cite[Chap.~1.~4.12]{Manes}):
\[\psi(\objectQuantaloida,\Qpresheaveg) =\sup\nolimits_{X}(\alpha(\underline{\phantom{x}},x)\circ \Qpresheaveg(x)),\qquad (\objectQuantaloida,\Qpresheaveg) \in 
P(\{x\},\delta)_0.\]
Since $\chi$ is cocontinuous, we compute:
\begin{align*} \chi(\psi(\objectQuantaloida,\Qpresheaveg))&=\sup\nolimits_{\mathds{R}(\R^{\objectQuantaloidc})}\bigl(\tbigvee_{z\in X_0} \varrho(\underline{\phantom{x}},\chi(z))\circ \alpha(z,x)\circ\Qpresheaveg(x)\bigr)=\sup\nolimits_{\mathds{R}(\R^{\objectQuantaloidc})}(\varrho(\underline{\phantom{x}},\chi(x))\circ\Qpresheaveg(x))\\
&= \Qpresheaveg( x)\searrow\bigl(\tbigwedge_{{\arrowQuantaloidf} \in \mathds{R}(\R_0^{\objectQuantaloidc})} ({\arrowQuantaloidf} \swarrow \chi(x))\searrow {\arrowQuantaloidf}\bigr)=\Qpresheaveg(x)\searrow \chi(x)\\
&=\Qpresheaveg(x)\searrow\tau(|x|,r)= \tau(\objectQuantaloida,\objectQuantaloidc).
\end{align*}
Hence we conclude $\chi\circ\psi= \true_{\objectQuantaloidc}\circ |\phantom{x}|$, and by the universal property of the pullback square (\ref{n4.3AA}), there exists a (unique) cocontinuous functor $\vartheta\colon (P(\{x\},\delta),\pi)\to (Y,\beta)$ such that $\psi= \varphi\circ \vartheta$. In the special case of the contravariant \donotbreakdash{$\Q$}presheaf $(|x|, \Qpresheaveg_0)$ with $\Qpresheaveg_0(x)=\alpha(x,x)$ we obtain:
\[\psi(|x|,\Qpresheaveg_0)= \sup\nolimits_{X}(\alpha(\underline{\phantom{x}},x)\circ\alpha(x,x))=x=\varphi(\vartheta(|x|,\Qpresheaveg_0)),\]
--- i.e.,~$x$ is contained in the image $\varphi(Y)$.
\\[2pt]
(2) Finally, for every $x\in X_0$ consider the contravariant \donotbreakdash{$\Q$}presheaf $(\objectQuantaloidc,\Qpresheaveg_x)$ on $(X,\alpha)$ defined by:
\[\Qpresheaveg_x(z)= \alpha(z,x)\circ \chi(x), \qquad z\in X_0.
\]
 Then we observe $\bigl(\mathds{P}(\chi)(\objectQuantaloidc,\Qpresheaveg_x)\bigr)({\arrowQuantaloidf})=\varrho({\arrowQuantaloidf}, \chi(x))\circ \chi(x)$ for all ${\arrowQuantaloidf} \in \mathds{R}(\R^{\objectQuantaloidc}_0)$. Since $\chi$ is cocontinuous and right-sided, we obtain:
\begin{align*}\chi(\sup\nolimits_{X}(\objectQuantaloidc,\Qpresheaveg_x))&=\sup\nolimits_{\mathds{R}(\R^{\objectQuantaloidc})}(\mathds{P}(\chi)(\objectQuantaloidc,\Qpresheaveg_x))=\chi(x) \searrow \bigl(\tbigwedge_{{\arrowQuantaloidf} \in \mathds{R}(\R_0^{\objectQuantaloidc})}(\varrho({\arrowQuantaloidf}, \chi(x))\searrow {\arrowQuantaloidf})\bigr) \\
&=\chi(x)\searrow \chi(x)= \tau(\objectQuantaloidc,\objectQuantaloidc).
\end{align*}
Now, from part (1), we conclude that $\sup\nolimits_{X}(\objectQuantaloidc,\Qpresheaveg_x)$ is contained in the image of $\varphi$ --- i.e.,~there exists an element $y_x\in Y_0$ such that $|y_x|=\objectQuantaloidc$ and $\sup\nolimits_{X}(\objectQuantaloidc,\Qpresheaveg_x)=\varphi(y_x)$. Hence, by property \ref{P1}, we obtain for all $z\in X_0$:
\begin{align*}\chi(z)&=\tbigvee_{x\in X_0} \alpha(z,x)\circ \chi(x)\le \tbigvee_{x\in X_0}\alpha(z, \sup\nolimits_{X}(\objectQuantaloidc,\Qpresheaveg_x))\\
&\le\tbigvee_{x\in X_0}\alpha(z,\varphi(y_x))\circ \tau(\objectQuantaloidc,\objectQuantaloidc)\le \tbigvee_{y\in Y_0} \alpha(z,\varphi(y)) \circ \tau(|y|,r).\end{align*}
Thus the assertion is verified.
\end{proof}

We can summarize the results of Theorem~\ref{them4.7AA} and Theorem~\ref{them4.8AA} as follows: 
 Let $\Q$ be a \donotbreakdash{$\objectQuantaloidc$}stable, small quantaloid.  
Then the pair $((\mathds{R}(\R^{\objectQuantaloidc}),\varrho), \true_{\objectQuantaloidc})$ 
satisfies the 
axioms (WS1) and (WS2) of a weak subobject classifier, as defined in \cite[Def.~4.1\,(b)]{ho25}. Therefore, we refer to
 \[((\mathds{R}(\R^{\objectQuantaloidc}),\varrho), \true_{\objectQuantaloidc})\] 
as the \emph{weak subobject classifier of type $\objectQuantaloidc$} in the category $\cat{Cat}_{sc}(\Q)$. In this context we recall that a weak subobject classifier is unique up to an isomorphism. 

With regard to axiom (WS2) (cf.\ Theorem~\ref{them4.8AA}) a monomorphism $\varphi\colon(Y,\beta) \to (X,\alpha)$ is called \donotbreakdash{$\objectQuantaloidc$}classifiable if the diagram (\ref{n4.3AA}) forms a pullback square in the category $\cat{Cat}_{sc}$. In this case, the cocontinuous functor $\chi$ in diagram (\ref{n4.3AA}) is uniquely determined by $\varphi$ and is called the \emph{characteristic morphism of $\varphi$}.  Accordingly, we denote it by  $\chi_{\varphi}$.

Finally, axiom (WS1) (cf.\ Theorem~\ref{them4.7AA}) ensures that global points in $\cat{Cat}_{sc}(\Q)$ are always \donotbreakdash{$\objectQuantaloidc$}classifiable. 
\begin{remark}\label{rem4.9AA} (1) Let $\Q$ be a quantaloid with a single object $\objectQuantaloide$ --- i.e.,~$\Q_0=\{\objectQuantaloide\}$ and $\hom (\objectQuantaloide,\objectQuantaloide)=(\alg{{Q}},\ast,1_{\objectQuantaloide})$ is a unital quantale.
Then $\Q$ is trivially \donotbreakdash{$\objectQuantaloide$}stable. The underlying set 
$\mathds{R}(\R_0^{\objectQuantaloidc})$ of the weak subobject classifier in $\cat{Cat}_{sc}(\Q)$ coincides with the subquantale $\mathds{R}(\alg{{Q}})$ of all right-sided elements of the quantale $\alg{{Q}}$.
 Since every separated and cocomplete \donotbreakdash{$\Q$}category is a right \donotbreakdash{$\alg{{Q}}$}module in $\cat{Sup}$ (cf.\ \cite[Sec.~3.3.3]{EGHK}), the underlying order of $(\mathds{R}(\R^{\objectQuantaloidc}),\varrho)$ is the dual order of $\mathds{R}(\alg{{Q}})$, and 
the right action on $\mathds{R}(\R^{\objectQuantaloidc})$ is given by (cf.\ \cite[Proof (a) of Thm.~3.3.22]{EGHK})
\[ {\arrowQuantaloidf}\boxdot \varkappa= \sup\nolimits_{\mathds{R}(\R_0^{\objectQuantaloidc})}(\varrho(\underline{\phantom{x}}, {\arrowQuantaloidf})\ast \varkappa)= \varkappa \searrow {\arrowQuantaloidf},\qquad {\arrowQuantaloidf}\in \mathds{R}(\R_0^{\objectQuantaloidc}),\,\varkappa\in \alg{{Q}}=\hom (\objectQuantaloide,\objectQuantaloide).\]
Thus, we recover the weak subobject classifier of right \donotbreakdash{$\alg{{Q}}$}modules in $\cat{Sup}$ 
(cf.\ \cite[Exam.~3\,(a), Thm.~4.2]{ho25}).
\\[2pt]
(2) Let $\Q$ be an \donotbreakdash{$\objectQuantaloidc$}stable and arbitrary quantaloid. Since    separated and cocomplete \donotbreakdash{$\Q$}categories are right \donotbreakdash{$\Q$}modules   
(cf.\ \cite{Rosenthal96,Stubbe06} and \cite[Def.~5.5 and Rem.~5.6]{ho14}), it follows from (1)  that 
Theorems~\ref{them4.7AA} and \ref{them4.8AA} provide a significant generalization of \cite[Thm.~4.2]{ho25}.
 \\[2pt]
(3)  
Let $\Q_{\boldsymbol{2}}$ be the quantaloid induced by the quantization  $\alg{{Q}}_{\boldsymbol{2}}$ of $\boldsymbol{2}$ (cf.\ Example~\ref{exam1.8}). 
Unlike in case (1), the set of objects of $\Q_{\boldsymbol{2}}$ consists of three distinct elements: $\bot$, $b$ and $\top$. 
The quantaloid $\Q_{\boldsymbol{2}}$ is both \donotbreakdash{$b$}stable and \donotbreakdash{$\top$}stable. Since $1_b=\tau(b,b)$ and $1_{\top}=\tau(\top,\top)$, the weak subobject classifiers of types $b$ and $\top$ have the form:
\[({\R}^{b},\varrho)= \bigl(\hom (b,b) \sqcup \hom (b,\top),\varrho\bigr)\quad \text{and}\quad ({\R}^{\top},\varrho)= \bigl(\hom (\top,\top) \sqcup \hom (\top,b),\varrho\bigr),\]
  where $\sqcup$ denotes the disjoint union. The right actions $\boxdot_{\objectQuantaloida}$ in separated and cocomplete \donotbreakdash{$\Q$}categories $(X,\alpha)$ are given by:
 \[x\boxdot_{\objectQuantaloida} \varkappa=\sup\nolimits_{X}(\alpha(\underline{\phantom{x}},x)\circ \varkappa), \qquad x\in X_0,\,\objectQuantaloida\in \Q_0,\,\varkappa\in \hom (\objectQuantaloida,|x|).\]
 It is therefore straightforward to construct the \donotbreakdash{$\Q_{\boldsymbol{2}}$}enriched functors $\Q_{\boldsymbol{2}}^{op} \to\cat{Sup}$ which express the module character of the related weak subobject classifiers. 
\end{remark} 

Since both $\true_{\objectQuantaloidc}$ and the type arrow $|\phantom{x}|\colon (X,\alpha)\to (\Q,\tau)$ of a \donotbreakdash{$\Q$}category are cocontinuous, 
their composition $\true_{\objectQuantaloidc}\circ |\phantom{x}|$ is also cocontinuous. This composition defines the largest contravariant  
 \donotbreakdash{$\Q$}presheaf of type $\objectQuantaloidc$ on $(X,\alpha)$. 
 In particular, $\true_{\objectQuantaloidc}(|x|)$ is right-sided for all $x\in X_0$. 
It is not difficult to show that diagram (\ref{n4.3AA}) is a pullback square
 if and only if $\varphi\colon(Y,\beta)\to (X,\alpha)$ is the equalizer of $(X,\alpha)\two^\chi_{\true_{\objectQuantaloidc}\circ |\phantom{x}|} \bigl(\mathds{R}(\R^{\objectQuantaloidc}),\varrho\bigr)
$. 
Hence, \donotbreakdash{$\objectQuantaloidc$}classifiable 
subobjects of $\cat{Cat}_{sc}(\Q)$ are always regular.

Let $(\mathds{R}(\R^{\objectQuantaloidc}),\varrho)\times(\mathds{R}(\R^{\objectQuantaloidc}),\varrho)$ be the binary product of $(\mathds{R}(\R^{\objectQuantaloidc}),\varrho)$ in $\cat{Cat}_{sc}(\Q)$. Its underlying set and the hom-arrow-assignment are given by:
\begin{align*}&S=\set{(\lambda_1,\lambda_2)\in \mathds{R}(\R^{\objectQuantaloidc}_0)\times  \mathds{R}(\R^{\objectQuantaloidc}_0)\mid |\lambda_1|=|\lambda_2|} 
\\ 
&(\varrho\times \varrho)((\lambda_1,\lambda_2),(\overline{\lambda}_1,\overline{\lambda}_2))=(\lambda_1\swarrow \overline{\lambda}_1)\wedge (\lambda_2\swarrow \overline{\lambda}_2),\qquad (\lambda_1,\lambda_2), \, (\overline{\lambda}_1,\overline{\lambda}_2)\in S.
\end{align*}
Let $\pi_i$ (for $i\in\set{1,2}$) be the respective projections. 
Then the left adjoint functor of the \donotbreakdash{$\Q$}enriched Yoneda embedding of the product is
 determined by (cf.\ \cite[Chap.~2, 1.11~Prop.]{Manes}):
\[ \sup\nolimits_{\mathds{R}(\R^{\objectQuantaloidc})\times\mathds{R}(R^{\objectQuantaloidc})}(\objectQuantaloida,\Qpresheaveg)=\bigl(\sup\nolimits_{\mathds{R}(\R^{\objectQuantaloidc})}\bigl(\mathds{P}(\pi_1)(\objectQuantaloida,\Qpresheaveg)\bigr),\sup\nolimits_{\mathds{R}(\R^{\objectQuantaloidc})}\bigl(\mathds{P}(\pi_2)(\objectQuantaloida,\Qpresheaveg)\bigr)\bigr).
  \]
Let $\langle \true_{\objectQuantaloidc},\true_{\objectQuantaloidc}\rangle$ be the global point $(\Q,\tau) \to (\mathds{R}(\R^{\objectQuantaloidc}),\varrho)\times(\mathds{R}(R^{\objectQuantaloidc}),\varrho)$, which is cocontinuous. Its corresponding characteristic morphism
$(\mathds{R}(\R^{\objectQuantaloidc}),\varrho)\times(\mathds{R}(\R^{\objectQuantaloidc}),\varrho)\to^{\chi_{\wedge}} (\mathds{R}(\R^{\objectQuantaloidc}),\varrho)$
has the form (cf.\ (\ref{n4.4AA})):
\[ \chi_{\wedge}(\lambda_1,\lambda_2)=\tbigvee_{b \in \Q_0} \bigl( (\lambda_1\swarrow \tau(b,\objectQuantaloidc))\wedge(\lambda_2\swarrow \tau(b,\objectQuantaloidc))\bigr)\circ \tau(b,\objectQuantaloidc)=\lambda_1 \wedge \lambda_2.\]

 In analogy with topos theory (cf. \cite{ho25}) the binary intersection of \donotbreakdash{$\objectQuantaloidc$}classifiable subobject is again \donotbreakdash{$\objectQuantaloidc$}classifiable. Since this construction can naturally be extended to multiple intersections, it follows that every subobject in $\cat{Cat}_{sc}(\Q)$ admits an \donotbreakdash{$\objectQuantaloidc$}classifiable hull.
\section{The category $\Q$-\cat{Set}}
\label{section5}
The category \donotbreakdash{$\Q$}\cat{Set} has \donotbreakdash{$\Q$}categories as objects and \emph{left adjoint distributors} as morphisms.

\begin{remarks}\label{remark4.1}
 (1) The \donotbreakdash{$\Q$}category $(\Q,\tau)$, constructed in Example~\ref{exam2.1}, 
 serves as the terminal object in \donotbreakdash{$\Q$}\cat{Set}.
  Indeed, for every \donotbreakdash{$\Q$}category $(X,\alpha)$, the type map $|{\phantom{x}}|\colon {X_0\to \Q_0}$ induces a left adjoint distributor $\Phi\colon (X,\alpha)\circlearrow (\Q,\tau)$ defined by $\Phi({\objectQuantaloida},x)= \tau({\objectQuantaloida},|x|)$ for all $({\objectQuantaloida},x)\in {\Q_0\times X_0}$. To verify the uniqueness of $\Phi$, consider another left adjoint distributor $\Phi^{\prime}\colon (X,\alpha)\circlearrow(\Q,\tau)$ with the corresponding right adjoint distributor $\Psi^{\prime}$.
 By axiom \ref{D1} we have $\Phi^{\prime}({\objectQuantaloida},x)\in \hom (|x|,{\objectQuantaloida})$ and $\Psi^{\prime}(x,\objectQuantaloida)\in 
 \hom ({\objectQuantaloida},|x|)$, which implies $\Phi^{\prime}({\objectQuantaloida},x)\le \tau({\objectQuantaloida},|x|)$ and $\Psi^{\prime}(x,\objectQuantaloida)\le \tau(|x|,{\objectQuantaloida})$ for all $({\objectQuantaloida},x)\in \Q_0\times X$. Hence, by Proposition~\ref{prop2.1} it follows that $\Phi^{\prime}=\Phi$.
\\[2pt]
 (2) Let ${\Phi\colon (X,\alpha)\circlearrow(Y,\beta)}$ and $\Psi\colon(Y,\beta)\circlearrow(X,\alpha)$ be a pair of adjoint distributors --- i.e.,~$\Phi \dashv \Psi$, and let $x\in X_0$.
Then the triple $\mu_x=(\sigma_x, |x|, \tau_x)$ is a presingleton of $(Y,\beta)$, where
$\sigma_x(y)=\Psi(x,y)\text{ and } \tau_x(y)=\Phi(y,x)$ for each $y\in Y_0$.
Moreover, it follows immediately that
\stepcounter{num}
\begin{equation} \label{newn4.1}
\alpha(x_1,x_2)\le \widehat{\beta}(\mu_{x_1},\mu_{x_2}),\qquad x_1,x_2\in X_0.
\end{equation}
\end{remarks}

\begin{theorem} \label{them4.2} A left adjoint distributor $\Phi\colon (X,\alpha)\circlearrow(Y,\beta)
$ is an epimorphism in \donotbreakdash{$\Q$}\cat{Set} if and only if its right adjoint distributor $\Psi$ satisfies the condition
\stepcounter{num}
\begin{equation} \label{n4.1}
\beta(y,y)\le\tbigvee_{x\in X_0} \Phi(y,x)\circ \Psi(x,y),\qquad y \in Y.
\end{equation}
\end{theorem}

\begin{comment} Since axioms \ref{D2} and \ref{D3} hold, we infer from (\ref{n2.1C4}) that the condition (\ref{n4.1}) is equivalent to the following property:
\begin{equation} \label{n4.1'}\tag{$5.2^{\prime}$}
\beta(y_1,y_2)= \tbigvee_{x\in X_0} \Phi(y_1,x)\circ \Psi(x,y_2)= (\Phi\otimes \Psi)(y_1,y_2),\qquad y_1,y_2\in Y.
\end{equation}
\end{comment}

\begin{proof}  (Necessity): Let  $\Phi\colon (X,\alpha)\circlearrow(Y,\beta)
$ be an epimorphism in \donotbreakdash{$\Q$}\cat{Set}. 
We consider the \donotbreakdash{$\Q$}category $(\Q(\widehat{Y},\widehat{\beta}),\upsilon)$ of covariant \donotbreakdash{$\Q$}pre\-sheaves 
on $(\widehat{Y},\widehat{\beta})$ and use the notation of Remarks~\ref{remark4.1}\,(2).  
  Then for each presingleton $\mu=({\Qpresheavef},{\objectQuantaloida},{\Qpresheaveg})$ of $(Y,\beta)$, we define the covariant \donotbreakdash{$\Q$}pre\-sheaves
$\widehat{\Qpresheavef}_\mu$ and $\widehat{\Qpresheaveg}_\mu$ on $(\widehat{Y},\widehat{\beta})$ by
\[\widehat{\Qpresheavef}_\mu(\mu^\prime)=\widehat{\beta}(\mu,\mu^\prime)\quad\text{and}\quad\widehat{\Qpresheaveg}_\mu(\mu^\prime)=\tbigvee_{x\in X_0} \widehat{\beta}(\mu, \mu_x) \circ \widehat{\beta} (\mu_x,\mu^\prime),\qquad \mu^\prime \in \widehat{Y}_0,\]
and observe that for each $\mu\in \widehat{Y}_0$ and $x\in  X_0$ the following properties hold: 
\stepcounter{num}
\begin{align}\label{n4.2}
&\widehat{\Qpresheaveg}_{\mu}(\mu_x)=\widehat{\beta}(\mu,\mu_x)=\widehat{\Qpresheavef}_\mu(\mu_x), \quad\widehat{\Qpresheaveg}_{\mu_x}(\mu)=\widehat{\beta}(\mu_x,\mu)=\widehat{\Qpresheavef}_{\mu_x}(\mu), \\
 &1_{|\mu_x|} \le \widehat{\Qpresheavef}_{\mu_x}(\mu_x)= \widehat{\Qpresheaveg}_{\mu_x}(\mu_x).\notag
\end{align}
Using the isomorphism $\Xi$ in (\ref{n3.1AA}), we obtain:
\stepcounter{num}
\begin{align} 
 \label{n4.2AA} 
(\Xi \otimes \Phi)(\mu,x)&
=\tbigvee_{y\in Y_0} {\Qpresheavef}(y)\circ \tau_x(y)
=\widehat{\beta}(\mu,\mu_x)=\widehat{\Qpresheavef}_\mu(\mu_x)=\widehat{\Qpresheaveg}_{\mu}(\mu_x).
\end{align}
Now, define the left adjoint distributors $\Gamma,\Theta\colon(\widehat{Y},\widehat{\beta}) \circlearrow (\Q(\widehat{Y},\widehat{\beta}),\upsilon)$ by
\[\Gamma(({\objectQuantaloidb},{\Qpresheaveg}),\mu)=\upsilon \bigl(({\objectQuantaloidb},{\Qpresheaveg}),(|\mu|,\widehat{\Qpresheavef}_\mu)\bigr)\quad \text{and}\quad \Theta(({\objectQuantaloidb},{\Qpresheaveg}),\mu)=\upsilon(({\objectQuantaloidb},{\Qpresheaveg}), (|\mu|,\widehat{\Qpresheaveg}_\mu)).
\]
From equations (\ref{n4.2}) and (\ref{n4.2AA}) we deduce:
\begin{align*}
 (\Gamma\otimes \Xi \otimes \Phi)(({\objectQuantaloidb},{\Qpresheaveg}),x)&
 =\tbigvee_{\mu\in \widehat{Y}_0}\upsilon(({\objectQuantaloidb},{\Qpresheaveg}), (|\mu|, \widehat{\Qpresheavef}_\mu))\circ \widehat{\Qpresheavef}_\mu(\mu_x)={\Qpresheaveg}(\mu_x)\\
 &=\upsilon(({\objectQuantaloidb},{\Qpresheaveg}),  (|\mu_{x}|,\widehat{\Qpresheavef}_{\mu_x})=\upsilon(({\objectQuantaloidb},{\Qpresheaveg}),  (|\mu_{x}|,\widehat{\Qpresheaveg}_{\mu_x})\\
 &=\tbigvee_{\mu\in \widehat{Y}_0}\upsilon(({\objectQuantaloidb},{\Qpresheaveg}), (|\mu|, \widehat{\Qpresheaveg}_\mu))\circ \widehat{\Qpresheaveg}_\mu(\mu_x)
 =(\Theta\otimes \Xi \otimes \Phi)(({\objectQuantaloidb},{\Qpresheaveg}),x).
 \end{align*}
We have shown $\Gamma\otimes \Xi \otimes \Phi=\Theta\otimes \Xi \otimes \Phi$.
Since $\Phi$ is an epimorphism in \donotbreakdash{$\Q$}\cat{Set}, it follows that $\Gamma\otimes \Xi=\Theta \otimes \Xi$. And because $\Xi$ is an isomorphism, we conclude $\Gamma=\Theta$. 
Now, using the fact that the \donotbreakdash{$\Q$}category $(\Q(\widehat{Y},\widehat{\beta}),\upsilon)$ is separated (or skeletal), 
we deduce that $\widehat{\Qpresheavef}_\mu=\widehat{\Qpresheaveg}_\mu$ for all $\mu\in \widehat{Y}_0$. 
Referring to Example~\ref{example3.4} we apply (\ref{n3.3}) and compute
\begin{align*} \beta(y,y)&=\widehat{\beta}(\widetilde{y},\widetilde{y})=\widehat{\Qpresheavef}_{\widetilde{y}}(\widetilde{y})=\widehat{\Qpresheaveg}_{\widetilde{y}}(\widetilde{y})= \tbigvee_{x\in X_0} \widehat{\beta}(\widetilde{y},\mu_x)\circ \widehat{\beta}(\mu_x,\widetilde{y})\\
&=\tbigvee_{x\in X_0} \tau_x(y)\circ \sigma_x(y)=\tbigvee_{x\in X_0} \Phi(y,x)\circ \Psi(x,y),
\end{align*}
 for each $y\in Y_0$ --- i.e.,~condition (\ref{n4.1}) is satisfied.
\\[2pt]
\noindent
(Sufficiency):
Let  $\Gamma_i\colon (Y,\beta)\circlearrow(Z,\gamma)$ be left adjoint distributors (for $i\in\set{1,2}$), and suppose $\Gamma_1 \otimes \Phi=\Gamma_2 \otimes \Phi$. Then, we use (\ref{n2.2CC}), \ref{D2}, \ref{D3} and (\ref{n4.1}), and compute: 
\begin{align*}\Gamma_1(z,y)&= \Gamma_1(z,y)\circ \beta(y,y)\le\tbigvee_{x\in X_0} \Gamma_1(z,y)\circ \Phi(y,x)\circ \Psi(x,y)\\
&\le\tbigvee_{x\in X_0} (\Gamma_1\otimes \Phi)(z,x)\circ \Psi(x,y)=\tbigvee_{x\in X_0} (\Gamma_2\otimes \Phi)(z,x)\circ \Psi(x,y)\\
&
\le\tbigvee_{y^\prime\in Y_0} \Gamma_2(z,y^\prime)\circ \beta(y^\prime,y)\le\Gamma_2(z,y)
\end{align*}
for each $y\in Y$ and $z\in Z$. By symmetry, we also obtain $\Gamma_2(z,y)\le\Gamma_1(z,y)$, and hence $\Gamma_1=\Gamma_2$. Therefore $\Phi$ is an epimorphism in \donotbreakdash{$\Q$}\cat{Set}.
\end{proof}
  
\begin{comment} The proof of Theorem~\ref{them4.2} is fundamentally different from the proof of \cite[Prop.~3.23]{PuDexue} by Q. Pu and D. Zhang, which applies to the special case of commutative, unital, and divisible quantales. It is also distinct from the proof in the context of frame-valued sets (\cite[Prop.~2.8.7]{Borceux3}).
\end{comment}


\begin{theorem}\label{them4.3} A left adjoint distributor $\Phi\colon (X,\alpha)\circlearrow(Y,\beta)
$ is an extremal monomorphism in \donotbreakdash{$\Q$}\cat{Set} if and only if its right adjoint distributor $\Psi$ satisfies the condition\textup:
\stepcounter{num}
\begin{equation} \label{n4.3}
\Psi(x_1,y)\circ \Phi(y,x_2)\le \alpha (x_1,x_2),\qquad x_1,x_2 \in X_0,\,y\in Y_0 
\end{equation}
\end{theorem}

\begin{comment} Referring to \ref{D2}, \ref{D3} and (\ref{n2.2CC}), condition (\ref{n4.3}) is equivalent to:
\begin{equation} \label{n4.5'}\tag{$5.5^{\prime}$}
\alpha (x_1,x_2)=\tbigvee_{y\in Y_0} \Psi(x_1,y)\circ \Phi(y,x_2)=(\Psi\otimes \Phi)(x_1,x_2),\qquad x_1,x_2 \in X_0.
\end{equation}
\end{comment}
 
\begin{proof} (Necessity): Let $\Phi\colon (X,\alpha)\circlearrow(Y,\beta)
$ 
be an extremal monomorphism in \donotbreakdash{$\Q$}\cat{Set}. Referring to Remark~\ref{remark4.1}\,(2) define:
\[Z_0:=\bigset{\mu\in \widehat{Y}_0\mid\widehat{\beta}(\mu,\mu)=\tbigvee_{x\in X} \widehat{\beta}(\mu, \mu_x)\circ \widehat{\beta}(\mu_x,\mu)}.
   \]
 Now consider the restriction of $\widehat{\beta}$ to $Z_0\times Z_0$, and define the following distributors:
\begin{equation*}\Xi\colon (X,\alpha)\circlearrow(Z,\widehat{\beta}),\ \Upsilon\colon (Z,\widehat{\beta})\circlearrow(X,\alpha), \ \Gamma\colon (Y,\beta)\circlearrow(Z,\widehat{\beta}),\ \Theta\colon(Z,\widehat{\beta})\circlearrow(Y,\beta)\end{equation*}
\begin{equation*}\Xi(\mu,x)=\widehat{\beta}(\mu,\mu_x),\quad \Upsilon(x,\mu)=\widehat{\beta}(\mu_x,\mu), \quad \Gamma(\mu,y)=\widehat{\beta}(\mu,\widetilde{y}), \quad \Theta(y,\mu)=\widehat{\beta}(\widetilde{y},\mu).
\end{equation*}
It is straightforward to verify that:
\begin{equation*}\Xi\dashv \Upsilon\quad \text{and} \quad \Theta\dashv \Gamma .
\end{equation*}

We observe that for each $x\in X_0$ and $y\in Y_0$:
\begin{equation*} (\Theta \otimes \Xi)(y,x)
=\tbigvee_{\mu \in Z_0} \widehat{\beta}(\widetilde{y},\mu)\circ \widehat{\beta}(\mu,\mu_x)= \widehat{\beta}(\widetilde{y},\mu_x)=\tau_x(y)=\Phi(y,x).
\end{equation*}
Thus we have: $\Phi=\Theta \otimes \Xi$. Analogously, we verify $\Psi=\Upsilon \otimes \Gamma$. Now, by the definition of $Z_0$ and Theorem~\ref{them4.2}, the left adjoint distributor $\Xi$ is an 
epimorphism. 
Since $\Phi$ is an extremal monomorphism, it follows that $\Xi$ must be an isomorphism. Therefore, for all $x_1,x_2\in X_0$ and $y\in Y_0$ we compute:
\begin{align*}\tbigvee\limits_{y\in Y_0}\Psi(x_1,y)\circ \Phi(y,x_2)&= \widehat{\beta}(\mu_{x_1},\mu_{x_2})=\tbigvee_{\mu \in Z_0} \widehat{\beta}(\mu_{x_1},\mu)\circ \widehat{\beta}(\mu, \mu_{x_2})\\
&=\tbigvee_{\mu \in Z_0} \Upsilon(x_1,\mu)\circ \Xi(\mu,x_2)=
\alpha(x_1,x_2),
 \end{align*}
which verifies condition (\ref{n4.5'}).
\\[2pt]
(Sufficiency): Assume a decomposition  $ \Theta \otimes \Xi= \Phi$ and $\Upsilon \otimes \Gamma = \Psi$ with:
\begin{equation*}\Xi\colon (X,\alpha)\circlearrow(Z,\gamma),\ \Upsilon\colon (Z,\gamma)\circlearrow(X,\alpha), \ \Gamma\colon (Y,\beta)\circlearrow(Z,\gamma),\ \Theta\colon(Z,\gamma)\circlearrow(Y,\beta),\end{equation*}
and adjunctions $\Xi\dashv \Upsilon$ and $\Theta\dashv \Gamma$.\\
We assume that $\Xi$ is an epimorphism and aim to show that $\Xi$ is an isomorphism.
 By Theorem~\ref{them4.2}, the relation $\gamma(z,z)\le \tbigvee_{x\in X_0}\Xi(z,x) \circ \Upsilon(x,z)$ holds
 for all $z\in Z_0$,
  which is equivalent to $\gamma(z_1,z_2)=\tbigvee_{x \in X_0} \Xi(z_1,x) \circ \Upsilon(x,z_2)$ 
 (see Comment below of Theorem~\ref{them4.2}). Thus, it suffices to verify: 
\stepcounter{num}
\begin{equation}\label{n4.6}
\alpha(x_1,x_2)=\tbigvee_{z\in Z_0} \Upsilon(x_1,z)\circ \Xi(z,x_2),\qquad x_1,x_2\in X_0.
\end{equation}
For all $x_1,x_2\in X_0$, we compute:
\begin{align*} \alpha(x_1,x_2)&\le\tbigvee_{z\in Z_0} \Upsilon (x_1,z) \circ \Xi(z,x_2) =\tbigvee_{z\in Z_0}\Upsilon (x_1,z) \circ \gamma(z,z)\circ \Xi(z,x_2) \\
&\le \tbigvee_{z\in Z_0,\, y\in Y_0} \Upsilon(x_1,z)\circ \Gamma(z,y) \circ \Theta(y,z)\circ  \Xi(z,x_2)\\
&\le\tbigvee_{ y\in Y_0}(\Upsilon\otimes \Gamma)(x_1,y)\circ (\Theta\otimes \Xi)(y,x_2)=\tbigvee_{y\in Y_0} \Psi(x_1,y) \circ \Phi(y,x_2).
\end{align*}
Now, applying condition (\ref{n4.3}):
\begin{equation*}\alpha(x_1,x_2)\le\tbigvee_{z\in Z_0} \Upsilon (x_1,z) \circ \Xi(z,x_2)\le \tbigvee_{y\in Y_0} \Psi(x_1,y) \circ \Phi(y,x_2)\le\alpha(x_1,x_2).\end{equation*}
Thus, condition (\ref{n4.6}) is verified, and $\Xi$ is an isomorphism. 
\end{proof}

\begin{corollary}\label{cor4.4} The category \donotbreakdash{$\Q$}$\cat{Set}$ is an  \donotbreakdash{\textup(epi, extremal mono\textup)}category. \end{corollary}

\begin{proof} (a) (Existence of the \donotbreakdash{(epi, extremal mono)}factorization)\\ 
Let $\Phi\colon(X,\alpha) \circlearrow (Y, \beta)$ be a left adjoint distributor, and let $\Psi\colon (Y,\beta) \circlearrow (X,\alpha)$ be its right adjoint distributor. 
Using the notation from Remark~\ref{remark4.1}\,(2) and the first part of the proof of Theorem~\ref{them4.3}, 
we consider the \donotbreakdash{$\Q$}category $(Z,\widehat{\beta})$, where for each $x\in X_0$, the presingleton associated with $x$ is $\mu_x=(\Psi(x,\underline{\phantom{x}}), |x|, \Phi(\underline{\phantom{x}},x))$. Define
\begin{equation*}Z_0=\bigset{\mu\in \widehat{Y}_0\mid \widehat{\beta}(\mu,\mu)= \tbigvee_{x\in X_0} \widehat{\beta}(\mu,\mu_x)\circ \widehat{\beta}(\mu_x,\mu)}.\end{equation*}
By definition of $Z_0$, the left adjoint distributor $\Xi\colon (X,\alpha)\circlearrow (Z,\widehat{\beta})$ defined by $\Xi(\mu,x)=\widehat{\beta}(\mu,\mu_x)$ is an epimorphism (cf.\ Theorem~\ref{them4.2}).    
With regard to $\Gamma$ and $\Theta$ (see the first part of the proof of Theorem~\ref{them4.3}) we now observe:
\begin{equation*}\tbigvee_{y\in Y_0} \Gamma(\mu_1,y)\circ \Theta(y,\mu_2)=\tbigvee_{y\in Y_0} \widehat{\beta}(\mu_1,\widetilde{y})\circ \widehat{\beta}(\widetilde{y},\mu_2)=\widehat{\beta}(\mu_1,\mu_2).
\end{equation*}
Hence, by Theorem~\ref{them4.3}, $\Theta$ is an extremal monomorphism, and since $\Phi=\Theta \otimes \Xi$, the factorization of $\Phi$ into an epimorphism followed by an extremal monomorphism is verified.
\pagebreak
\\[2pt]
(b) (Uniqueness of the \donotbreakdash{(epi, extremal mono)}factorization)\\
 Let $\Xi_i\colon(X,\alpha) \circlearrow (Z_i,\gamma_i)$ and $\Theta_i\colon (Z_i,\gamma_i) \circlearrow (Y,\beta)$ be two \donotbreakdash{(epi, extremal mono)}factorization of $\Phi$ in \donotbreakdash{$\Q$}$\cat{Set}$ (for $i\in\set{1,2}$). Let $\Upsilon_i\colon(Z_i,\gamma_i) \circlearrow (X,\alpha)$ be the right adjoint distributor of $\Xi_i$, and $\Gamma_i\colon (Y,\beta)\circlearrow (Z_i,\gamma_i)$ be the right adjoint distributor of $\Theta_i$. We show that $\Xi_2\otimes \Upsilon_1\colon (Z_1,\gamma_1) \circlearrow (Z_2,\gamma_2)$ is an isomorphism, with the inverse $\Xi_1 \otimes \Upsilon_2$.\\
 Since both $\Xi_1$ and $\Xi_2$ are epimorphisms, we apply condition (\ref{n4.1'}) and compute:
 \begin{align*} \Theta_1 \otimes \Xi_1\otimes\Upsilon_2\otimes \Xi_2 \otimes \Upsilon_1&= \Phi\otimes \Upsilon_2 \otimes \Xi_2 \otimes \Upsilon_1= \Theta_2 \otimes \Xi_2 \otimes \Upsilon_2 \otimes \Xi_2 \otimes \Upsilon_1\\
 &= \Theta_2 \otimes \gamma_2 \otimes\Xi_2 \otimes\Upsilon_1=\Phi\otimes \Upsilon_1=\Theta_1\otimes \Xi_1 \otimes \Upsilon_1=\Theta_1 \otimes \gamma_1.
  \end{align*}
Since $\Theta_1$ is an extremal monomorphism, we apply condition (\ref{n4.5'}) and conclude: 
\begin{equation*}(\Xi_1\otimes\Upsilon_2)\otimes (\Xi_2 \otimes \Upsilon_1)=\gamma_1.\end{equation*}
Analogously we show $(\Xi_2\otimes\Upsilon_1)\otimes (\Xi_1 \otimes \Upsilon_2)=\gamma_2$. Hence, the uniqueness of the
\donotbreakdash{(epi, extremal mono)}factorization is verified.
\end{proof}

\section{The symmetry axiom and the Cauchy completion}\label{section6}
A \donotbreakdash{$\Q$}category $(X,\alpha)$ is said to be \emph{Cauchy complete} if, for every presingleton $\mu=({\Qpresheavef},{\objectQuantaloida},{\Qpresheaveg})$, there exists a unique element $x\in X_0$ such that $\mu=\widetilde{x}$ --- i.e.,~$|x|={\objectQuantaloida}$, ${\Qpresheavef}=\alpha(x,\underline{\phantom{x}})$ and ${\Qpresheaveg}=\alpha(\underline{\phantom{x}},x)$ (cf.\ Example~\ref{example3.4}). In particular, every Cauchy complete \donotbreakdash{$\Q$}category is separated.

If $(\Q,j)$ is an involutive quantaloid, then a \donotbreakdash{$\Q$}category $(X,\alpha)$ is Cauchy complete if and only if its dual \donotbreakdash{$\Q$}category $(X,\alpha^{op})$ is Cauchy complete.
Moreover, every separated and cocomplete \donotbreakdash{$\Q$}category $(X,\alpha)$ is Cauchy complete. 
In fact, if $\sup\nolimits_{X}$ is the left adjoint of the \donotbreakdash{$\Q$}enriched Yoneda embedding $\eta_{(X,\alpha)}$, 
then applying (\ref{n3.0}) yields, for every presingleton $\mu=({\Qpresheavef},{\objectQuantaloida},{\Qpresheaveg})$ of $(X,\alpha)$ and every $y\in X_0$:
\begin{equation*}\alpha(\sup\nolimits_{X}({\objectQuantaloida},{\Qpresheaveg}),y)= \pi(({\objectQuantaloida},{\Qpresheaveg}),\eta_{(X,\alpha)}(y))=\tbigwedge_{x\in X} (\Qpresheaveg(x)\searrow \alpha(x,y))=\Qpresheavef(y).
\end{equation*}
 Hence, it follows that $\mu=\widetilde{\sup\nolimits_{X}({\objectQuantaloida},{\Qpresheaveg})}$. The uniqueness of $\sup\nolimits_{X}({\objectQuantaloida},{\Qpresheaveg})$ is guaranteed by the separation axiom.

\begin{lemma}\label{lem5.1} If $(X,\alpha)$ is a \donotbreakdash{$\Q$}category, then its presingleton space $(\widehat{X},\widehat{\alpha})$ is Cauchy complete.\end{lemma}

\begin{proof} We first show that $(\widehat{X},\widehat{\alpha})$ is separated. 
Let $\mu=({\Qpresheavef},{\objectQuantaloida},{\Qpresheaveg})$ and $\mu^{\prime}=({\Qpresheavef}^{\prime},{\objectQuantaloida},{\Qpresheaveg}^{\prime})$ be presingletons such that $|\mu|={\objectQuantaloida}=|\mu^{\prime}|$ and assume:
\begin{equation*}\widehat{\alpha}(\mu,\mu)=\widehat{\alpha}(\mu,\mu^{\prime})  \quad \text{and}\quad \widehat{\alpha}(\mu^{\prime},\mu^{\prime})=\widehat{\alpha}(\mu^{\prime},\mu).
\end{equation*}
Then we conclude from (\ref{n2.1C4}) and (\ref{n3.3}) that the following relation holds:
\begin{equation*}{\Qpresheaveg}(x)=\widehat{\alpha}(\widetilde{x},\mu)\circ \widehat{\alpha}(\mu,\mu)=\widehat{\alpha}(\widetilde{x},\mu)\circ \widehat{\alpha}(\mu,\mu^{\prime})\le \widehat{\alpha}(\widetilde{x},\mu^{\prime})={\Qpresheaveg}^{\prime}(x).
\end{equation*}
By symmetry, we obtain ${\Qpresheaveg}^{\prime}\le {\Qpresheaveg}$, and hence $\mu=\mu^{\prime}$ by (\ref{n3.0}).
  \\
Now, let $\widehat{\mu}=(\widehat{\Qpresheavef},{\objectQuantaloida},\widehat{\Qpresheaveg})$ be a presingleton of $(\widehat{X},\widehat{\alpha})$. Define a covariant \donotbreakdash{$\Q$}presheaf ${\Qpresheavef}_0$ and a contravariant \donotbreakdash{$\Q$}presheaf ${\Qpresheaveg}_0$ on $(X,\alpha)$ as follows:
\begin{equation*}{\Qpresheavef}_0(x)=\tbigvee_{\mu\in \widehat{X}_0} \widehat{\Qpresheavef}(\mu)\circ {\Qpresheavef}(x)\quad \text{and}\quad {\Qpresheaveg}_0(x)=\tbigvee_{\mu\in \widehat{X}_0} {\Qpresheaveg}(x)\circ \widehat{\Qpresheaveg}(\mu),\qquad x\in X.\end{equation*}
We show that ${\Qpresheaveg}_0$ is left adjoint to ${\Qpresheavef}_0$. In fact, the following relations hold:
\begin{align*}\tbigvee_{x\in X_0} {\Qpresheavef}_0(x)\circ {\Qpresheaveg}_0(x)&\ge\tbigvee_{x\in X_0,\, \mu \in \widehat{X}_0}\widehat{\Qpresheavef}(\mu)\circ {\Qpresheavef}(x)\circ {\Qpresheaveg}(x)\circ \widehat{\Qpresheaveg}(\mu)\ge\tbigvee_{\mu\in \widehat{X}_0} \widehat{\Qpresheavef}(\mu)\circ 1_{|\mu|}\circ \widehat{\Qpresheaveg}(\mu)\ge 1_{{\objectQuantaloida}}\\
{\Qpresheaveg}_0(x_1)\circ {\Qpresheavef}_0(x_2) &= \tbigvee_{\mu,\mu^{\prime}\in \widehat{X}_0} {\Qpresheaveg}(x_1)\circ \widehat{\Qpresheaveg}(\mu)\circ \widehat{\Qpresheavef}(\mu^{\prime})\circ {\Qpresheavef}^{\prime}(x_2)\le \tbigvee_{\mu,\mu^{\prime}\in \widehat{X}_0} {\Qpresheaveg}(x_1)\circ \widehat{\alpha}(\mu,\mu^{\prime})\circ {\Qpresheavef}^{\prime}(x_2)\\
&=  \tbigvee_{\mu,\mu^{\prime}\in \widehat{X}_0} \widehat{\alpha}(\widetilde{x_1},\mu)\circ \widehat{\alpha}(\mu,\mu^{\prime})\circ \widehat{\alpha}(\mu^{\prime}, \widetilde{x_2})\le \alpha(x_1,x_2).
\end{align*}
Hence $\mu_0:=({\Qpresheavef}_0,{\objectQuantaloida},{\Qpresheaveg}_0)$ is a presingleton of $(X,\alpha)$. \\
 Finally, we show that $\widehat{\mu}=\widetilde{\mu_0}$ --- i.e.,~$\widehat{\Qpresheavef}=\widehat{\alpha}(\mu_0,\underline{\phantom{x}})$ and $\widehat{\Qpresheaveg}=\widehat{\alpha}(\underline{\phantom{x}},\mu_0)$. For arbitrary presingletons $\mu=({\Qpresheavef},|\mu|,{\Qpresheaveg})$ and $\mu^{\prime}=({\Qpresheavef}^{\prime},|\mu^{\prime}|,{\Qpresheaveg}^{\prime})$ of $(X,\alpha)$ we compute:
\begin{align*}
\widehat{\alpha}(\mu_0,\mu)&=\tbigvee_{x\in X_0} {\Qpresheavef}_0(x)\circ {\Qpresheaveg}(x)=
\tbigvee_{x\in X_0}\bigl(\tbigvee_{\mu^{\prime}\in \widehat{X}_0}\widehat{\Qpresheavef}(\mu^{\prime})\circ {\Qpresheavef}^{\prime}(x)\bigr)\circ {\Qpresheaveg}(x)=\tbigvee_{\mu^{\prime}\in \widehat{X}_0}\widehat{\Qpresheavef}(\mu^{\prime}) \circ \widehat{\alpha}(\mu^{\prime},\mu)\le \widehat{\Qpresheavef}(\mu).
\end{align*}
Analogously, we show $\widehat{\alpha}(\mu,\mu_0)\le \widehat{\Qpresheaveg}(\mu)$  for all $\mu\in \widehat{X}_0$. Hence Proposition~\ref{prop2.1} implies $\widehat{\Qpresheavef}=\widehat{\alpha}(\mu_0,\underline{\phantom{x}})$ and $\widehat{\Qpresheaveg}=\widehat{\alpha}(\underline{\phantom{x}},\mu_0)$.
\end{proof}

Because of Lemma~\ref{lem5.1}, the presingleton space of a \donotbreakdash{$\Q$}category $(X,\alpha)$ is also referred to as its \emph{Cauchy completion}. 
With regard to the distributor defined in $(\ref{n3.1AA})$, it is easily seen that $(X,\alpha)$ is always isomorphic to its Cauchy completion in the sense of \donotbreakdash{$\Q$}\cat{Set}.
 Consequently, \donotbreakdash{$\Q$}\cat{Set} is equivalent to the subcategory \donotbreakdash{$C\Q$}\cat{Set}, consisting of Cauchy complete \donotbreakdash{$\Q$}categories with functors as morphisms. 
 In fact, the embedding functor $\mathcal{E}\colon \text{$C\Q$-\cat{Set}}\to\text{$\Q$-\cat{Set}}$ is given by (cf.\ \cite[Prop.~7.14]{Stubbe05}):
\begin{equation*}\mathcal{E}(X,\alpha)=(X,\alpha), \quad (X,\alpha)\to^{\varphi} (Y,\beta), \quad \mathcal{E}(\mathcal{\varphi})=\beta(\underline{\,\,\,},\varphi(\underline{\phantom{x}}))\colon (X,\alpha)\circlearrow (Y,\beta).
\end{equation*}
The left adjoint functor $\mathcal{F}\colon\text{$\Q$-\cat{Set}}\to\text{C$\Q$-\cat{Set}}
$ of $\mathcal{E}$ has the form:
\begin{align*}&\mathcal{F}(X,\alpha)=(\widehat{X},\widehat{\alpha}), \quad \Phi\colon(X,\alpha) \circlearrow (Y,\beta),\quad (\widehat{X},\widehat{\alpha}) \to^{\mathcal{F}(\Phi)} (\widehat{Y},\widehat{\beta}),\\
& \mathcal{F}(\Phi)\bigl(\sigma,p,\tau)=(\overline{\sigma},p,\overline{\tau}\bigr),\quad \overline{\sigma}(y)=\tbigvee_{x\in X_0} \sigma(x) \circ \Psi(x,y),\quad \overline{\tau}(y)= \tbigvee_{x\in X_0}\Phi(y,x)\circ \tau(x),
\end{align*}
where $\Psi$ is the right adjoint distributor to $\Phi$. With regard to Corollary~\ref{cor4.4} we come to the conclusion that \donotbreakdash{$C\Q$}\cat{Set} is an \donotbreakdash{(epi, extremal mono)}category.

If $(\Q,j)$ is an \emph{involutive} quantaloid, then the symmetry axiom of a \donotbreakdash{$\Q$}category $(X,\alpha)$ can be expressed as follows:
\begin{enumerate}[label=\textup{(C\arabic*)},start=4,topsep=3pt,itemsep=0pt
]
\item\label{C4} $\alpha(x,y)=j(\alpha(y,x)), \qquad x,y,\in X$. \hfill (Symmetry)
\end{enumerate}
Obviously, $(X,\alpha)$ is symmetric if and only if $(X,\alpha)$ coincides with 
its dual \donotbreakdash{$\Q$}category $(X,\alpha^{op})$ (cf.\ Remark~\ref{rem3.3})
  
On this background we introduce the following terminology. The Cauchy completion preserves the symmetry axiom if the Cauchy completion of a symmetric \donotbreakdash{$\Q$}category is again symmetric. It follows immediately from the hom-arrow-assignment of the presingleton space and the symmetry axiom \ref{C4} that the 
    Cauchy completion preserves the symmetry axiom if and only if every presingleton 
 $({\Qpresheavef},{\objectQuantaloida},{\Qpresheaveg})$ is a \emph{singleton} --- i.e.,~$\Qpresheaveg=j\circ {\Qpresheavef}$ (cf.\ \cite[Def.~6.4]{HK11}). However, there exist simple examples of involutive quantaloids for which the Cauchy completion does not always preserve the symmetry axiom.

\begin{example}\label{exam5.2} Let $G_3=\set{e,a,b}$ be the cyclic group with unity $e$. This group induces a unital and commutative  quantale $\alg{{Q}}_5$ with 5 elements. The Hasse diagram and the multiplication table are given as follows:
 \[
\renewcommand\arraystretch{1.}
\setlength\doublerulesep{0pt}
{\footnotesize
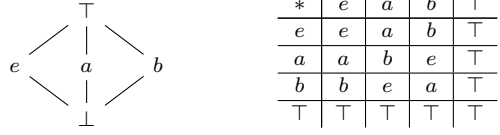
\begin{tikzcd}[column sep=15pt,row sep=9pt]
&\top \arrow[-]{dl}\arrow[-]{d}\arrow[-]{dr}\\
e&a&b\\
&\bot\arrow[-]{ul}\arrow[-]{u}\arrow[-]{ur}&
\end{tikzcd}}
\qquad\qquad
{\footnotesize\begin{tabular}{c||c|c|c|c}
$\ast$ & $e$ & $a$ & $b$ &  $\top$\\
\hline\hline
$e$ & $e$ & $a$ & $b$ & $\top$ \\\hline
$a$ &  $a$ & $b$ & $e$ & $\top$ \\\hline
$b$ &  $b$ & $e$ & $a$ &$\top$ \\\hline
$\top$ &  $\top$ & $\top$ & $\top$ & $\top$ \\
\end{tabular}}
  \]
We view $\alg{{Q}}_5$ as a quantaloid with one object, where the type function is  trivial and the identity acts as the involution  $j$. \\
 Now we consider the discrete \donotbreakdash{$\alg{{Q}}_5$}category on $X=\set{a,b}$ --- i.e.$\alpha(a,b)=\alpha(b,a)=\bot$ and $\alpha(a,a)=\alpha(b,b)=e$. Then we introduce the following pair of 
$\alg{{Q}}_5$-presheaves on $(X,\alpha)$:
\begin{equation*} {\Qpresheavef}(a)=a,\quad {\Qpresheavef}(b)=\bot,\quad {\Qpresheaveg}(a)=b,\quad {\Qpresheaveg}(b)=\bot,\end{equation*}
Since the following relations hold:
\begin{equation*}e=({\Qpresheavef}(a)\ast {\Qpresheaveg}(a)) \vee ({\Qpresheavef}(b)\ast {\Qpresheaveg}(b))\quad \text{and}\quad {\Qpresheaveg}(x)\ast {\Qpresheavef}(y)\le \alpha(x,y), \qquad x,y\in \{a,b\}.   
\end{equation*}
the axioms \ref{P2} and \ref{P3} are satisfied, and consequently $\mu=({\Qpresheavef},|\mu|,{\Qpresheaveg})$ is a presingleton of $(X,\alpha)$. However, since ${\Qpresheavef}\neq {\Qpresheaveg}$, the Cauchy completion of $(X,\alpha)$ is not symmetric.
\end{example} 

Motivated by the previous example we provide a characterization of the property that the Cauchy completion w.r.t.\ a given involutive quantaloid preserves the symmetry axiom. This result extends \cite[Prop.~3.1]{ho11} and addresses a problem posed in \cite{BW82}.

\begin{proposition} \label{prop5.3} Let $(\Q,j)$ be a small involutive quantaloid. Then the following statements are equivalent\textup:
\begin{enumerate}[label=\textup{(\arabic*)},topsep=3pt,itemsep=0pt
,leftmargin=20pt,labelwidth=10pt,itemindent=0pt,labelsep=5pt,topsep=5pt,itemsep=3pt
]
\item The Cauchy completion preserves the symmetry axiom of \donotbreakdash{$\Q$}categories.
\item For every system $(\newmorphism{20}{\objectQuantaloida_i}{\arrowQuantaloidf_i}{\objectQuantaloida_0}
, \newmorphism{20}{\objectQuantaloida_0}{\arrowQuantaloidg_i}{\objectQuantaloida_i}
)_{i\in I}$  of morphisms in $\Q$, satisfying the following properties\textup:
\begin{enumerate}[label=\textup{(\alph*)},topsep=0pt,itemsep=0pt
]
\item $1_{{\objectQuantaloida}_0} \le \tbigvee_{i\in I} {\arrowQuantaloidf}_i \circ {\arrowQuantaloidg}_i$,
\item for all $i_1,i_2 \in I$\textup: \\
$\begin{array}{rl}
{\arrowQuantaloidf}_{i_1}\circ {\arrowQuantaloidg}_{i_1} \circ {\arrowQuantaloidf}_{i_2} \le {\arrowQuantaloidf}_{i_2}, &{\arrowQuantaloidg}_{i_1} \circ {\arrowQuantaloidf}_{i_2} \circ {\arrowQuantaloidg}_{i_2}\le {\arrowQuantaloidg}_{i_1}\\
j({\arrowQuantaloidg}_{i_1}) \circ {\arrowQuantaloidg}_{i_1} \circ {\arrowQuantaloidf}_{i_2}\le j({\arrowQuantaloidg}_{i_2}),&{\arrowQuantaloidg}_{i_1} \circ {\arrowQuantaloidf}_{i_2} \circ j({\arrowQuantaloidf}_{i_2}) \le j({\arrowQuantaloidf}_{i_1}),
\end{array}
$
\end{enumerate}
 the following relations hold\textup:
\begin{equation*}1_{{\objectQuantaloida}_0}\le \tbigvee_{i\in I} {\arrowQuantaloidf}_i \mathbin{\circ_{{\objectQuantaloida}_{i}}} j( {\arrowQuantaloidf}_i) \quad \text{and}\quad 1_{{\objectQuantaloida}_0} \le \tbigvee_{i\in I}j({\arrowQuantaloidg}_i)\mathbin{\circ_{{\objectQuantaloida}_{i}}} {\arrowQuantaloidg}_i.\end{equation*}
\end{enumerate}
\end{proposition}

\begin{proof}  (1)$\Longrightarrow$(2):
Let $I$ be a typed set with $|i|={\objectQuantaloida}_i\in \Q_0$ for each $i\in I$. From the system $(\newmorphism{20}{\objectQuantaloida_i}{\arrowQuantaloidf_i}{\objectQuantaloida_0}
,   \newmorphism{20}{\objectQuantaloida_0}{\arrowQuantaloidg_i}{\objectQuantaloida_i}
)_{i\in I}$, we construct a symmetric \donotbreakdash{$\Q$}category $\mathds{I}=(I,\alpha_I)$ by defining:
\begin{equation*}\alpha_I({i_1},{i_2})= ({\arrowQuantaloidf}_{i_1} \searrow {\arrowQuantaloidf}_{i_2}) \wedge ({\arrowQuantaloidg}_{i_1} \swarrow {\arrowQuantaloidg}_{i_2})\wedge
 (j({\arrowQuantaloidg}_{i_1})\searrow j({\arrowQuantaloidg}_{i_2}))\wedge 
 (j({\arrowQuantaloidf}_{i_1})\swarrow j({\arrowQuantaloidf}_{i_2})),
    \end{equation*}
 for each $i_1,i_2\in I$.
Then property (b) implies: 
\stepcounter{num}
\begin{equation}\label{n5.1} {\arrowQuantaloidg}_{i_1}\circ {\arrowQuantaloidf}_{i_2} \le \alpha_I({i_1},{i_2}),\qquad  {i_1},{i_2}\in I.
\end{equation}
Now define a covariant \donotbreakdash{$\Q$}presheaf ${\Qpresheavef}$ and a contravariant  ${\Qpresheaveg}$ of type ${\objectQuantaloida}_0$ on $\mathds{I}$ by:
\begin{equation*}{\Qpresheavef}(i)={\arrowQuantaloidf}_i\quad\text{and}\quad {\Qpresheaveg}(i)={\arrowQuantaloidg}_i, \qquad i\in I.
\end{equation*}
By property (a) and inequality (\ref{n5.1}), ${\Qpresheavef}$ is right adjoint to ${\Qpresheaveg}$ --- i.e.,~$({\Qpresheavef},{\objectQuantaloida}_0,{\Qpresheaveg})$ is a presingleton of $\mathds{I}$. Since the Cauchy completion preserves the symmetry axiom, it follows that ${\Qpresheaveg}=j\circ{\Qpresheavef}$. Thus, condition (2) follows.\\[2pt]
 (2)$\Longrightarrow$(1): Let $(X,\alpha)$ be a symmetric \donotbreakdash{$\Q$}category with type map $t\colon X_0 \to \Q_0$, and let $({\Qpresheavef},{\objectQuantaloida}_0, {\Qpresheaveg})$ be a presingleton of $(X,\alpha)$. Then, by axioms \ref{P2} and \ref{P3}, the following properties hold for all $x,y\in X_0$:
\begin{equation*}1_{{\objectQuantaloida}_0}\le \tbigvee_{x\in X_0}  {\Qpresheavef}(x) \circ {\Qpresheaveg}(x)\quad\text{and}\quad  {\Qpresheaveg}(x)\circ {\Qpresheavef}(y)\le \alpha(x,y).\end{equation*}

Now, applying the extensionality of contravariant and covariant \donotbreakdash{$\Q$}pre\-sheaves  (cf.\ \ref{P1} and \ref{Q1}) we obtain:
\begin{align*}
&{\Qpresheavef}(x)\circ {\Qpresheaveg}(x)\circ {\Qpresheavef}(y)\le {\Qpresheavef}(x)\circ \alpha(x,y)\le {\Qpresheavef}(y),\\
&  
{\Qpresheaveg}(x) \circ {\Qpresheavef}(y) \circ {\Qpresheaveg}(y)\le \alpha(x,y)\circ {\Qpresheaveg}(y)\le {\Qpresheaveg}(x).
\end{align*}

\noindent Using the symmetry axiom \ref{C4}, and again the extensionality, we derive:
\stepcounter{num}
\begin{align}\label{n5.2}
&{\Qpresheaveg}(x)\circ {\Qpresheavef}(y)\circ j({\Qpresheavef}(y))\le \alpha(x,y) \circ j ( {\Qpresheavef}(y))=j({\Qpresheavef}(y)\circ \alpha(y,x)) \le j({\Qpresheavef}(x))\\
&j({\Qpresheaveg}(x)) \circ {\Qpresheaveg}(x) \circ {\Qpresheavef}(y)\le j({\Qpresheaveg}(x))\circ  \alpha(x,y)= j(\alpha(y,x)\circ {\Qpresheaveg}(x) )\le j({\Qpresheaveg}(y)). \notag
\end{align}
\noindent Thus, the system $(\newmorphism{32}{t(x)}{{\Qpresheavef}(x)}{{\objectQuantaloida}_0},\newmorphism{32}{{\objectQuantaloida}_0}{{\Qpresheaveg}(x)}{t(x)})_{x\in X_0}$ 
satisfies the hypothesis of condition~(2) --- i.e.,~(a) and (b). Therefore, by application of (2) the relations:
\begin{equation*}1_{{\objectQuantaloida}_0}\le \tbigvee_{x\in X_0} {\Qpresheavef}(x)\circ j({\Qpresheavef}(x)) \quad \text{and} \quad 1_{{\objectQuantaloida}_0} \le \tbigvee_{x\in X_0} j({\Qpresheaveg}(x))\circ {\Qpresheaveg}(x)\end{equation*}
follow. Now we refer to (\ref{n5.2}) and obtain for each $x,y\in X_0$:
\begin{align*}& {\Qpresheaveg}(x) \le \tbigvee_{y\in X_0} {\Qpresheaveg}(x)\circ {\Qpresheavef}(y)\circ j({\Qpresheavef}(y))\le  j({\Qpresheavef}(x)),\\
&{\Qpresheavef}(y)\le \tbigvee_{x\in X_0} j({\Qpresheaveg}(x))\circ {\Qpresheaveg}(x)\circ {\Qpresheavef}(y)\le j({\Qpresheaveg}(y)).
\end{align*}
\noindent We conclude that 
${\Qpresheavef}=j\circ {\Qpresheaveg}$ --- i.e.,~the Cauchy completion preserves the symmetry axiom.
\end{proof}

\begin{corollary} \label{corollary5.4}Let $(\Q,j)$ be a small involutive quantaloid such that for every system of morphisms $(\newmorphism{20}{\objectQuantaloida_i}{\arrowQuantaloidf_i}{\objectQuantaloida_0},   \newmorphism{20}{\objectQuantaloida_0}{\arrowQuantaloidg_i}{\objectQuantaloida_i})_{i\in I}$ in $\Q$ satisfying the condition\textup: 
\begin{equation*}1_{{\objectQuantaloida}_0} \le \tbigvee_{i\in I} {\arrowQuantaloidf}_i \circ {\arrowQuantaloidg}_i,
    \end{equation*}
the following additional conditions also hold\textup:
\begin{equation*}1_{{\objectQuantaloida}_0}\le \tbigvee_{i\in I} {\arrowQuantaloidf}_i \circ j( {\arrowQuantaloidf}_i)\quad\text{and}\quad 1_{{\objectQuantaloida}_0} \le \tbigvee_{i\in I}j({\arrowQuantaloidg}_i)\circ {\arrowQuantaloidg}_i.\end{equation*}
Then the Cauchy completion preserves the symmetry axiom of \donotbreakdash{$\Q$}categories.
\end{corollary}

\begin{remark}\label{rem5.5} 
It is easily seen that in the quantaloid $\Q_{\boldsymbol{2}}$ in Example~\ref{exam1.8} the condition $1_{{\objectQuantaloida}_0} \le \tbigvee_{i\in I} {\arrowQuantaloidf}_i \circ {\arrowQuantaloidg}_i$ already implies $1_{{\objectQuantaloida}_0}\le \tbigvee_{i\in I} {\arrowQuantaloidf}_i \circ j( {\arrowQuantaloidf}_i)$ 
and $1_{{\objectQuantaloida}_0} \le \tbigvee_{i\in I}j({\arrowQuantaloidg}_i)\circ {\arrowQuantaloidg}_i$.
 Thus, by Corollary~\ref{corollary5.4}, the Cauchy completion preserves the symmetry axiom of \donotbreakdash{$\Q_{\boldsymbol{2}}$}categories.
 \\
 As an illustration of this situation, consider the discrete \donotbreakdash{$\Q_{\boldsymbol{2}}$}category $(\{\cdot\},\delta)$ with $|\cdot|=b$ and $\delta(\cdot,\cdot)=1_{b}=b$. Its Cauchy completion $(\widehat{\{\cdot\}},\widehat{\delta})$ has three presingletons:
\begin{enumerate}[label=\textup{--},leftmargin=12pt,topsep=3pt,itemsep=0pt
]
\item  $\mu_1=(f_1,\top,g_1)$, with $f_1(\cdot)=a_{\ell}$ and $g_1(\cdot)=a_r$,
\item $\mu_2=(f_2,b,g_2)$, with $f_2(\cdot)=b$ and $g_2(\cdot)=b$,  
\item $\mu_3=(f_3,\bot,g_3)$, with $f_3(\cdot)=\bot$ and $g_3(\cdot)=\bot$.
\end{enumerate}
Since $g_i=j\circ f_i$ for all $i\in\set{1,2,3}$, the Cauchy completion $(\widehat{\{\cdot\}},\widehat{\delta}) $ is symmetric, as expected, and is isomorphic to the terminal object of \donotbreakdash{$\Q_{\boldsymbol{2}}$}\cat{Set}
  (cf.\ Remark~\ref{remark4.1}\,(1)). In fact, the hom-arrow-assignment of the terminal object in $(\Q_{\boldsymbol{2}},\tau)$ (cf.\ Example~\ref{exam2.1}) is:
\begin{equation*}\tau(\top,\top)=\top, \quad \tau(b,b)=b,\quad \tau(\top,b)=a_{\ell}, \quad \tau(b,\top)=a_r,\,\,\, \tau(\bot,\underline{\phantom{x}})=\tau(\underline{\phantom{x}},\bot)=\bot. 
\end{equation*}
\end{remark}

Referring now to Remark~\ref{newremark1.3} we have the following:

\begin{corollary} \label{corollary5.6new} Let $\alg{Q}=(\alg{Q},\ast)$ be a commutative quantale, and let $(D\alg{Q},j)$ be the induced involutive quantaloid as described in Remark~\ref{newremark1.3}. 
Then the Cauchy completion preserves the symmetry axiom of \donotbreakdash{$D\alg{Q}$}categories 
if and only if the following symmetry property holds\textup: 
\begin{equation*}{\elementQa}\le \tbigvee_{i\in I} (\newmorphism{20}{\elementQa_i}{\elementQlambda_i}{\elementQa}
)\mathbin{\circ_{{\elementQa}_i}} (\newmorphism{20}{\elementQa}{\elementQlambda_i}{\elementQa_i})\quad\text{and}\quad{\elementQa}\le \tbigvee_{i\in I}  (\newmorphism{20}{\elementQa_i}{\elementQmu_i}{\elementQa}
)\mathbin{\circ_{{\elementQa}_i}} (\newmorphism{20}{\elementQa}{\elementQmu_i}{\elementQa_i}
)\end{equation*}
for every system $(\newmorphism{20}{\elementQa_i}{\elementQlambda_i}{\elementQa}, \newmorphism{20}{\elementQa}{\elementQmu_i}{\elementQa_i}
)_{i\in I}$  of morphisms in the quantaloid $D\alg{Q}$
satisfying the property ${{\elementQa}\le  \tbigvee_{i\in I} {\elementQlambda}_i \mathbin{\circ_{{\elementQa}_i}} {\elementQmu}_i}$ and the following conditions for all $i_1,i_2 \in I$\textup:
\begin{align*}
&{\elementQlambda}_{i_1} \mathbin{\circ_{{\elementQa}_{i_1}}} {\elementQmu}_{i_1} \mathbin{\circ_{{\elementQa}}} {\elementQlambda}_{i_2} \le {\elementQlambda}_{i_2}, \quad
{\elementQmu}_{i_1} \mathbin{\circ_{{\elementQa}}} {\elementQlambda}_{i_2} \mathbin{\circ_{{\elementQa}_{i_2}}} {\elementQmu}_{i_2}\le {\elementQmu}_{i_1},\\
& (\newmorphism{22}{\elementQa_{i_1}}{\elementQmu_{i_1}}{\elementQa}
) \mathbin{\circ_{{\elementQa}_{i_1}}} (\newmorphism{22}{\elementQa}{\elementQmu_{i_1}}{\elementQa_{i_1}}
) \mathbin{\circ_{{\elementQa}}} (\newmorphism{22}{\elementQa_{i_2}}{\elementQlambda_{i_2}}{\elementQa}
)\le (\newmorphism{22}{\elementQa_{i_2}}{\elementQmu_{i_2}}{\elementQa}
),\\ 
&(\newmorphism{22}{\elementQa}{\elementQmu_{i_1}}{\elementQa_{i_1}}
) \mathbin{\circ_{{\elementQa}}} (\newmorphism{22}{\elementQa_{i_2}}{\elementQlambda_{i_2}}{\elementQa}
) \mathbin{\circ_{{\elementQa}_{i_2}}} (\newmorphism{22}{\elementQa}{\elementQlambda_{i_2}}{\elementQa_{i_2}}
) \le(\newmorphism{22}{\elementQa}{\elementQlambda_{i_1}}{\elementQa_{i_1}}
).
\end{align*}
\end{corollary}

\begin{remark}\label{rem5.7} 
In the case of the involutive quantaloid $(D\alg{Q},j)$ in Example~\ref{exam1.7}, the condition ${\elementQa}\le  
\tbigvee_{i\in I} {\elementQlambda}_i \mathbin{\circ_{{\elementQa}_i}} {\elementQmu}_i
$ already implies the conditions 
\begin{equation*}{\elementQa}\le \tbigvee_{i\in I} (\newmorphism{20}{\elementQa_i}{\elementQlambda_i}{\elementQa}
)\mathbin{\circ_{{\elementQa}_i}} (\newmorphism{20}{\elementQa}{\elementQlambda_i}{\elementQa_i}
)\quad  \text{and} \quad{\elementQa}\le \tbigvee_{i\in I}  (\newmorphism{20}{\elementQa_i}{\elementQmu_i}{\elementQa}
)\mathbin{\circ_{{\elementQa}_i}} (\newmorphism{20}{\elementQa}{\elementQmu_i}{\elementQa_i}
).\end{equation*}

 Therefore, by Corollary~\ref{corollary5.6new}, 
 the Cauchy completion preserves the symmetry axiom of \donotbreakdash{$D\alg{Q}$}categories. 
\end{remark}

\begin{corollary} \label{corollary5.7} Let $\alg{Q}=(\alg{Q},\ast,\top)$ be an integral and commutative quantale satisfying the following additional property\textup:
\stepcounter{num}
\begin{equation}\label{n5.3}  {\elementQa}\ast {\elementQb}\le ({\elementQa}\ast {\elementQa})\vee ({\elementQb}\ast {\elementQb}), \qquad {\elementQa},{\elementQb} \in \alg{Q}.
\end{equation} 
Then the Cauchy completion with respect to the involutive quantaloid $(D\alg{Q},j)$ preserves the symmetry axiom.
\end{corollary}

\begin{proof} For all ${\elementQlambda}, {\elementQmu} \in \hom ({\elementQa},{\elementQa})$ we first refer to (\ref{eqn1.2}) and  apply (\ref{n5.3}): 
\stepcounter{num}
\begin{align}\label{n5.4}  {\elementQlambda}\mathbin{\circ_{\elementQa}} {\elementQmu}&={\elementQa}\ast({\elementQa} \rightarrow{\elementQlambda})\ast({\elementQa}\rightarrow{\elementQmu})\\
\notag &\le {\elementQa}\ast((({\elementQa}\rightarrow{\elementQlambda})\ast({\elementQa}\rightarrow{\elementQlambda}))\vee(({\elementQa}\rightarrow{\elementQmu})\ast({\elementQa}\rightarrow{\elementQmu})))=({\elementQlambda}\mathbin{\circ_{\elementQa}} {\elementQlambda})\vee(\elementQmu\mathbin{\circ_{\elementQa}} {\elementQmu}).
\end{align}
\\
Now, with regard to Corollary~\ref{corollary5.4}, we consider a family $\{({\elementQlambda}_i, {\elementQmu}_i) \mid  i\in I\}$ of pairs $({\elementQlambda}_i, {\elementQmu}_i)$ of morphisms  such that ${\elementQmu}_i\in\hom ({\elementQa},{\elementQa}_i)$ and ${\elementQlambda}_i\in\hom ({\elementQa}_i,{\elementQa})$ for each $i\in I$ and the property ${\elementQa}\le  
\tbigvee_{i\in I} {\elementQlambda}_i \mathbin{\circ_{{\elementQa}_i}} {\elementQmu}_i
$ holds. An application of (\ref{n5.4}) leads the  following relation:
\stepcounter{num}
\begin{equation}\label{n6.5CC}
{\elementQa}={\elementQa}\mathbin{\circ_{\elementQa}}{\elementQa}\le  \tbigvee_{i_1,i_2\in I} ({\elementQlambda}_{i_1} \mathbin{\circ_{{\elementQa}_{i_1}}} {\elementQmu}_{i_1})\mathbin{\circ_{\elementQa}} ({\elementQlambda}_{i_2}  \mathbin{\circ_{{\elementQa}_{i_2}}} {\elementQmu}_{i_2})\le  \tbigvee_{i\in I} ({\elementQlambda}_{i} \mathbin{\circ_{{\elementQa}_i}} {\elementQmu}_{i})\mathbin{\circ_{\elementQa}} ({\elementQlambda}_{i} \mathbin{\circ_{{\elementQa}_i}} {\elementQmu}_{i}).
\end{equation}
Using the commutativity and integrality of $\alg{Q}$, we obtain $\newmorphism{20}{\elementQa}{\elementQlambda_i}{\elementQa_i}\in \hom (\elementQa,\elementQa_i)$ and observe  
\begin{equation*}({\elementQlambda}_{i} \mathbin{\circ_{{\elementQa}_i}} {\elementQmu}_{i})\mathbin{\circ_{\elementQa}} ({\elementQlambda}_{i} \mathbin{\circ_{{\elementQa}_i}} {\elementQmu}_{i})\le ({\elementQa}_i\rightarrow {\elementQlambda}_i )\ast {\elementQlambda}_i=({\elementQa}_i,{\elementQlambda}_i,{\elementQa})\mathbin{\circ_{{\elementQa}_i}}({\elementQa},{\elementQlambda}_i,{\elementQa}_i).
\end{equation*}
We apply the previous relation to (\ref{n6.5CC}). Hence ${\elementQa}\le \tbigvee_{i\in I} (\newmorphism{20}{\elementQa_i}{\elementQlambda_i}{\elementQa}
)\mathbin{\circ_{{\elementQa}_i}} (\newmorphism{20}{\elementQa}{\elementQlambda_i}{\elementQa_i}
)$ 
 follows.
 Analogously we verify  ${\elementQa}\le \tbigvee_{i\in I}  (\newmorphism{20}{\elementQa_i}{\elementQmu_i}{\elementQa}
)\mathbin{\circ_{{\elementQa}_i}} (\newmorphism{20}{\elementQa}{\elementQmu_i}{\elementQa_i}
)$. Thus, by Corollary~\ref{corollary5.4}, the Cauchy completion w.r.t.\ the involutive quantaloid $(D\alg{Q},j)$ preserves the symmetry axiom.
\end{proof}  


\begin{remarks}\label{rem5.5AA} 
 (1) Since condition (\ref{n5.3}) is trivially satisfied in any linearly ordered quantale, Corollary~\ref{corollary5.7}
applies to all integral and commutative quantales on $[0,1]$ --- i.e.,~to all left-continuous $t$-norms  (cf.\ \cite[Ex.~2.4.5]{Stubbe18}).
 \pagebreak
  \\[2pt]
(2) If a unital and commutative quantale $\alg{Q}$ is divisible --- meaning that for every pair $({\elementQa}, {\elementQb})\in\alg{Q}\times \alg{Q}$ with ${\elementQa}\le{\elementQb}$, there exists ${\elementQc}\in \alg{Q}$ such that ${\elementQa}={\elementQb}\ast{\elementQc}$ --- then $\alg{Q}$ is necessarily integral, and condition (\ref{n5.3}) holds (cf.\ \cite[Prop.~2.6]{ho95}). Therefore, Corollary~\ref{corollary5.7} also applies to all unital, commutative, and divisible quantales (cf.\ \cite[Prop.~5.4]{PuDexue}).
\end{remarks}

Finally, since the involution $j$ of an involutive quantaloid $(\Q,j)$ acts as identity on objects, the symmetrization of a \donotbreakdash{$\Q$}category $(X,\alpha)$ always exists (cf.\ \cite[p.~279]{Stubbe11}) and is given by $(X, \alpha\wedge \alpha^{op}))$, where:
\begin{equation*}(\alpha\wedge \alpha^{op})(x,y)=\alpha(x,y)\wedge j(\alpha(y,x)), \qquad x,y \in X_0.\end{equation*}
In this context we recall the following important fact (cf.\ \cite[Prop.~6.9]{HK11}).
\begin{proposition} \label{prop5.7} If the involutive quantaloid $(\Q,j)$ preserves the symmetry axiom, then the symmetrization of a Cauchy complete \donotbreakdash{$\Q$}category is again Cauchy complete.
\end{proposition}



\section{Quantale-valued preorders and quantaloid enriched categories}\label{section7}
Let $\alg{Q}=(\alg{Q},\ast,\mbox{}^{\prime})$ be an involutive quantale, and let $(D\alg{Q},j)$ be the involutive quantaloid associated with $\alg{Q}$ by means of diagonal arrows (cf.\ Remark~\ref{newremark1.3}). Every morphism $\newmorphism{16}{\elementQa}{\lambda}{\elementQb}$ of $D\alg{Q}$ corresponds to a unique element $\lambda \in \alg{Q}$. Hence every  hom-arrow-assignment $\alpha$ of a \donotbreakdash{$D\alg{Q}$}category $(X,\alpha)$ induces a map $\psi\colon X_0\times X_0 \to \alg{Q}$ such that $\alpha(x,y)=\newmorphism{40}{|y|}{\psi(x,y)}{|x|}$ for all $ x,y\in X_0$.
In the following we investigate the properties of $\psi$. Since the composition in $D\alg{Q}$ is given by
\[(\newmorphism{16}{\elementQa}{\lambda}{\elementQb})\mathbin{\circ_{\elementQa}}( \newmorphism{16}{\elementQc}{\mu}{\elementQa})=\newmorphism{80}{\elementQc}{(\lambda\ast(|\elementQa|\searrow \mu))}{\elementQb}=\newmorphism{80}{\elementQc}{((\lambda\swarrow|\elementQa|)\ast \mu)}{\elementQb}  
  \]
  for each morphism $\newmorphism{16}{\elementQa}{\lambda}{\elementQb}$ and $ \newmorphism{16}{\elementQc}{\mu}{\elementQa}$ of $D\alg{Q}$, we conclude from (\ref{n2.1C4}) that $\psi$ satisfies the following property for all $x,y \in X_0$: 
\stepcounter{num}
 \begin{equation}\label{n6.1} \psi(x,x) \ast(|x|\searrow\psi(x,y))=\psi(x,y) =(\psi(x,y)\swarrow |y|) \ast \psi (y,y).
 \end{equation}
 Hence $\psi(x,x)$ is left divisor of $\psi(x,y)$ and $\psi(y,y)$ is right divisor of $\psi(x,y)$, and so the divisibility property
\stepcounter{num}
\begin{equation} \label{n6.2} \psi(x,x)\ast (\psi(x,x) \searrow \psi(x,y))=\psi(x,y)= (\psi(x,y)\swarrow \psi(y,y))\ast \psi(y,y)
\end{equation}
holds for all $x,y\in X_0$.
Combining (\ref{n6.1}) and (\ref{n6.2}), we obtain:
\begin{align*}
\psi(x,{y})\ast (\psi({y},{y})\searrow \psi({y},{z}))&=(\psi({x},{y}) \swarrow |{y}|)\ast \psi({y},{y})\ast (\psi({y},{y})\searrow \psi({y},{z}))\\
&=(\psi({x},{y}) \swarrow |{y}|)\ast \psi({y},{z}),
\end{align*}
and hence, by axiom \ref{C2} the following relation follows: 
\stepcounter{num}
\begin{equation}\label{n6.4}
\psi({x},{y})\ast (\psi({y},{y})\searrow \psi({y},{z}))\le \psi({x},{z}), \qquad {x},{y},{z}\in X_0
  \end{equation}

A pair $(X_0,\psi)$ is called a \emph{\donotbreakdash{$\alg{Q}$}valued preordered set} if $X_0$ is a set and $\psi\colon X_0\times X_0\to \alg{Q}$ is a map satisfying the conditions (\ref{n6.2}) and (\ref{n6.4}), referred to as the \emph{\donotbreakdash{$\alg{Q}$}preorder} of $(X_0,\psi)$.

We have shown that every \donotbreakdash{$D\alg{Q}$}category can be understood as a \donotbreakdash{$\alg{Q}$}valued preordered set. This correspondence is a motivation to interpret  condition (\ref{n6.2}) in terms of preorder semantics.
  For this purpose we assume that every element of the underlying quantale $\alg{Q}$ is \emph{self-divisible}.  Then it is well known that $\alg{Q}$ induces two natural  preorders on its elements:
\begin{enumerate}[label=\textup{--},leftmargin=12pt,topsep=3pt,itemsep=0pt]
\item The left-preorder: For ${\elementQa},{\elementQb}\in \alg{Q}$, we say that ${\elementQb}$ is smaller than ${\elementQa}$ if ${\elementQa}$ is a left divisor of ${\elementQb}$ --- i.e.,~${\elementQb}={\elementQa}\ast ({\elementQa}\searrow {\elementQb})$.
\item The right-preorder: Similarly, ${\elementQb}$ is smaller than ${\elementQa}$ if ${\elementQa}$ is right divisor of ${\elementQb}$ --- i.e.,~${\elementQb}=({\elementQb}\swarrow {\elementQa})\ast {\elementQa}$.
\end{enumerate}
 \pagebreak
  Given a \donotbreakdash{$\alg{Q}$}preorder $\psi$, we can now interpret its values semantically: 
  The value $\psi({x},{x})$ expresses the \emph{extent} to which ${x}$ \emph{exists}, while $\psi({x},{y})$ expresses the \emph{extent} to which ${x}$ is smaller then ${y}$. 
   Under this interpretation,  condition (\ref{n6.2}) express the following principle:
\begin{quote}{\small The extent to which ${x}$ is smaller than ${y}$ is bounded by the extent of existence of ${x}$ (via the left-preorder) and of ${y}$ (via the right-preorder).}
\end{quote}
This is in line with the foundational idea in logic and ontology that properties of elements require first their existence.

Condition (\ref{n6.4}) represents the \emph{transitivity axiom} for \donotbreakdash{$\alg{Q}$}preorders, which differs from traditional approaches in many-valued logics, but is consistent with the \emph{composition law} in the quantaloid $D\alg{Q}$. It ensures that the structure respects the enriched categorical framework while preserving the semantic intuition of orderings.

If $\alg{Q}$ is a commutative, unital, and divisible quantale, then the left- and right-preorder coincides with the order of the underlying lattice of $\alg{Q}$. In this case, condition (\ref{n6.2}) is equivalent to the strictness condition:
\stepcounter{num}
\begin{equation}\label{n6.5}
\psi({x},{y})\le \psi({x},{x})\wedge \psi({y},{y}), \qquad {x},{y}\in X_0.
\end{equation}
This condition reflects the idea of  geometric logic that the extent to which $x$ is related to $y$ is bounded by the extent to which $x$ \emph{and} $y$ exist. It originates from the strictness axioms of \donotbreakdash{$\Omega$}valued sets (cf.\ \cite{Scott1} and \cite{Scott2}).

Let $X_0$ be a set. 
A \donotbreakdash{$\alg{Q}$}preorder $\varepsilon\colon {{X_0\times X_0}\to \alg{Q}}$ is a \emph{\donotbreakdash{$\alg{Q}$}valued equivalence relation} on $X_0$ if $\varepsilon$  satisfies the  additional symmetry axiom\textup:
\stepcounter{num}
\begin{equation}\label{n6.6} \varepsilon({x},{y})=(\varepsilon({y},{x}))^{\prime}, \qquad {x},{y}\in X_0 \tag{Symmetry}
  \end{equation}
If $\varepsilon$ is symmetric, then the pair $(X_0,\varepsilon)$ is called a \emph{\donotbreakdash{$\alg{Q}$}valued set}. 

To identify a \donotbreakdash{$\alg{Q}$}valued set $(X_0,\varepsilon)$ with a symmetric \donotbreakdash{$\alg{Q}$}category, we define a type function on $X_0$ as follows. Since (\ref{n6.2}) and the symmetry axiom imply that $\varepsilon(x,x)$ is self-divisible and  hermitian for each $x\in X_0$, we set: 
\stepcounter{num}\begin{equation}\label{n6.7}
t(x)=\varepsilon(x,x), \qquad x\in X_0.
\end{equation}
This defines the typed set $X=(X_0,t)$, and the \donotbreakdash{$\alg{Q}$}valued set $(X_0,\varepsilon)$ becomes a symmetric \donotbreakdash{$D\alg{Q}$}category $(X,\alpha)$, 
 where $\alpha$ is given by $\alpha(x,y)= \newmorphism{40}{\varepsilon(y,y)}{\varepsilon(x,y)}{\varepsilon(x,x)}$ for all  ${x,y\in X_0}$.
 In this sense we  always identify \donotbreakdash{$\alg{Q}$}valued sets with symmetric \donotbreakdash{$D\alg{Q}$}categories, where the type function is determined by (\ref{n6.7}). 

Finally, note that the value
 $t(x)=\varepsilon(x,x)$ of the type function at $x$, 
interpreted as the \emph{extent of existence} of $x$, provides a nontrivial semantic layer to \donotbreakdash{$\alg{Q}$}valued sets. 
It ensures that the structure not only encodes relationships between elements but also reflects their individual degrees of presence or certainty within the system.


The category of \donotbreakdash{$\alg{Q}$}valued sets, with left adjoint distributors as morphisms, is denoted by $\cat{Set}(\alg{Q})$. With   regard to the results in Section~\ref{section5} and Section~\ref{section6}  the following important fact holds.

\begin{theorem}\label{thm7.2} If the Cauchy completion w.r.t.\ $D\alg{Q}$  preserves the symmetry axiom,  then $\cat{Set}(\alg{Q})$ is an $($epi, extremal mono$)$-category.
\end{theorem}
  \pagebreak
  
\begin{example} \label{example6.3} Let $(D\alg{Q}_0,\omega)$ be defined by 
\[\omega(\elementQa,\elementQb)=\tbigvee\set{ \lambda\in \alg{Q}\mid \newmorphism{16}{\elementQb}{\elementQlambda}{\elementQa} \in \hom (\elementQb,\elementQa)}, \qquad \elementQa,\elementQb \in D\alg{Q}_0.\]
We show that this pair forms a \donotbreakdash{$\alg{Q}$}valued set. First, we notice that $\newmorphism{42}{\elementQb}{\omega(\elementQa,\elementQb)}{\elementQa}$ is the top element of the \donotbreakdash{$\hom$}space $\text{hom}(\elementQb,\elementQa)$ --- i.e.,~$\tau(a,b)=\newmorphism{42}{\elementQb}{\omega(\elementQa,\elementQb)}{\elementQa}$. Hence  $(D\alg{Q}_0,\omega)$ is the \donotbreakdash{$\alg{Q}$}valued preordered  set induced by the \donotbreakdash{$D\alg{Q}$}category $(D\alg{Q},\tau)$ (cf.\ Example~\ref{exam2.1}).
 Moreover $\omega$ is symmetric --- i.e.,~$(D\alg{Q}_0,\omega)$ is a \donotbreakdash{$\alg{Q}$}valued set. Referring to the divisibility condition (\ref{n6.2}) the following special property of $\omega$ follows immediately from its definition:  
\stepcounter{num}
\begin{equation}\label{n6.9} 
\omega(\elementQa,\elementQb)=\omega(\elementQa,\omega(\elementQb,\elementQb))=\omega(\omega(\elementQa,\elementQa),\elementQb)=\omega(\omega(\elementQa,\elementQa),\omega(\elementQb,\elementQb)),\qquad \elementQa,\elementQb\in D\alg{Q}_0.
\end{equation}
Hence the associated hom-arrow-assignment $\xi$ of the symmetric \donotbreakdash{$D\alg{Q}$}category $(D\alg{Q},\xi)$ corresponding to $(D\alg{Q}_0,\omega)$   has the  form:
\[\xi(\elementQa,\elementQb)=\newmorphism{43}{\omega(\elementQb,\elementQb)}{\omega(\elementQa,\elementQb)}{\omega(\elementQa,\elementQa)}=\tau(\omega(\elementQa,\elementQa),\omega(\elementQb,\elementQb)),\qquad \elementQa,\elementQb\in D\alg{Q}_0,\]
where we have used (\ref{n6.9}). Since the relations  (\ref{n6.1}) and (\ref{n6.2}) imply for all $\elementQa,\elementQb,\elementQc\in D\alg{Q}_0$
\begin{align*}\omega(\elementQa,\elementQb)\ast (\omega(\elementQb,\elementQb) \searrow \omega (\elementQb,\elementQc))&=(\omega(\elementQa,\elementQb)\swarrow \elementQb) \ast \omega(\elementQb,\elementQb)\ast (\omega(\elementQb,\elementQb) \searrow \omega (\elementQb,\elementQc))\\
&= \bigl(\omega(\elementQa,\elementQb)\swarrow \elementQb\bigr) \ast \omega(\elementQb,\elementQc),
\end{align*}
the following property holds:
\stepcounter{num}
\begin{equation}\label{n6.9CC} \tau(\elementQa,\elementQb) \circ_{\elementQb} \tau(\elementQb,\elementQc)=\newmorphism{43}{\omega(\elementQb,\elementQb)}{\omega(\elementQa,\elementQb)}{\elementQa}\circ_{\omega(\elementQb,\elementQb)} \newmorphism{43}{\elementQc}{\omega(\elementQb,\elementQc)}{\omega(\elementQb,\elementQb)}, \qquad \elementQa,\elementQb,\elementQc \in D\alg{Q}_0.
\end{equation}
\end{example}


\begin{lemma}\label{lem6.4} The \donotbreakdash{$\alg{Q}$}valued set  $(D\alg{Q}_0,\omega)$ is the terminal object in $\cat{Set}(\alg{Q})$. 
\end{lemma}
\begin{proof} Let $(X_0,\varepsilon)$ be a \donotbreakdash{$\alg{Q}$}valued set, which we identify with the symmetric \donotbreakdash{$D\alg{Q}$}category $(X,\alpha)$. Referring to (\ref{n6.9}) we define a map $\Phi\colon D\alg{Q}_0\times X_0 \to \morph(D\alg{Q})$ by
\[ \Phi(\elementQa,x)= \newmorphism{75}{\varepsilon(x,x)}{\omega(\elementQa, \varepsilon(x,x))}{\omega(a,a)}=\tau(\omega(a,a),\varepsilon(x,x)), \qquad \elementQa \in D\alg{Q}_0,\, x\in X_0.
\]
The relation $\tau(\omega(\elementQb,\elementQb),\omega(\elementQa,\elementQa))\mathbin{\circ_{\omega(\elementQa,\elementQa)}}\Phi(\elementQa,x)\le  \Phi(\elementQb,x)$ is obvious. Further we observe for all $\elementQa,\elementQb \in D\alg{Q}_0$ and all $x,y\in X_0$:
\begin{align*}
\Phi(\elementQa,x)\mathbin{\circ_{\varepsilon(x,x)}}\alpha(x,y) \le  \Phi(\elementQa,x) \mathbin{\circ_{\varepsilon(x,x)}}\tau(\varepsilon(x,x),\varepsilon(y,y)) \le \Phi(\elementQa,y),
\end{align*}
i.e., $\Phi:(X,\alpha) \circlearrow (D\alg{Q},\xi)$ is a distributor. 
\\
To show that $\Phi$ is left adjoint, we define a distributor $\Psi\colon(D\alg{Q},\xi) \circlearrow (X,\alpha)$ by 
\[\Psi(x,\elementQa)= \newmorphism{70}{\omega(a,a)}{\omega(\varepsilon(x,x),\elementQa)}{\varepsilon(x,x)}=\tau(\varepsilon(x,x),\omega(\elementQa,\elementQa)), \qquad x\in X_0,\,\elementQa \in D\alg{Q}_0.
\]
Clearly, $\Phi(\elementQa,x)\mathbin{\circ_{\varepsilon(x,x)}} \Psi(x,\elementQb)\le \tau(\omega(\elementQa,\elementQa),\omega(\elementQb,\elementQb))$ holds. Further, we apply  (\ref{n6.9CC}) and obtain: 
\begin{align*}
\alpha(x,x)\le\tbigvee_{\elementQa \in D\alg{Q}_0} \tau(\varepsilon(x,x),\elementQa)\mathbin{\circ_{\elementQa}} \tau(\elementQa,\varepsilon(x,x)) =\tbigvee_{\elementQa \in D\alg{Q}_0} \bigl(\Psi(x,\elementQa)\mathbin{\circ_{\omega(\elementQa,\elementQa)}} \Phi(\elementQa,x)\bigr).
\end{align*}
Hence  $\Psi$ is right adjoint to $\Phi$.\\
Let $\Phi^{\prime}\colon (X,\alpha) \circlearrow (D\alg{Q}_0,\xi)$ be an arbitrary left adjoint distributor with the corresponding right adjoint distributor $\Psi^{\prime}$. Since $\Phi^{\prime}(a,x)\in \text{hom}(\varepsilon(x,x),\omega(a,a))$, the relation $\Phi^{\prime}(\elementQa,x)\le\Phi(\elementQa,x)$ follows. Analogously, we verify $\Psi^{\prime}(x,\elementQa)\le\Psi(x,\elementQa)$. By Proposition~\ref{prop2.1}, this implies that $\Phi$ and $\Phi^{\prime}$ coincide.
\end{proof}

Given that the type function of \donotbreakdash{$\alg{Q}$}valued sets is defined by equation (\ref{n6.7}), a \donotbreakdash{$\alg{Q}$}valued set $(X_0,\alpha)$ is \emph{separated} if and only if the following implication holds for all $\elementQx,\elementQy\in X_0$:
\[\varepsilon(\elementQx,\elementQx)=\varepsilon(\elementQx,\elementQy)=\varepsilon(\elementQy,\elementQx)=\varepsilon(\elementQy,\elementQy) \implies x=y.
\]
This condition ensures that distinct elements cannot be indistinguishable in terms of their mutual and self-relatedness, preserving the identity of elements in the enriched structure.

\begin{proposition}\label{proposition6.5} Let $\alg{Q}=(\alg{Q},\ast,e,\mbox{}^{\prime})$ be a unital and involutive quantale, and let $(D\alg{Q}_0,\omega)$ be the terminal object in \cat{Set}$(\alg{Q})$. Then the following properties are equivalent\textup:
\begin{enumerate}[label=\textup{(\arabic*)},topsep=3pt,itemsep=0pt
,leftmargin=20pt,labelwidth=10pt,itemindent=0pt,labelsep=5pt,topsep=5pt,itemsep=3pt
]
\item $\alg{Q}$ is integral \textup(i.e.,~$e=\top$\textup).
\item $(D\alg{Q}_0,\omega)$ is Cauchy complete.
\item $(D\alg{Q}_0,\omega)$ is separated.
\end{enumerate}
\end{proposition}

\begin{proof} We begin by noting that both $\top$ and $e$ 
are hermitian elements of $\alg{Q}$, that is, ${\top, e \in D\alg{Q}_0}$.
 Furthermore, $\hom (e,e)$ is the quantale $\alg{Q}$, $\hom (\top,\top)$ is the \donotbreakdash{$\text{hom}$}space of two-sided elements,
$\hom (\top,e)$ is the \donotbreakdash{$\text{hom}$}space of right-sided elements, and $\hom (e,\top)$ is the \donotbreakdash{$\text{hom}$}space of left-sided elements of $\alg{Q}$. From this, we conclude:
\[\omega(\top,\top)=\omega(\top,e)=\omega(e,\top)=\omega(e,e)=\top.\] 
Consequently, if $(D\alg{Q}_0,\omega)$ is separated, then we obtain $\top=e$ --- i.e.,~$\alg{Q}$ is integral. Therefore, we have shown that
(3)$\Longrightarrow$(1). Since Cauchy completeness implies separation, it remains to prove (1)$\Longrightarrow$(2).
\\
 Let $\mu=(\Qpresheavef,\elementQa,\Qpresheaveg)$ be a  presingleton of the symmetric  \donotbreakdash{$D\alg{Q}$}category $(D\alg{Q},\xi)$ corresponding to $(D\alg{Q}_0,\omega)$. 
Since $\alg{Q}$ is integral, $\omega(\elementQc,\elementQc)=\elementQc$ follows for all $\elementQc\in D\alg{Q}_0$ --- i.e.\ the type function coincides with the identity of $D\alg{Q}_0$, and 
consequently $\newmorphism{45}{b}{\omega(\elementQa,\elementQb)}{\elementQa}=\tau(\elementQa,\elementQb)$ holds for all $\elementQa,\elementQb\in D\alg{Q}_0$. Hence, we have $\Qpresheavef(\elementQb) \le \tau(\elementQa, \elementQb)$ and $\Qpresheaveg(\elementQb)\le\tau(\elementQb,\elementQa)$. By Proposition~\ref{prop2.1}, this implies  $\mu=\widetilde{\elementQa}$ --- i.e.\ $\mu$ is represented by the element $\elementQa$ of $D\alg{Q}_0$. Since the type function coincides with the identity, the uniqueness of $\elementQa$ is evident.
\end{proof}

\begin{comments} (1) Although the Cauchy completion w.r.t.\ $D\alg{Q}$ does not necessarily preserves the symmetry axiom, the proof of the implication (1)$\Longrightarrow$(2) in 
Proposition~\ref{proposition6.5} shows that in the case of integral and involutive 
quantales $\alg{Q}$, every presingleton of the terminal object in \cat{Set}$(\alg{Q})$  is in fact a 
singleton. This reflects a strong form of representability in the case of quantale-valued sets.
\\[2pt]
(2) If  $\alg{Q}$ is non-integral and the Cauchy completion w.r.t.\ $D\alg{Q}$ preserves the symmetry axiom, then the Cauchy completion of  $(D\alg{Q}_0,\omega)$ is also the terminal object of $\cat{Set}(\alg{Q})$ (cf.\ Section~\ref{section6}). Hence in this setting the category $\cat{Set}(\alg{Q})$ has always  a Cauchy complete terminal object.
\end{comments}

If $\alg{Q}$ is a frame $\Omega$ with the identity map as involution, then $\cat{Set}(\Omega)$ coincides with the category of \donotbreakdash{$\Omega$}valued sets  (cf.\ \cite[Sect.~2.8 and 2.9]{Borceux3}).  Recently, in the special case of commutative, unital, and divisible quantales $(\alg{Q},\ast,e)$,  X.~Hu and L.~Shen showed that {$\cat{Set}(\alg{Q})$} is a topos if and only if $\alg{Q}$ is a frame --- i.e.,~$\ast=\wedge$ (cf.\ \cite{HuShen}). This result can be seen as a bridge between a special class of quatale-valued sets and topos theory, highlighting the foundational role of frames in categorical semantics.


\end{document}